\documentclass[10pt]{amsart}
\usepackage[utf8]{inputenc}

\usepackage[margin=1in]{geometry}

\usepackage{amssymb, bm, mathtools, eucal, blkarray, enumerate, amsmath}
\usepackage{cancel}
\usepackage[colorlinks=true,citecolor=black,linkcolor=black,urlcolor=blue]{hyperref}

\usepackage[usenames,dvipsnames]{xcolor}
\usepackage{tikz}
\usepackage{multicol}
\usepackage{genyoungtabtikz}
\usepackage{blkarray}
\usepackage{lscape}

\numberwithin{equation}{section}

\definecolor{darkblue}{rgb}{0.0,0,0.7} 
\definecolor{darkred}{rgb}{0.7,0,0} 
\definecolor{darkgreen}{rgb}{0, .6, 0} 

\newcommand{\defncolor}{\color{darkred}}
\newcommand{\defn}[1]{{\defncolor\emph{#1}}} 

\newtheorem{thm}[equation]{Theorem}
\newtheorem{cor}[equation]{Corollary}
\newtheorem{lemma}[equation]{Lemma}
\newtheorem{prop}[equation]{Proposition}
\theoremstyle{definition}

\newtheorem{rem}[equation]{Remark}

\newtheorem{example}[equation]{Example}

\newcommand{\CC}{\mathbb{C}}
\newcommand{\RR}{\mathbb{R}}
\newcommand{\ZZ}{\mathbb{Z}}  
 
\newcommand{\Sym}{\mathbb{S}}
\newcommand{\HH}{\mathbb{H}}
\newcommand{\GG}{\mathbb{G}}

\newcommand{\n}{\mathbf{n}}
\renewcommand{\v}{\mathbf{v}\hskip-1pt}
\renewcommand{\u}{\mathbf{u}}

\newcommand{\T}{{{\mathtt{T}}}}  
\renewcommand{\S}{{\mathtt{S}}} 
\newcommand{\C}{{\mathtt{C}}}
\newcommand{\R}{{\mathtt{R}}}
\newcommand{\U}{{\mathtt{U}}}

\newcommand{\Th}{{{T}}}  

\newcommand{\wt}{\mathsf{wt}}
\newcommand{\inv}{\mathsf{inv}}

\newcommand{\calA}{\mathcal{A}}  
\newcommand{\B}{\mathcal{B}}
 
\newcommand{\calN}{\mathcal{N}}

\newcommand{\calV}{\mathcal{V}}

\newcommand{\YT}{\mathrm{YT}}
\newcommand{\SYT}{\mathrm{SYT}}

\newcommand{\rpar}{{\underline{\lambda}}}

\title[
Transition Matrices]{Transition matrices between Young's natural\\ and seminormal representations 
}

\author[S. Armon]{Sam Armon}
\address{Department of Mathematics \newline\indent
University of Southern California \newline\indent
Los Angeles, California 90089.}
\email{armon@usc.edu}

\author[T. Halverson]{Tom Halverson}
\address{Department of Mathematics\\ Statistics\\ and Computer Science \newline\indent
Macalester College \newline\indent
Saint Paul\\ MN 55105  }
\email{halverson@macalester.edu}

\thanks{The authors gratefully acknowledge support from Simons Foundation grant 283311.}

\begin{document}

\date{\today}
\maketitle


\begin{abstract}
We derive a formula for the entries in the change-of-basis matrix between Young's seminormal and natural representations of the symmetric group. These entries are determined as sums over weighted paths in the weak Bruhat graph on standard tableaux, and we show that they can be computed recursively as the weighted sum of at most two previously-computed entries in the matrix. We generalize our results to work for affine Hecke algebras, Ariki-Koike algebras, Iwahori-Hecke algebras, and complex reflection groups given by the wreath product of a  finite cyclic group with the symmetric group. 
\end{abstract}

\vspace{0.5cm} 




\section{Introduction}
\label{sec:intro}

Young \cite{Yng2} defined three bases of the complex irreducible symmetric group module $\Sym_n^\lambda$ for each integer partition $\lambda \vdash n$. The corresponding representations are now called Young's \emph{natural}, \emph{seminormal}, and \emph{orthogonal} representations  (see \cite{Ru}, \cite{JK}, and \cite{Sa} for modern accounts). The seminormal and orthogonal representations are related by a diagonal transition matrix known to Rutherford \cite{Ru} (and, likely, Young), whereas the seminormal and natural representations are related by a triangular transition  matrix \cite[Sec.~5]{Ram-skew} whose entries are previously unknown rational expressions involving the contents of the boxes in the tableaux of shape $\lambda$ (see \cite{ryom2009denominators}). In this paper, we compute these change-of-basis coefficients as sums of weights on paths in the weak Bruhat graph on standard tableaux of shape $\lambda$ (see Theorem \ref{thm:main}). 

Our methods work equally well for skew-shape representations \cite{JamesPeel,GarsiaWachsSkew,Ram-skew}, and we derive them in that context. For each skew partition shape $\lambda/\mu$ with $n$ boxes, the skew-shape module $\Sym_n^{\lambda/\mu}$ has dimension equal to the number of standard tableaux of shape $\lambda/\mu$, and the module  $\Sym_n^{\lambda/\mu}$ is irreducible if and only if $\mu = \emptyset$ and $\lambda = \lambda/\emptyset$ is a partition of $n$. 

To compute our change-of-basis we use the weak Bruhat graph $\mathcal{B}_n^{\lambda/\mu}$, which is the Hasse diagram of  weak order on standard tableaux  \cite{BW}. The vertices of  $\mathcal{B}_n^{\lambda/\mu}$ are indexed by the standard tableaux  $\SYT(\lambda/\mu)$ of shape $\lambda/\mu$, and for $\S,\T \in \SYT(\lambda/\mu)$ there is an edge $\S \overset{s_{i}}{\longrightarrow} \T$, labeled by the simple (adjacent) transposition $s_i= (i,i+1) \in \Sym_n$, if $s_i(\S) = \T$. The weight on $\S \overset{s_{i}}{\longrightarrow} \T$ is the reciprocal of the the ``axial" distance  between $i$ and $i+1$ in $\S$, or, equivalently, the difference between the content of the box containing $i$ in $\S$ and the content of the box containing $i$ in $\T$. The weights on $\mathcal{B}_n^{\lambda/\mu}$ encode the seminormal matrix entries,  illustrated in Example \ref{eg:seminormalFromBruhat},  as was observed by Lascoux \cite{lascoux2001youngs}.
Russell and Tymoczko \cite{RussellTymoczko} use $\mathcal{B}_{2n}^{(n,n)}$ to study properties of the change-of-basis between the natural basis and the web basis in the  case where $\lambda = (n,n)$.

The seminormal basis $\{\v_\T \mid \T \in \SYT(\lambda/\mu)\}$ and the natural basis $\{\n_\T \mid \T \in \SYT(\lambda/\mu)\}$ satisfy three properties that are the key ingredients to our work: 
(1) The action of a simple transposition $s_i$ on the seminormal basis satisfies $s_i \v_\T = a_i \v_\T + b_i \v_{s_i (\T)}$ for scalars $a_i,b_i$;  
(2)  the action of $\sigma \in \Sym_n$ on the natural basis satisfies  $\sigma \n_\T = \n_{\sigma(\T)}$  when $\sigma(\T)$ is standard; and 
(3) the seminormal and natural basis vectors are equal (up to a scalar) at the column reading tableau $\C$ (the minimal element in Bruhat order), i.e., $\n_\C = \v_\C$.

Ram \cite{Ram-skew} observed that with these three properties the natural basis is completely derivable (up to a scalar) from the seminormal basis, and it is this technique that we use to derive our transition coefficients. Furthermore, Ram proved that generalzed versions of these same three properties  hold for the calibrated skew-shape representations of affine Hecke algebras. Using this fact we extend our result to give the transition matrix between the natural and seminormal representations, in the semisimple cases, of the following algebras: (1) affine Hecke algebras of type A; (2) cyclotomic Hecke algebras (i.e., Ariki-Koike algebras), (3) Iwahori-Hecke algebras of type A and B; and (4) the group algebras of the complex reflection groups $\GG_{r,n} = \ZZ_r \wr \Sym_n$. The proof of our result for the symmetric group (Theorem \ref{thm:main}) generalizes essentially identically to work for all of these other algebras. We could have proved the result for the affine Hecke algebra and then specialized it to work in the other cases, including the symmetric group algebra. However, we chose to organize this paper with the symmetric group first in order to get to this fundamental result without the extra notation needed to describe Hecke algebra representations.

For each of these algebras, the transition matrix is upper triangular under any linear ordering on tableaux that extends Bruhat order. In Proposition \ref{prop:diagonals}, we show  that the diagonal entries of this matrix have a nice form as a product over inversions in the corresponding tableau. In Theorem \ref{thm:orthogonal-transition}, we give the diagonal transition matrix between the seminormal and orthogonal representations, also as a product over the inversions in the tableaux. 

In Corollary \ref{cor:recursive} and Remark \ref{rem:q-recursive}, we show that   the transition matrix entries can be computed recursively in such a way that each entry is the weighted sum of at most two entries in a previous column. In this way, it takes $\mathcal{O}((f^{\lambda/\mu})^2)$ operations to compute the transition matrix, where $f^{\lambda/\mu}$ is the number of standard tableaux of shape $\lambda/\mu$, as discussed in Remark \ref{rem:complexity}.  This is useful in computing the fast Fourier transform (FFT) on these algebras. The seminormal basis is well adapted to the FFT, because the matrices of the generators are sparse and restrict nicely to subgroups (see, for example,  \cite{CB,DR}).  However,  applications may demand using the natural basis, and, in this event, one can  compute the FFT in the seminormal basis and convert the final result to the natural basis by conjugating by the transition matrix.

For the cyclotomic Hecke algebras and complex reflection groups $\GG_{r,n}$, skew shapes are identified with $r$-tuples of partitions having a total of $n$ boxes. In this case, the transition matrix decomposes as the direct sum of identical copies of an $r$-fold tensor product of symmetric group transition matrices (see Corollary \ref{cor:decomp} and Example \ref{eg:tensor}).  In Section \ref{sec:examples}, we give the Bruhat graphs and all of the nontrivial transition matrices for the irreducible symmetric group modules for $\Sym_n$ with $n \le 5$, along with a transition matrix for an irreducible Hecke algebra module.

\section{The Bruhat Graph on Standard Tableaux}
\label{sec:BruhatGraph}

The symmetric group $\Sym_n$ of permutations on $\{1,2, \ldots, n\}$ is generated by the simple (adjacent) transpositions $s_i = (i,i+1)$, which swap $i$ and $i+1$ for $1 \le i \le n-1$, subject to the relations,
\begin{equation}\label{S_n-relations}
\begin{array}{clll}
\text{(a)} & s_i s_j = s_j s_i, & |i-j| >1, &\qquad \\
\text{(b)} & s_i s_{i+1} s_i = s_{i+1} s_i s_{i+1}, & 1 \le i \le n-2, \\
\text{(c)} & s_i^2 = 1, & 1 \le i \le n-1.
\end{array}
\end{equation}
Each permutation $w \in \Sym_n$ can be written as a product (a word) $w = s_{i_1} s_{i_2} \cdots s_{i_t}$ of simple transpositions. The product $w = s_{i_1} s_{i_2} \cdots s_{i_t}$ is a \defn{reduced word} if $t$ is minimal, and in this case the \defn{length} of $w$ is $\ell(w) = t$.

\subsection{Young tableaux}\label{sec:tableaux} A partition of $n \in \ZZ_{\ge 1}$ is a sequence $\lambda = (\lambda_1, \lambda_2, \ldots,  \lambda_\ell)$ such that $\lambda_1 \ge \lambda_2 \ge \cdots \ge \lambda_\ell > 0$ and $|\lambda|:=\lambda_1  + \cdots + \lambda_\ell = n$. We write $\lambda \vdash n$ to signify that $\lambda$ is a partition of $n$ and identify $\lambda$ with its diagram consisting of $\lambda_i$ boxes left-justified in row $i$.  If $\mu = (\mu_1, \mu_2, \ldots,  \mu_k)$ and $\lambda = (\lambda_1, \lambda_2, \ldots,  \lambda_\ell)$ are partitions of $m$ and $m+n$ respectively, then $\mu \subseteq \lambda$ if $0 < \mu_i \le \lambda_i$ for $i = 1, 2, \ldots, k$, and the skew shape $\lambda/\mu$ consists of the boxes that are in $\lambda$ but not in $\mu$. 
If $\mu = \emptyset$, then $\lambda =\lambda/\emptyset$. If $\lambda/\mu$ is a skew shape with $n$ boxes, then a \defn{Young tableau} of shape $\lambda/\mu$ is a filling of the boxes of $\lambda/\mu$ with the integers $1, \ldots, n$ such that each integer appears exactly once. In a \defn{standard Young tableau}  the entries increase from left to right in each row and from top to bottom in each column.  We denote the set of  Young tableaux of shape $\lambda/\mu$ by $\YT(\lambda/\mu)$ and the set of standard Young tableaux of shape $\lambda/\mu$ by $\SYT(\lambda/\mu)$, so for example,
$$
\SYT((3,2))  = \left\{
\begin{array}{ccccc}
\begin{tikzpicture}[scale=.37,line width=1.0pt] 
\draw (0,0) rectangle (1,1); \path (0.5,0.5) node {$1$};
\draw (0,-1) rectangle (1,0); \path (0.5,-0.5) node {$2$};
\draw (1,0) rectangle (2,1); \path (1.5,0.5) node {$3$};
\draw (1,-1) rectangle (2,0); \path (1.5,-0.5) node {$4$};
\draw (2,0) rectangle (3,1); \path (2.5,0.5) node {$5$};
\end{tikzpicture},
&
\begin{tikzpicture}[scale=.37,line width=1.0pt] 
\draw (0,0) rectangle (1,1); \path (0.5,0.5) node {$1$};
\draw (0,-1) rectangle (1,0); \path (0.5,-0.5) node {$3$};
\draw (1,0) rectangle (2,1); \path (1.5,0.5) node {$2$};
\draw (1,-1) rectangle (2,0); \path (1.5,-0.5) node {$4$};
\draw (2,0) rectangle (3,1); \path (2.5,0.5) node {$5$};
\end{tikzpicture},
&
\begin{tikzpicture}[scale=.37,line width=1.0pt] 
\draw (0,0) rectangle (1,1); \path (0.5,0.5) node {$1$};
\draw (0,-1) rectangle (1,0); \path (0.5,-0.5) node {$2$};
\draw (1,0) rectangle (2,1); \path (1.5,0.5) node {$3$};
\draw (1,-1) rectangle (2,0); \path (1.5,-0.5) node {$5$};
\draw (2,0) rectangle (3,1); \path (2.5,0.5) node {$4$};
\end{tikzpicture},
&
\begin{tikzpicture}[scale=.37,line width=1.0pt] 
\draw (0,0) rectangle (1,1); \path (0.5,0.5) node {$1$};
\draw (0,-1) rectangle (1,0); \path (0.5,-0.5) node {$3$};
\draw (1,0) rectangle (2,1); \path (1.5,0.5) node {$2$};
\draw (1,-1) rectangle (2,0); \path (1.5,-0.5) node {$5$};
\draw (2,0) rectangle (3,1); \path (2.5,0.5) node {$4$};
\end{tikzpicture},
&
\begin{tikzpicture}[scale=.37,line width=1.0pt] 
\draw (0,0) rectangle (1,1); \path (0.5,0.5) node {$1$};
\draw (0,-1) rectangle (1,0); \path (0.5,-0.5) node {$4$};
\draw (1,0) rectangle (2,1); \path (1.5,0.5) node {$2$};
\draw (1,-1) rectangle (2,0); \path (1.5,-0.5) node {$5$};
\draw (2,0) rectangle (3,1); \path (2.5,0.5) node {$3$};
\end{tikzpicture}
\end{array}
\right\}
$$
and  
$$
\SYT((3,3,1)/(2,1)) = \left\{\!\!
\begin{array}{cccccccc}
\begin{tikzpicture}[scale=.37,line width=1.0pt] 
\draw (0,-1) rectangle (1,0); \path (0.5,-0.5) node {$1$};
\draw (1,0) rectangle (2,1); \path (1.5,0.5) node {$2$};
\draw (2,1) rectangle (3,2); \path (2.5,1.5) node {$3$};
\draw (2,0) rectangle (3,1); \path (2.5,0.5) node {$4$};
\end{tikzpicture},\!
&
\begin{tikzpicture}[scale=.37,line width=1.0pt] 
\draw (0,-1) rectangle (1,0); \path (0.5,-0.5) node {$1$};
\draw (1,0) rectangle (2,1); \path (1.5,0.5) node {$3$};
\draw (2,1) rectangle (3,2); \path (2.5,1.5) node {$2$};
\draw (2,0) rectangle (3,1); \path (2.5,0.5) node {$4$};
\end{tikzpicture},\!
&
\begin{tikzpicture}[scale=.37,line width=1.0pt] 
\draw (0,-1) rectangle (1,0); \path (0.5,-0.5) node {$2$};
\draw (1,0) rectangle (2,1); \path (1.5,0.5) node {$1$};
\draw (2,1) rectangle (3,2); \path (2.5,1.5) node {$3$};
\draw (2,0) rectangle (3,1); \path (2.5,0.5) node {$4$};
\end{tikzpicture},\!
&
\begin{tikzpicture}[scale=.37,line width=1.0pt] 
\draw (0,-1) rectangle (1,0); \path (0.5,-0.5) node {$2$};
\draw (1,0) rectangle (2,1); \path (1.5,0.5) node {$3$};
\draw (2,1) rectangle (3,2); \path (2.5,1.5) node {$1$};
\draw (2,0) rectangle (3,1); \path (2.5,0.5) node {$4$};
\end{tikzpicture},\!
&
\begin{tikzpicture}[scale=.37,line width=1.0pt] 
\draw (0,-1) rectangle (1,0); \path (0.5,-0.5) node {$3$};
\draw (1,0) rectangle (2,1); \path (1.5,0.5) node {$1$};
\draw (2,1) rectangle (3,2); \path (2.5,1.5) node {$2$};
\draw (2,0) rectangle (3,1); \path (2.5,0.5) node {$4$};
\end{tikzpicture},\!
&
\begin{tikzpicture}[scale=.37,line width=1.0pt] 
\draw (0,-1) rectangle (1,0); \path (0.5,-0.5) node {$3$};
\draw (1,0) rectangle (2,1); \path (1.5,0.5) node {$2$};
\draw (2,1) rectangle (3,2); \path (2.5,1.5) node {$1$};
\draw (2,0) rectangle (3,1); \path (2.5,0.5) node {$4$};
\end{tikzpicture},\!
&
\begin{tikzpicture}[scale=.37,line width=1.0pt] 
\draw (0,-1) rectangle (1,0); \path (0.5,-0.5) node {$4$};
\draw (1,0) rectangle (2,1); \path (1.5,0.5) node {$1$};
\draw (2,1) rectangle (3,2); \path (2.5,1.5) node {$2$};
\draw (2,0) rectangle (3,1); \path (2.5,0.5) node {$3$};
\end{tikzpicture},\!
&
\begin{tikzpicture}[scale=.37,line width=1.0pt] 
\draw (0,-1) rectangle (1,0); \path (0.5,-0.5) node {$4$};
\draw (1,0) rectangle (2,1); \path (1.5,0.5) node {$2$};
\draw (2,1) rectangle (3,2); \path (2.5,1.5) node {$1$};
\draw (2,0) rectangle (3,1); \path (2.5,0.5) node {$3$};
\end{tikzpicture}\!\!
\end{array}
\right\}.
$$
The number of standard Young tableaux of shape $\lambda / \mu$ is denoted by $f^{\lambda/\mu} = \# \SYT(\lambda/\mu)$. The numbers $f^\lambda$ and $f^{\lambda/\mu}$ can be computed using the hook formula and skew-hook formula, respectively (see for example \cite[7.16.3, 7.21.6]{StanleyEC2}).

The \defn{column reading tableau} $\C$ of shape $\lambda/\mu$ is the standard tableau obtained by entering $1,2,\ldots,n$ consecutively down the columns of $\lambda/\mu$, beginning with the southwest most connected component and filling the columns from left to right. The \defn{row reading tableau} $\R$ of shape $\lambda/\mu$ is the standard tableau obtained by entering $1,2,\ldots,n$ left to right across the rows of $\lambda/\mu$, beginning with the northeast most connected component and filling the rows from top to bottom. In the examples above, the column reading tableau is listed first and the row reading tableau is listed last.  If $\T$ is any tableau of shape $\lambda/\mu$ with $n$ boxes and $\sigma \in \Sym_n$, then $\sigma(\T)$ is the tableau of shape $\lambda/\mu$ obtained by permuting the entries of $\T$ according to $\sigma$.

\subsection{Bruhat and weak order on the symmetric group}
\label{subsec:BruhatSn}

Bruhat order (see for example \cite{BB}) is the partial order $\le$ on $\Sym_n$ generated by the covering relation $\lessdot$ given on permutations $\sigma, \tau \in \Sym_n$ by
\begin{equation}\label{Sn:Bruhat}
\sigma \lessdot \tau 
\quad \text{if} \quad  \tau = (i,j)\sigma, \quad \text{for some $1 \le i < j \le n$},
\quad \hbox{and} \quad  
\ell(\sigma) < \ell(\tau).
\end{equation}
Thus $\sigma \le \tau$ in Bruhat order  if there exist $\sigma_0, \ldots, \sigma_k \in \Sym_n$ 
such that
$\sigma = \sigma_0\lessdot \sigma_1 \lessdot \cdots \lessdot \sigma_k = \tau.$
The \defn{subword property} (see \cite[Thm.~2.2.2]{BB}) says that $\sigma \le \tau$ if and only if a subword of some reduced word for $\tau$ is a word for $\sigma$. That is, $\sigma \le \tau$ if and only if $\tau = s_{i_1}s_{i_2} \cdots s_{i_k}$ is a reduced word and there exists a reduced word $\sigma = s_{i_{a_1}}  s_{i_{a_2}} \cdots s_{i_{a_\ell}}$ such that $1 \leq a_1 < \cdots < a_\ell \leq k$.

Weak Bruhat order (see  \cite{BB}) is the partial order $\le_W$ on $\Sym_n$ generated by the covering relation $\lessdot_W$ on permutations $\sigma, \tau \in \Sym_n$  given by
\begin{equation}\label{Sn:weakBruhat}
\sigma \lessdot_W \tau \quad \text{if} 
\quad 
\tau = s_i \sigma= (i,i+1) \sigma, \quad \text{for some $1 \le i \le n-1$},
\quad \hbox{and} \quad  
\ell(\sigma) < \ell(\tau).
\end{equation}
Thus $\sigma \le_W \tau$ if there exist $\sigma_0, \ldots, \sigma_k \in \Sym_n$ such that
$\sigma = \sigma_0 \lessdot_W \sigma_1 \lessdot_W \cdots \lessdot_W \sigma_k = \tau.$
Observe that   $\sigma \le_W \tau$ implies $\sigma \le \tau$, and thus  Bruhat order is an extension of weak order.

\subsection{Bruhat and weak order on standard tableaux}
\label{subsec:BruhatTab}

Let $\lambda/\mu$ be a skew shape with $|\lambda/\mu| = n$. Following \cite{BW} we use Bruhat and weak order on $\Sym_n$ to induce corresponding  partial orders on $\SYT(\lambda/\mu)$. 
For $\T \in \YT(\lambda/\mu)$ the \defn{word of $\T$} is the unique permutation $w_\T \in \Sym_n$ such that $w_\T(\C) = \T$.  For example, if $\lambda/\mu = (4,4,2,1)/(2,2)$, then below are the column reading tableau $\C$ of shape $\lambda/\mu$, another standard tableau $\T$ of shape $\lambda/\mu$, and the word $w_\T \in \Sym_7$ of $\T$:
$$
\C=\begin{array}{c}
\begin{tikzpicture}[scale=.37,line width=1.0pt] 
\draw (2,1) rectangle (3,2); \path (2.5,1.5) node {$4$};
\draw (2,0) rectangle (3,1); \path (2.5,0.5) node {$5$};
\draw (3,1) rectangle (4,2); \path (3.5,1.5) node {$6$};
\draw (3,0) rectangle (4,1); \path (3.5,0.5) node {$7$};
\draw (0,-1) rectangle (1,0); \path (0.5,-0.5) node {$1$};
\draw (1,-1) rectangle (2,0); \path (1.5,-0.5) node {$3$};
\draw (0,-2) rectangle (1,-1); \path (0.5,-1 .5) node {$2$};
\end{tikzpicture}
\end{array},
\qquad
\T=\begin{array}{c}
\begin{tikzpicture}[scale=.37,line width=1.0pt] 
\draw (2,1) rectangle (3,2); \path (2.5,1.5) node {$1$};
\draw (2,0) rectangle (3,1); \path (2.5,0.5) node {$5$};
\draw (3,1) rectangle (4,2); \path (3.5,1.5) node {$3$};
\draw (3,0) rectangle (4,1); \path (3.5,0.5) node {$6$};
\draw (0,-1) rectangle (1,0); \path (0.5,-0.5) node {$2$};
\draw (1,-1) rectangle (2,0); \path (1.5,-0.5) node {$4$};
\draw (0,-2) rectangle (1,-1); \path (0.5,-1 .5) node {$7$};
\end{tikzpicture}
\end{array},
\qquad 
w_\T = \begin{pmatrix}
1 & 2 & 3 & 4 & 5 & 6 & 7 \\
2 & 7 & 4 & 1 & 5 & 3 & 6
\end{pmatrix}.
$$

For $\S, \T \in  \SYT(\lambda/\mu)$, define $\S \le \T$ if and only if $w_\S \le w_\T$. It follows from \eqref{Sn:Bruhat} (see  \cite[Sec.~7]{BW}) that  Bruhat order $\le$ on $\SYT(\lambda/\mu)$ is generated by the following covering relation for $\S, \T \in \SYT(\lambda/\mu)$:
\begin{equation}\label{SYT:Bruhat}
\S \lessdot \T \quad \text{if} \quad \T = (i,j)(\S),  \quad \text{for some $1 \le i < j \le n$},\quad \text{and $i$ is in a lower row of $\S$ than $j$}.
\end{equation}
Thus $\S \le \T$ if there exist $\S_0, \ldots, \S_k  \in \SYT(\lambda/\mu)$  such that
$\S = \S_0 \lessdot \S_1  \lessdot \cdots  \lessdot \S_k = \T.$

Similarly, define $\S \le_W \T$ if and only if $w_\S \le_W w_\T$.  Then weak Bruhat order $\le_W$ on $\SYT(\lambda/\mu)$ is generated by the covering relation
\begin{equation}\label{SYT:weakBruhat}
\S  \lessdot_W \T \quad \text{if} \quad \T = s_i(\S), \quad \text{for some $1 \le i \le n-1$},\quad \text{and $i$ is in a lower row of $\S$ than $i+1$}.
\end{equation}
Thus   $\S \le_W \T$ if there exist $\S_0, \ldots, \S_k  \in \SYT(\lambda/\mu)$ such that $\S = \S_0 \lessdot_W \S_1  \lessdot_W \cdots  \lessdot_W \S_k = \T.$
Again, $\S \le_W \T$ implies $\S \le \T$, so  Bruhat order is an extension of weak  order on $\SYT(\lambda/\mu)$.  We will use  the following theorem which relates weak order on $\SYT(\lambda/\mu)$ to an interval in weak order on $\Sym_n$.

\begin{thm} \label{thm:BW}
{\rm (\cite[Thm.~7.2]{BW})} The map $\T \mapsto w_\T$ is a bijection from weak order on $\SYT(\lambda/\mu)$ to the interval $[w_\C,w_\R]$ in weak order on $\Sym_n$.
\end{thm}

\subsection{The weak Bruhat graph}
\label{subsec:BruhatTab}

For a skew shape $\lambda/\mu$ with $n$ boxes, the \defn{weak Bruhat graph}  $\B_n^{\lambda/\mu}$ is the Hasse diagram of weak order on $\SYT(\lambda/\mu)$. That is, $\B_n^{\lambda/\mu}$ is the simple graph with vertices  $\SYT(\lambda/\mu)$ such that for $\S,\T \in \SYT(\lambda/\mu)$ there is a labeled edge $\S \overset{s_i} {\longrightarrow} \T$
 if $\T = s_i (\S)$.  Since $\T = s_i (\S)$ if and only if $\S = s_i (\T)$, we draw the graph $\B_n^{\lambda/\mu}$  with undirected edges. Since weak order on $\SYT(\lambda/\mu)$ is isomorphic to the interval $[w_\C, w_\R]$ in weak order on $\Sym_n$, the weak Bruhat graph has a unique minimal element $\C$ and a unique maximal element $\R$. Figure \ref{fig:Bruhat} displays the weak Bruhat graph $\B_5^\lambda$ for the partition  shape $\lambda = (3,2)$ and the corresponding interval in weak order on $\Sym_5$. Figure \ref{fig:skewBruhat} displays the weak Bruhat graph $\B_4^{\lambda / \mu}$ for the skew shape $\lambda / \mu = (3,3,1) / (2,1)$ and the corresponding interval in weak  order on $\Sym_4$.  Figure \ref{fig:Bruhat321} displays the weak Bruhat graph $\B_6^\lambda$ for the partition  shape $\lambda = (3,2,1).$

 \begin{figure}[htbp] 
   \centering
   \includegraphics[height=2.5in,page=1]{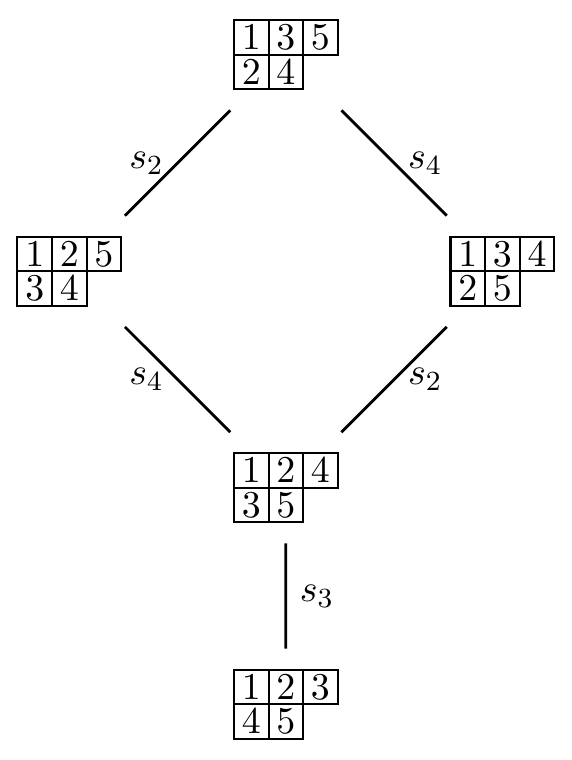}  \hskip.5in  \includegraphics[height=2.5in,page=2]{BruhatGraphs.pdf} 
   \caption{The weak Bruhat graph $\B_5^{(3,2)}$ on standard Young tableaux of shape $(3,2)$ and the corresponding interval $[1,s_3 s_2 s_4]$ in weak  order on $\Sym_5$. \label{fig:Bruhat}}

\end{figure}

 \begin{figure}[htbp] 
   \centering
   \includegraphics[height=3.0in,page=4]{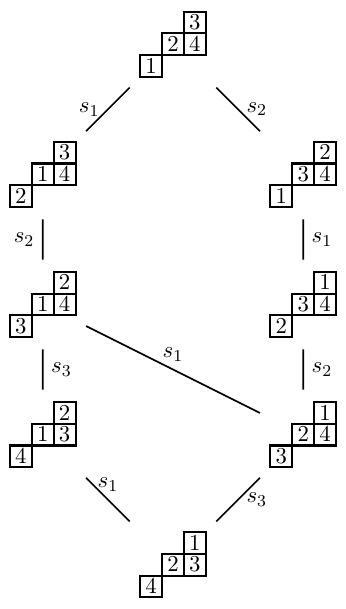}  \hskip.5in  \includegraphics[height=3.0in,page=5]{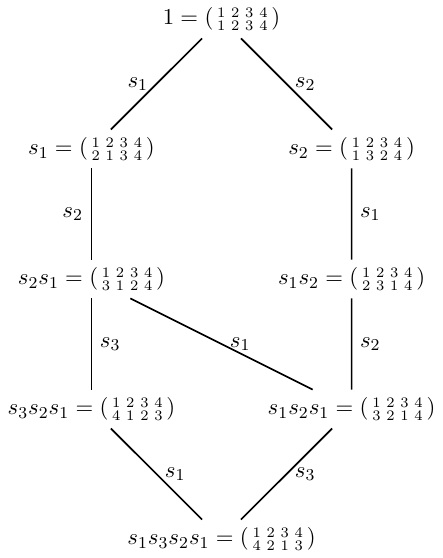} 
   \caption{The weak Bruhat graph $\B_4^{(3,3,1)/(2,1)}$ on standard Young tableaux of skew shape $(3,3,1)/(2,1)$ and the corresponding interval $[1,s_1s_3s_2s_1]$ in weak  order on $\Sym_4$. \label{fig:skewBruhat}}
  
\end{figure}

 \begin{figure}[htbp] 
   \centering
   \includegraphics[height=4.5in,page=6]{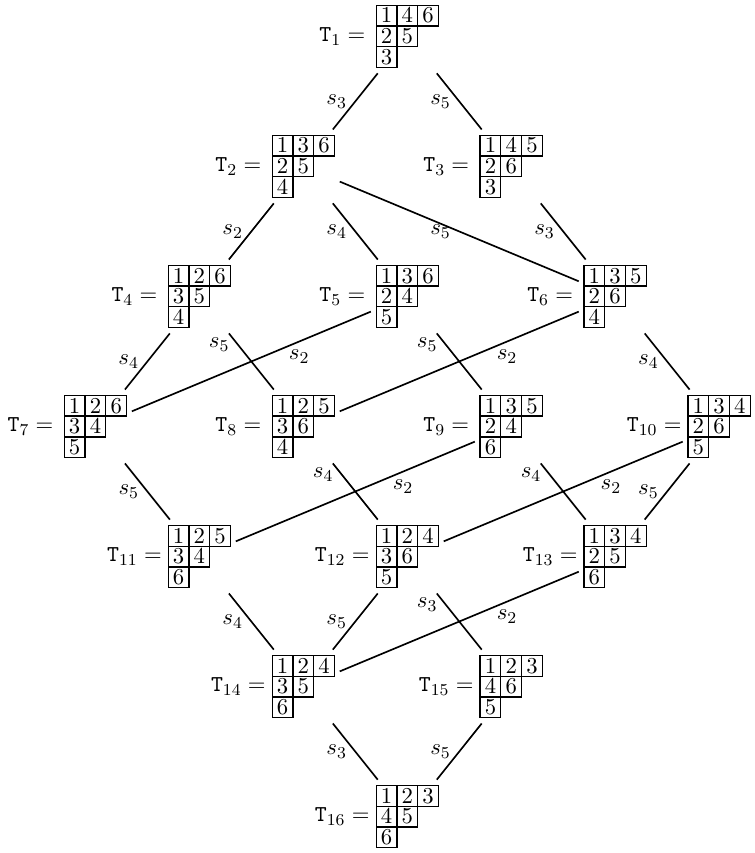}
   \caption{The weak Bruhat Graph $\B_5^{(3,2, 1)}$ on standard Young tableaux of shape $(3,2,1)$. \label{fig:Bruhat321}}

\end{figure}

The \defn{depth} of the tableau $\T \in \B_n^{\lambda/\mu}$ is the length of a shortest path from the root $\C$ to $\T$. By Theorem \ref{thm:BW}, the depth of $\T$ is equal to the length $\ell(w_\T)$ of the word of $\T$.  An \defn{inversion} in $\T$ is a pair $(i,j)$ such that $i > j$ and $i$ is strictly south and strictly west of $j$ in $\T$. An inversion in a permutation $w \in \Sym_n$ is a pair $(i,j)$ such that $i > j$ and $i$ appears left of $j$ in the one-line notation for $w$. By definition of $w_\T$,   $(i,j)$ is an inversion in $\T$ if and only if $(i,j)$ is an inversion in $w_\T$.  Thus, the depth of $\T$ equals the number of inversions in $\T$. For example, 
$$
\text{ if }
\T=\begin{array}{l}
\begin{tikzpicture}[scale=.37,line width=1.0pt] 
\draw (0,0) rectangle (1,1); \path (0.5,0.5) node {$1$};
\draw (1,0) rectangle (2,1); \path (1.5,0.5) node {$2$};
\draw (2,0) rectangle (3,1); \path (2.5,0.5) node {$4$};
\draw (0,-1) rectangle (1,0); \path (0.5,-0.5) node {$3$};
\draw (1,-1) rectangle (2,0); \path (1.5,-0.5) node {$6$};
\draw (0,-2) rectangle (1,-1); \path (0.5,-1.5) node {$5$};
\end{tikzpicture}\end{array} \\
\qquad\hbox{then} \qquad w_\T =
\begin{pmatrix} 1 & 2 & 3 & 4 & 5 & 6 \\  1& 3 & 5 & 2 & 6 & 4 \end{pmatrix} =  135264.
$$
We see that  $\T$ and $w_\T$  have 4 inversions: $(3,2), (5,2),$ $(5,4),$ and $(6,4)$. Thus $\T$ has depth 4 (as can be seen in Figure \ref{fig:Bruhat321}) and $w_\T$ has length 4; indeed, $w_\T = s_4 s_5 s_2 s_3$ is a reduced expression for $w_\T$. 

\begin{lemma} \label{lemma:inversionswapping}
If $\T \in \SYT(\lambda/\mu)$ and $s_i(\T) \lessdot \T$, then $\inv(\T) = s_i(\inv(s_i(\T))) \cup \{ (i+1, i)\}$, where $\inv(\T)$ is the set of inversions in $\T$.  
\end{lemma}

\begin{proof} It is helpful to compare with Example \ref{ex:inversionswap} below.  Since only $i$ and $i+1$ are exchanged, inversions in $\T$ and $s_i(\T)$ that do not involve $i$ or $i+1$ are exactly the same.  Inversions that involve $i$ or $i+1$ and another entry $j$ are in bijection by swapping $i$ and $i+1$ (using the fact that $j$ cannot be between $i$ and $i+1$). The only new inversion is  $(i+1,i) \in \inv(\T)$.
\end{proof}

\Yboxdim{12pt}
\def\ten{{\bf 10}}
\def\eleven{{\bf 11}}
\def\twelve{12}
\def\thirteen{13}

\begin{example}\label{ex:inversionswap} In the example below $s_{10}(\T) = \T'$ and $\inv(\T) = s_{10}(\inv(\T')) \cup \{11,10\}$. That is, the inversions of $\T$ are obtained by exchanging 10 and 11 in the inversions of $\T'$ and then adding the inversion $(11,10)$. Note that under this pairing, the corresponding inversions are in the same positions in $\T$ as in $\T'$.
$$
\begin{array}{ccc}
\T = \young(14679,258\ten,3\eleven\thirteen,\twelve) &\overset{s_{10}}{\longrightarrow}& \T' =  \young(14679,258\eleven,3\ten\thirteen,\twelve) \\
\begin{array}{l}
\{(8,7), (10,9), (11,6), (11,7), (11,8), (11,9), \\
(11,10), (12,4),(12,5), (12,6),(12,7), (12,8), \\
(12,9),(12,10),(12,11), (13,7), (13,9), (13,10)\}
\end{array}
&\overset{s_{10}}{\longrightarrow}& \begin{array}{l}
\{(8,7), (11,9), (10,6), (10,7), (10,8), (10,9), \\
\cancel{(10,11)}, (12,4),(12,5), (12,6),(12,7), (12,8), \\
(12,9),(12,11),(12,10), (13,7), (13,9), (13,11)\}
\end{array}
\end{array}
$$
\end{example}


\section{Representations of the Symmetric Group}
\label{sec:Irreps}

For a skew shape  $\lambda/\mu$ with  $n$ boxes, we follow \cite{Ram-skew} and define $\CC \Sym_n$-modules $\Sym_n^{\lambda/\mu}$ such that, when  $\mu = \emptyset$ and $\lambda = \lambda/\emptyset$ is a partition of $n$, the module $\Sym_n^\lambda$ is irreducible.
In this section, we describe the natural, seminormal, and orthogonal representations of $\Sym_n^{\lambda/\mu}$, originally due to Young \cite{Yng2} for partition shapes. For skew shapes, the natural representations of $\Sym_n^{\lambda/\mu}$ are described in  \cite{JamesPeel,GarsiaWachsSkew} and the seminormal representations in \cite{Ram-skew}.

\subsection{Content and axial distance}
Let $\lambda/\mu$ be a skew shape with  $|\lambda/\mu|=n$. If $b$ is a box in $\lambda/\mu$, we define the \defn{content} of the box $b$ as
\begin{equation}\label{def:Sncontent}
ct(b) = y - x,  \qquad \hbox{if $b$ is in position $(x,y)$ of the partition $\lambda$ (i.e., row $x$ and column $y$)}.
\end{equation}
Note that for $b \in \lambda/\mu$, we use its position in the outer partition $\lambda$.  For example, the following skew diagram has shape $\lambda/\mu = (9,7,7,4,2,2,1)/(4,3,2,2,2)$, and we have filled each box with its content:
$$
\begin{tikzpicture}[scale=.37,line width=1.0pt] 
\draw [dotted,gray]  (0,-3) -- (0,2)--(4,2);
\draw [dotted,gray]  (1,-3) -- (1,2);\draw [dotted,gray]  (2,-3) -- (2,-1);
\draw [dotted,gray]  (2,-2) -- (2,-1);\draw [dotted,gray]  (2,0) -- (2,2);\draw [dotted,gray]  (3,1) -- (3,2);
\draw [dotted,gray] (0,1) -- (3,1);\draw [dotted,gray] (0,0) -- (2,0);\draw [dotted,gray] (0,-1) -- (2,-1);\draw [dotted,gray] (0,-2) -- (2,-2);
\draw (4,1) rectangle (5,2); \path (4.5,1.5) node {$4$};
\draw (5,1) rectangle (6,2); \path (5.5,1.5) node {$5$};
\draw (6,1) rectangle (7,2); \path (6.5,1.5) node {$6$};
\draw (7,1) rectangle (8,2); \path (7.5,1.5) node {$7$};
\draw (8,1) rectangle (9,2); \path (8.5,1.5) node {$8$};
\draw (2,-1) rectangle (3,0); \path (2.5,-0.5) node {$0$};
\draw (3,-1) rectangle (4,0); \path (3.5,-0.5) node {$1$};
\draw (4,-1) rectangle (5,0); \path (4.5,-0.5) node {$2$};
\draw (5,-1) rectangle (6,0); \path (5.5,-0.5) node {$3$};
\draw (6,-1) rectangle (7,0); \path (6.5,-0.5) node {$4$};
\draw (2,-2) rectangle (3,-1); \path (2.5,-1.5) node {-$1$};
\draw (3,-2) rectangle (4,-1); \path (3.5,-1.5) node {$0$};
\draw (3,0) rectangle (4,1); \path (3.5,0.5) node {$2$};
\draw (4,0) rectangle (5,1); \path (4.5,0.5) node {$3$};
\draw (5,0) rectangle (6,1); \path (5.5,0.5) node {$4$};
\draw (6,0) rectangle (7,1); \path (6.5,0.5) node {$5$};
\draw (0,-4) rectangle (1,-3);\path (0.5,-3.5) node {-$5$};
\draw (1,-4) rectangle (2,-3);\path (1.5,-3.5) node {-$4$};
\draw (0,-5) rectangle (1,-4);\path (0.5,-4.5) node {-$6$};
\end{tikzpicture}.
$$
For $\T \in \SYT(\lambda/\mu)$, let $\T(i)$ be the box of $\T$ containing $i$ and define
\begin{equation}\label{def:symmetric-diagonal}
a_{i,j}(\T) = \frac{1}{ ct(\T(j)) - ct(\T(i))} \qquad \hbox{and}\qquad a_{i}(\T) = a_{i,i+1}(\T) = \frac{1}{ ct(\T(i+1)) - ct(\T(i))}.
\end{equation}
Observe that $ct(\T(j)) - ct(\T(i))$ equals the number of steps in a walk from from $i$ to $j$ in the tableau $\T$ counted positively in the northeast direction and negatively in the southwest direction, known as the \defn{axial distance} from $i$ to $j$ in $\T$. For example, the axial distance from 10 to 11 in the tableaux $\T$ of Example \ref{ex:inversionswap} is $-3$ and $a_{10}(\T) =  \frac{1}{-1-2} = -\frac{1}{3}$.

\subsection{Seminormal basis} 
Let  $\mathcal{V}^{\lambda/\mu} = \left\{ \v_\T \mid \T \in \SYT(\lambda/\mu)\right\}$ be a set of vectors indexed by the standard Young tableaux of shape $\lambda/\mu$, and define $\Sym_n^{\lambda/\mu} = \CC\text{-span}\left\{ \v_\T \mid \T \in \SYT(\lambda/\mu)\right\}$, a $\CC$-vector space with basis $\mathcal{V}^{\lambda/\mu}$.  For $1 \le i \le n-1$, define the action of the simple transposition $s_i$ on this basis by
\begin{equation}\label{eq:seminormal}
s_i \v_\T = a_i(\T)  \v_\T+ \left(1 +a_i(\T)  \right) \v_{s_i(\T)}, \qquad \text{where $ \v_{s_i(\T)} = 0$ if $s_i(\T)$ is nonstandard.}
\end{equation}
This action is extended linearly to all of $\Sym_n^{\lambda/\mu}$, and one can check by direct calculation that it respects the relations \eqref{S_n-relations} and thus affords a representation $\rho_\mathcal{V}$, called the seminormal representation of $\CC\Sym_n$, whose matrix entries are  rational numbers. See Example \ref{eg:seminormalFromBruhat} below for an example of computing the seminormal matrix entries using the corresponding weak Bruhat graph.

The seminormal representation gives explicit matrix entries for $\rho_\mathcal{V}(s_i)$ on th e generating set of simple transpositions $s_i, 1 \le i \le n-1$. To produce the matrix $\rho_\mathcal{V}(\sigma)$ of a general permutation $\sigma \in \Sym_n$, one must decompose $\sigma$ into a product of simple transpositions and multiply the corresponding matrices. Upon restriction from $\Sym_n$ to the subgroup $\Sym_{n-1}$ the matrices $\rho_\mathcal{V}(\sigma)$ block-diagonalize into irreducible $\Sym_{n-1}$ seminormal matrices. This is a key ingredient in the algorithm for computing the fast Fourier transform on $\Sym_n$ (see \cite{CB,DR}), and it makes the seminormal representation well-adapted to the the Okounkov-Vershik \cite{OV} and  \cite{RamSeminormal} approaches to symmetric group representations via the tower $\Sym_0 \subseteq \Sym_ 1 \subseteq \cdots \subseteq \Sym_n$.  Furthermore, the matrices $\rho_\mathcal{V}(s_i)$ have at most one off-diagonal entry in each row and column, and sparse matrix techniques can be used to get computational speedups.

\subsection{Orthogonal basis}
Young's orthogonal representation closely resembles the seminormal representation, but the off-diagonal entries are chosen to make the matrices orthogonal. In Section \ref{sec:orthogonal} we show that the seminormal and orthogonal representations differ by a diagonal change-of-basis matrix, and we compute these diagonal entries in Theorem \ref{thm:orthogonal-transition}.  Let $\mathcal{O}^{\lambda/\mu} = \left\{ \u_\T \mid \T \in \SYT(\lambda/\mu)\right\}$ and define the action of a simple transposition on this basis by
\begin{equation}\label{eq:orthogonal}
s_i \u_\T = a_i(\T) \u_\T + \sqrt{1 -a_i(\T)^2\,} \u_{s_i(\T)}, \qquad \text{where $\u_{s_i(\T)} = 0$ if $s_i(\T)$ is nonstandard.}
\end{equation}
Again, the action extends to a representation $\rho_\mathcal{O}$ of $\CC\Sym_n$. In this case, the matrix entries are in a finite field extension of the rationals (to include the square-roots in \eqref{eq:orthogonal}). The matrices of the generators are symmetric and, in general, $\rho_\mathcal{O}(\sigma)^{-1} = \rho_\mathcal{O}(\sigma)^{\textsf{T}}$ for $\sigma \in \Sym_n$ (where $\scriptstyle{\textsf{T}}$ denotes matrix transpose). The orthogonal representation shares the same block-diagonal property under the restriction $\Sym_{n-1} \subseteq \Sym_n$ and the same sparseness properties as the seminormal representation, making it equally well-adapted to computing the fast Fourier transform.

\subsection{Natural basis}

We follow Ram \cite{Ram-skew} and derive Young's natural basis from the seminormal basis. 
Let $\v_\C$ be the seminormal basis vector indexed by the column reading tableau $\C \in \SYT(\lambda/\mu)$, and for each tableau $\T \in \YT(\lambda/\mu)$ (\emph{not necessarily standard}) define the following vector in $ \Sym_n^{\lambda/\mu}$:
\begin{equation}\label{natural-rep}
\n_\T := w_\T \v_\C, \qquad\text{where $w_\T \in \Sym_n$ such that $w_\T(\C) =\T$ (i.e., $w_\T$ is the word of $\T$)}.
\end{equation}
It follows that for any $\sigma \in \Sym_n$ and $\T \in \YT(\lambda/\mu)$, we have $\sigma \n_\T = \sigma w_\T \v_\C = w_{\sigma(\T)} \v_\C = \n_{\sigma(\T)}.$
We then define Young's natural basis as the subset of these vectors $\mathcal{N}^{\lambda/\mu} = \left\{ \n_\T \mid \T \in \SYT(\lambda/\mu)\right\}$ indexed by standard tableaux. Then, we have
\begin{equation}\label{eq:YoungNatural}
 \n_\C = \v_\C, \qquad\hbox{and}\qquad
\sigma \n_\T  =  \n_{\sigma(\T)}, \qquad \text{for all $\sigma \in \Sym_n$ and all $\T \in \SYT(\lambda/\mu)$}.
\end{equation}
When $\sigma(\T)$ is nonstandard, the vector $\n_{\sigma(\T)}$ can be re-expressed as an \emph{integer} linear combination of basis vectors corresponding to standard tableaux using a straightening algorithm such as Garnir relations (see, for example,  \cite{Garnir,JK,Sa,Ram-skew}), turbo straightening \cite{turbo}, or tableaux intersection \cite{GM}.  Indeed, Ram \cite{Ram-skew}  proves that if $\n_\T$ is defined according to  \eqref{eq:YoungNatural} up to scalar multiple, then these vectors satisfy the classical Garnir relations of \cite{Garnir,JK,Sa}. The straightening algorithms can be done over the integers, so that the corresponding natural representation $\rho_\mathcal{N}(\sigma)$ is an integral representation. 

The natural basis is not well-adapted to the fast Fourier transform, but as discussed in \cite{DR}, applications may demand using the natural form. In this case, one can first compute the fast Fourier transform in the seminormal basis and then convert the final result to the natural basis by conjugating by the transition matrix $\mathcal{A}_{\lambda/\mu}$ determined in the next section.

\begin{example}\label{eg:seminormalFromBruhat} The weak Bruhat graph can be used to read off the entries in the seminormal representation. For example, consider $\Sym_6^{(3,2,1)}$ whose Bruhat graph $\B_{5}^{(3,2,1)}$ is given in Figure \ref{fig:Bruhat321}. The axial distance between 5 and 6 in $\T_{11}$ is $-4$, and $s_5(\T_{11}) = \T_{7}$, so using the seminormal basis,
$$
s_5 \v_{\T_{11}} = -\tfrac{1}{4} \v_{\T_{11}} + (1 - \tfrac{1}{4}) \v_{\T_{7}},
$$
and $\T_{11}  \overset{s_5}{\longrightarrow} \T_7$ is an edge in $\B_{5}^{(3,2,1)}$. 
The axial distance between 3 and 4 in $\T_{11}$ is $1$, but $s_3(\T_{11})$ is nonstandard, so
$$
s_3 \v_{\T_{11}} =  \v_{\T_{11}},
$$
and there is no edge labeled by $s_3$ incident to $\T_{11}$ in $\B_{5}^{(3,2,1)}$. 
In the natural basis, we have $s_5 \n_{\T_{11}} = \n_{s_5(\T_{11})}=  \n_{\T_7}$, but $s_3(\T_{11})$ is nonstandard so $s_3 \n_{\T_{11}} = \n_{s_3(\T_{11})}$ must be re-expressed in the basis of standard tableaux using Garnir relations  (e.g., see \cite{JK,Sa,Ram-skew}) which, in this case, after several recursive computations, turns out to be: $\n_{s_3(\T_{11})} = \n_{\T_{11}} - \n_{\T_9}- \n_{\T_8} +  \n_{\T_6} - \n_{\T_3}$. 
\end{example}

\section{Transition Matrices for the Symmetric Group}
\label{sec:Transition}

For a skew shape $\lambda/\mu$ with $n$ boxes,  let $\mathcal{N}^{\lambda/\mu} = \left\{ \n_\T \mid \T \in \SYT(\lambda/\mu)\right\}$ and $\mathcal{V}^{\lambda/\mu} = \left\{ \v_\T \mid \T \in \SYT(\lambda/\mu)\right\}$ be the natural and seminormal bases of $\Sym_n^{\lambda/\mu}$ defined in Section \ref{sec:Irreps}. The goal of this section is to compute $A_{\S,\T} \in \CC$ for $\S,\T \in \SYT(\lambda/\mu)$, uniquely defined by the following change-of-basis equation:
\begin{equation}\label{TransitionCoefficients}
\n_\T = \sum_{\S \in \SYT(\lambda/\mu)} A_{\S,\T} \, \v_\S.
\end{equation}

As in Section \ref{subsec:BruhatTab}, for  $\S, \T \in \SYT(\lambda/\mu)$ write $\S \overset{s_i}{\longrightarrow} \T$ if $\T= s_i(\S)$ for a simple transposition $s_i$.  A \defn{path of length $k$} in the weak Bruhat graph $\B_n^{\lambda/\mu}$ is a  sequence of the form
\begin{equation}\label{def:path}
\pi = (\T_0 \overset{s_{i_1}}{\longrightarrow} \T_1  \overset{s_{i_2}}{\longrightarrow}  \T_2 \overset{s_{i_3}}{\longrightarrow} \cdots \overset{s_{i_k}}{\longrightarrow} \T_k)
\end{equation}
such that $\T_j \in \SYT(\lambda/\mu)$ for  $0\le j \le k$.  The corresponding word $w=s_{i_k} s_{i_{k-1}} \cdots s_{i_1} \in \Sym_n$ satisfies $w(\T_0) = \T_k$. The path $\pi$ is \defn{minimal}  if $s_{i_k} s_{i_{k-1}} \cdots s_{i_1}$ is a reduced word for the permutation $w \in \Sym_n$ that sends $\T_0$ to $\T_k$.  Since $\B_n^{\lambda/\mu}$ corresponds to an interval in weak order on $\Sym_n$, a path $\pi$ is minimal if and only if it is a shortest path from $\T_0$ to $\T_k$  in $\B_n^{\lambda/\mu}$.  

Additionally, we write $\T \overset{e}{\longrightarrow} \T$ for the identity element $e \in \Sym_n$ and any $\T\in \SYT(\lambda/\mu)$.  A \defn{subpath} of the path $\pi$ in \eqref{def:path} is a path of the form 
\begin{equation}\label{def:subpath}
\omega = ( \T_0 = \S_0  \overset{z_{1}}{\longrightarrow} \S_1  \overset{z_{2}}{\longrightarrow}  \S_2 \overset{z_{3}}{\longrightarrow} \cdots \overset{z_{k}}{\longrightarrow} \S_k)
\end{equation}
where $\S_j \in \SYT(\lambda/\mu)$ for  $0\le j \le k$, and  $z_{j} \in \{ s_{i_j}, e \}$.  Thus, a subpath is a path in $\B_n^{\lambda/\mu}$ where one waits at node $\S_{j-1}$ when $z_j = e$ and moves from $\S_{j-1}$ to $\S_j$ when $z_j = s_{i_j}$ . Let $\omega \subseteq \pi$  denote that $\omega$ is a subpath of $\pi$, and  say that the subpath $\omega$ \defn{terminates} at $\S_k$.  Since the $z_j$'s can be identity elements, the subpath $\omega$ need not terminate at the same tableau as the path $\pi$. Examples \ref{ex:path1} and  \ref{ex:path2}  give examples of paths and subpaths. If $\pi$ is a path with word $w=s_{i_k} s_{i_{k-1}} \cdots s_{i_1}$ and $\omega \subseteq \pi$ is a subpath with word $w'=z_k z_{k-1} \cdots z_1$, with $z_j  \in \{s_{i_j}, e\}$ for each $j$, then $w'$ is a subword of $w$. However, not every subword of $w$ corresponds to a subpath, since it is possible that not every tableau in the corresponding path is standard, as seen in the case of $\nu$ in Example \ref{ex:path1}.

\begin{example}\label{ex:path1} Below is a minimal path $\pi$ in $\B_6^{(3,2,1)}$ (see Figure \ref{fig:Bruhat321}) and two subpaths $\omega_1,\omega_2 \subseteq \pi$. 
$$
\begin{array}{r rc cc cc cc cc cc |}
\pi = & \bigg(
\begin{array}{c}\begin{tikzpicture}[scale=.30,line width=1.0pt] 
\draw (0,0) rectangle (1,1); \path (0.5,0.5) node {$\scriptstyle{1}$};
\draw (0,-1) rectangle (1,0); \path (0.5,-0.5) node {$\scriptstyle{2}$};
\draw (0,-2) rectangle (1,-1); \path (0.5,-1.5) node {$\scriptstyle{3}$};
\draw (1,0) rectangle (2,1); \path (1.5,0.5) node {$\scriptstyle{4}$};
\draw (1,-1) rectangle (2,0); \path (1.5,-0.5) node {$\scriptstyle{5}$};
\draw (2,0) rectangle (3,1); \path (2.5,0.5) node {$\scriptstyle{6}$};
\end{tikzpicture}\end{array}
&\!\!\overset{s_{3}}{\longrightarrow}\!\!&
 \begin{array}{c}\begin{tikzpicture}[scale=.30,line width=1.0pt] 
\draw (0,0) rectangle (1,1); \path (0.5,0.5) node {$\scriptstyle{1}$};
\draw (0,-1) rectangle (1,0); \path (0.5,-0.5) node {$\scriptstyle{2}$};
\draw (0,-2) rectangle (1,-1); \path (0.5,-1.5) node {$\scriptstyle{4}$};
\draw (1,0) rectangle (2,1); \path (1.5,0.5) node {$\scriptstyle{3}$};
\draw (1,-1) rectangle (2,0); \path (1.5,-0.5) node {$\scriptstyle{5}$};
\draw (2,0) rectangle (3,1); \path (2.5,0.5) node {$\scriptstyle{6}$};
\end{tikzpicture}\end{array} & 
\!\!\overset{s_{2}}{\longrightarrow}\!\! & 
 \begin{array}{c}\begin{tikzpicture}[scale=.30,line width=1.0pt] 
\draw (0,0) rectangle (1,1); \path (0.5,0.5) node {$\scriptstyle{1}$};
\draw (0,-1) rectangle (1,0); \path (0.5,-0.5) node {$\scriptstyle{3}$};
\draw (0,-2) rectangle (1,-1); \path (0.5,-1.5) node {$\scriptstyle{4}$};
\draw (1,0) rectangle (2,1); \path (1.5,0.5) node {$\scriptstyle{2}$};
\draw (1,-1) rectangle (2,0); \path (1.5,-0.5) node {$\scriptstyle{5}$};
\draw (2,0) rectangle (3,1); \path (2.5,0.5) node {$\scriptstyle{6}$};
\end{tikzpicture}\end{array} & 
\!\!\overset{s_{5}}{\longrightarrow}\!\! & 
 \begin{array}{c}\begin{tikzpicture}[scale=.30,line width=1.0pt] 
\draw (0,0) rectangle (1,1); \path (0.5,0.5) node {$\scriptstyle{1}$};
\draw (0,-1) rectangle (1,0); \path (0.5,-0.5) node {$\scriptstyle{3}$};
\draw (0,-2) rectangle (1,-1); \path (0.5,-1.5) node {$\scriptstyle{4}$};
\draw (1,0) rectangle (2,1); \path (1.5,0.5) node {$\scriptstyle{2}$};
\draw (1,-1) rectangle (2,0); \path (1.5,-0.5) node {$\scriptstyle{6}$};
\draw (2,0) rectangle (3,1); \path (2.5,0.5) node {$\scriptstyle{5}$};
\end{tikzpicture}\end{array} & 
\!\!\overset{s_{4}}{\longrightarrow}\!\!&
 \begin{array}{c}\begin{tikzpicture}[scale=.30,line width=1.0pt] 
\draw (0,0) rectangle (1,1); \path (0.5,0.5) node {$\scriptstyle{1}$};
\draw (0,-1) rectangle (1,0); \path (0.5,-0.5) node {$\scriptstyle{3}$};
\draw (0,-2) rectangle (1,-1); \path (0.5,-1.5) node {$\scriptstyle{5}$};
\draw (1,0) rectangle (2,1); \path (1.5,0.5) node {$\scriptstyle{2}$};
\draw (1,-1) rectangle (2,0); \path (1.5,-0.5) node {$\scriptstyle{6}$};
\draw (2,0) rectangle (3,1); \path (2.5,0.5) node {$\scriptstyle{4}$};
\end{tikzpicture}\end{array} & 
\!\!\overset{s_{3}}{\longrightarrow}\!\!&
 \begin{array}{c}\begin{tikzpicture}[scale=.30,line width=1.0pt] 
\draw (0,0) rectangle (1,1); \path (0.5,0.5) node {$\scriptstyle{1}$};
\draw (0,-1) rectangle (1,0); \path (0.5,-0.5) node {$\scriptstyle{4}$};
\draw (0,-2) rectangle (1,-1); \path (0.5,-1.5) node {$\scriptstyle{5}$};
\draw (1,0) rectangle (2,1); \path (1.5,0.5) node {$\scriptstyle{2}$};
\draw (1,-1) rectangle (2,0); \path (1.5,-0.5) node {$\scriptstyle{6}$};
\draw (2,0) rectangle (3,1); \path (2.5,0.5) node {$\scriptstyle{3}$};
\end{tikzpicture}\end{array}\bigg)
\\
\omega_1 = & \bigg(
\begin{array}{c}\begin{tikzpicture}[scale=.30,line width=1.0pt] 
\draw (0,0) rectangle (1,1); \path (0.5,0.5) node {$\scriptstyle{1}$};
\draw (0,-1) rectangle (1,0); \path (0.5,-0.5) node {$\scriptstyle{2}$};
\draw (0,-2) rectangle (1,-1); \path (0.5,-1.5) node {$\scriptstyle{3}$};
\draw (1,0) rectangle (2,1); \path (1.5,0.5) node {$\scriptstyle{4}$};
\draw (1,-1) rectangle (2,0); \path (1.5,-0.5) node {$\scriptstyle{5}$};
\draw (2,0) rectangle (3,1); \path (2.5,0.5) node {$\scriptstyle{6}$};
\end{tikzpicture}\end{array}
&\!\!\overset{s_{3}}{\longrightarrow}\!\!&
 \begin{array}{c}\begin{tikzpicture}[scale=.30,line width=1.0pt] 
\draw (0,0) rectangle (1,1); \path (0.5,0.5) node {$\scriptstyle{1}$};
\draw (0,-1) rectangle (1,0); \path (0.5,-0.5) node {$\scriptstyle{2}$};
\draw (0,-2) rectangle (1,-1); \path (0.5,-1.5) node {$\scriptstyle{4}$};
\draw (1,0) rectangle (2,1); \path (1.5,0.5) node {$\scriptstyle{3}$};
\draw (1,-1) rectangle (2,0); \path (1.5,-0.5) node {$\scriptstyle{5}$};
\draw (2,0) rectangle (3,1); \path (2.5,0.5) node {$\scriptstyle{6}$};
\end{tikzpicture}\end{array} & 
\!\!\overset{e}{\longrightarrow}\!\! & 
 \begin{array}{c}\begin{tikzpicture}[scale=.30,line width=1.0pt] 
\draw (0,0) rectangle (1,1); \path (0.5,0.5) node {$\scriptstyle{1}$};
\draw (0,-1) rectangle (1,0); \path (0.5,-0.5) node {$\scriptstyle{2}$};
\draw (0,-2) rectangle (1,-1); \path (0.5,-1.5) node {$\scriptstyle{4}$};
\draw (1,0) rectangle (2,1); \path (1.5,0.5) node {$\scriptstyle{3}$};
\draw (1,-1) rectangle (2,0); \path (1.5,-0.5) node {$\scriptstyle{5}$};
\draw (2,0) rectangle (3,1); \path (2.5,0.5) node {$\scriptstyle{6}$};
\end{tikzpicture}\end{array} & 
\!\!\overset{s_{5}}{\longrightarrow}\!\! & 
 \begin{array}{c}\begin{tikzpicture}[scale=.30,line width=1.0pt] 
\draw (0,0) rectangle (1,1); \path (0.5,0.5) node {$\scriptstyle{1}$};
\draw (0,-1) rectangle (1,0); \path (0.5,-0.5) node {$\scriptstyle{2}$};
\draw (0,-2) rectangle (1,-1); \path (0.5,-1.5) node {$\scriptstyle{4}$};
\draw (1,0) rectangle (2,1); \path (1.5,0.5) node {$\scriptstyle{3}$};
\draw (1,-1) rectangle (2,0); \path (1.5,-0.5) node {$\scriptstyle{6}$};
\draw (2,0) rectangle (3,1); \path (2.5,0.5) node {$\scriptstyle{5}$};
\end{tikzpicture}\end{array} & 
\!\!\overset{s_{4}}{\longrightarrow}\!\!&
 \begin{array}{c}\begin{tikzpicture}[scale=.30,line width=1.0pt] 
\draw (0,0) rectangle (1,1); \path (0.5,0.5) node {$\scriptstyle{1}$};
\draw (0,-1) rectangle (1,0); \path (0.5,-0.5) node {$\scriptstyle{2}$};
\draw (0,-2) rectangle (1,-1); \path (0.5,-1.5) node {$\scriptstyle{5}$};
\draw (1,0) rectangle (2,1); \path (1.5,0.5) node {$\scriptstyle{3}$};
\draw (1,-1) rectangle (2,0); \path (1.5,-0.5) node {$\scriptstyle{6}$};
\draw (2,0) rectangle (3,1); \path (2.5,0.5) node {$\scriptstyle{4}$};
\end{tikzpicture}\end{array} & 
\!\!\overset{e}{\longrightarrow}\!\! &
 \begin{array}{c}\begin{tikzpicture}[scale=.30,line width=1.0pt] 
\draw (0,0) rectangle (1,1); \path (0.5,0.5) node {$\scriptstyle{1}$};
\draw (0,-1) rectangle (1,0); \path (0.5,-0.5) node {$\scriptstyle{2}$};
\draw (0,-2) rectangle (1,-1); \path (0.5,-1.5) node {$\scriptstyle{5}$};
\draw (1,0) rectangle (2,1); \path (1.5,0.5) node {$\scriptstyle{3}$};
\draw (1,-1) rectangle (2,0); \path (1.5,-0.5) node {$\scriptstyle{6}$};
\draw (2,0) rectangle (3,1); \path (2.5,0.5) node {$\scriptstyle{4}$};
\end{tikzpicture}\end{array}\bigg)
\\
\omega_2 = & \bigg(
\begin{array}{c}\begin{tikzpicture}[scale=.30,line width=1.0pt] 
\draw (0,0) rectangle (1,1); \path (0.5,0.5) node {$\scriptstyle{1}$};
\draw (0,-1) rectangle (1,0); \path (0.5,-0.5) node {$\scriptstyle{2}$};
\draw (0,-2) rectangle (1,-1); \path (0.5,-1.5) node {$\scriptstyle{3}$};
\draw (1,0) rectangle (2,1); \path (1.5,0.5) node {$\scriptstyle{4}$};
\draw (1,-1) rectangle (2,0); \path (1.5,-0.5) node {$\scriptstyle{5}$};
\draw (2,0) rectangle (3,1); \path (2.5,0.5) node {$\scriptstyle{6}$};
\end{tikzpicture}\end{array}
&\!\!\overset{s_{3}}{\longrightarrow}\!\!&
 \begin{array}{c}\begin{tikzpicture}[scale=.30,line width=1.0pt] 
\draw (0,0) rectangle (1,1); \path (0.5,0.5) node {$\scriptstyle{1}$};
\draw (0,-1) rectangle (1,0); \path (0.5,-0.5) node {$\scriptstyle{2}$};
\draw (0,-2) rectangle (1,-1); \path (0.5,-1.5) node {$\scriptstyle{4}$};
\draw (1,0) rectangle (2,1); \path (1.5,0.5) node {$\scriptstyle{3}$};
\draw (1,-1) rectangle (2,0); \path (1.5,-0.5) node {$\scriptstyle{5}$};
\draw (2,0) rectangle (3,1); \path (2.5,0.5) node {$\scriptstyle{6}$};
\end{tikzpicture}\end{array} & 
\!\!\overset{e}{\longrightarrow}\!\! & 
  \begin{array}{c}\begin{tikzpicture}[scale=.30,line width=1.0pt] 
\draw (0,0) rectangle (1,1); \path (0.5,0.5) node {$\scriptstyle{1}$};
\draw (0,-1) rectangle (1,0); \path (0.5,-0.5) node {$\scriptstyle{2}$};
\draw (0,-2) rectangle (1,-1); \path (0.5,-1.5) node {$\scriptstyle{4}$};
\draw (1,0) rectangle (2,1); \path (1.5,0.5) node {$\scriptstyle{3}$};
\draw (1,-1) rectangle (2,0); \path (1.5,-0.5) node {$\scriptstyle{5}$};
\draw (2,0) rectangle (3,1); \path (2.5,0.5) node {$\scriptstyle{6}$};
\end{tikzpicture}\end{array} & 
\!\!\overset{e}{\longrightarrow}\!\! & 
 \begin{array}{c}\begin{tikzpicture}[scale=.30,line width=1.0pt] 
\draw (0,0) rectangle (1,1); \path (0.5,0.5) node {$\scriptstyle{1}$};
\draw (0,-1) rectangle (1,0); \path (0.5,-0.5) node {$\scriptstyle{2}$};
\draw (0,-2) rectangle (1,-1); \path (0.5,-1.5) node {$\scriptstyle{4}$};
\draw (1,0) rectangle (2,1); \path (1.5,0.5) node {$\scriptstyle{3}$};
\draw (1,-1) rectangle (2,0); \path (1.5,-0.5) node {$\scriptstyle{5}$};
\draw (2,0) rectangle (3,1); \path (2.5,0.5) node {$\scriptstyle{6}$};
\end{tikzpicture}\end{array} & 
\!\!\overset{e}{\longrightarrow}\!\!&
 \begin{array}{c}\begin{tikzpicture}[scale=.30,line width=1.0pt] 
\draw (0,0) rectangle (1,1); \path (0.5,0.5) node {$\scriptstyle{1}$};
\draw (0,-1) rectangle (1,0); \path (0.5,-0.5) node {$\scriptstyle{2}$};
\draw (0,-2) rectangle (1,-1); \path (0.5,-1.5) node {$\scriptstyle{4}$};
\draw (1,0) rectangle (2,1); \path (1.5,0.5) node {$\scriptstyle{3}$};
\draw (1,-1) rectangle (2,0); \path (1.5,-0.5) node {$\scriptstyle{5}$};
\draw (2,0) rectangle (3,1); \path (2.5,0.5) node {$\scriptstyle{6}$};
\end{tikzpicture}\end{array} & 
\!\!\overset{s_3}{\longrightarrow}\!\! &
\begin{array}{c}\begin{tikzpicture}[scale=.30,line width=1.0pt] 
\draw (0,0) rectangle (1,1); \path (0.5,0.5) node {$\scriptstyle{1}$};
\draw (0,-1) rectangle (1,0); \path (0.5,-0.5) node {$\scriptstyle{2}$};
\draw (0,-2) rectangle (1,-1); \path (0.5,-1.5) node {$\scriptstyle{3}$};
\draw (1,0) rectangle (2,1); \path (1.5,0.5) node {$\scriptstyle{4}$};
\draw (1,-1) rectangle (2,0); \path (1.5,-0.5) node {$\scriptstyle{5}$};
\draw (2,0) rectangle (3,1); \path (2.5,0.5) node {$\scriptstyle{6}$};
\end{tikzpicture}\end{array}\bigg)\
\end{array}
$$
The sequence $\nu$, below, is \emph{not} a path in $\B_6^{(3,2,1)}$, and thus not a subpath of $\pi$, because at the last step the tableau is not standard.
$$
\begin{array}{r rc cc cc cc cc cc |}
\nu = & \bigg(
\begin{array}{c}\begin{tikzpicture}[scale=.30,line width=1.0pt] 
\draw (0,0) rectangle (1,1); \path (0.5,0.5) node {$\scriptstyle{1}$};
\draw (0,-1) rectangle (1,0); \path (0.5,-0.5) node {$\scriptstyle{2}$};
\draw (0,-2) rectangle (1,-1); \path (0.5,-1.5) node {$\scriptstyle{3}$};
\draw (1,0) rectangle (2,1); \path (1.5,0.5) node {$\scriptstyle{4}$};
\draw (1,-1) rectangle (2,0); \path (1.5,-0.5) node {$\scriptstyle{5}$};
\draw (2,0) rectangle (3,1); \path (2.5,0.5) node {$\scriptstyle{6}$};
\end{tikzpicture}\end{array}
&\!\!\overset{s_{3}}{\longrightarrow}\!\!&
 \begin{array}{c}\begin{tikzpicture}[scale=.30,line width=1.0pt] 
\draw (0,0) rectangle (1,1); \path (0.5,0.5) node {$\scriptstyle{1}$};
\draw (0,-1) rectangle (1,0); \path (0.5,-0.5) node {$\scriptstyle{2}$};
\draw (0,-2) rectangle (1,-1); \path (0.5,-1.5) node {$\scriptstyle{4}$};
\draw (1,0) rectangle (2,1); \path (1.5,0.5) node {$\scriptstyle{3}$};
\draw (1,-1) rectangle (2,0); \path (1.5,-0.5) node {$\scriptstyle{5}$};
\draw (2,0) rectangle (3,1); \path (2.5,0.5) node {$\scriptstyle{6}$};
\end{tikzpicture}\end{array} & 
\!\!\overset{s_{2}}{\longrightarrow}\!\! & 
 \begin{array}{c}\begin{tikzpicture}[scale=.30,line width=1.0pt] 
\draw (0,0) rectangle (1,1); \path (0.5,0.5) node {$\scriptstyle{1}$};
\draw (0,-1) rectangle (1,0); \path (0.5,-0.5) node {$\scriptstyle{3}$};
\draw (0,-2) rectangle (1,-1); \path (0.5,-1.5) node {$\scriptstyle{4}$};
\draw (1,0) rectangle (2,1); \path (1.5,0.5) node {$\scriptstyle{2}$};
\draw (1,-1) rectangle (2,0); \path (1.5,-0.5) node {$\scriptstyle{5}$};
\draw (2,0) rectangle (3,1); \path (2.5,0.5) node {$\scriptstyle{6}$};
\end{tikzpicture}\end{array} & 
\!\!\overset{s_{5}}{\longrightarrow}\!\! & 
 \begin{array}{c}\begin{tikzpicture}[scale=.30,line width=1.0pt] 
\draw (0,0) rectangle (1,1); \path (0.5,0.5) node {$\scriptstyle{1}$};
\draw (0,-1) rectangle (1,0); \path (0.5,-0.5) node {$\scriptstyle{3}$};
\draw (0,-2) rectangle (1,-1); \path (0.5,-1.5) node {$\scriptstyle{4}$};
\draw (1,0) rectangle (2,1); \path (1.5,0.5) node {$\scriptstyle{2}$};
\draw (1,-1) rectangle (2,0); \path (1.5,-0.5) node {$\scriptstyle{6}$};
\draw (2,0) rectangle (3,1); \path (2.5,0.5) node {$\scriptstyle{5}$};
\end{tikzpicture}\end{array} & 
\!\!\overset{e}{\longrightarrow}\!\!&
 \begin{array}{c}\begin{tikzpicture}[scale=.30,line width=1.0pt] 
\draw (0,0) rectangle (1,1); \path (0.5,0.5) node {$\scriptstyle{1}$};
\draw (0,-1) rectangle (1,0); \path (0.5,-0.5) node {$\scriptstyle{3}$};
\draw (0,-2) rectangle (1,-1); \path (0.5,-1.5) node {$\scriptstyle{4}$};
\draw (1,0) rectangle (2,1); \path (1.5,0.5) node {$\scriptstyle{2}$};
\draw (1,-1) rectangle (2,0); \path (1.5,-0.5) node {$\scriptstyle{6}$};
\draw (2,0) rectangle (3,1); \path (2.5,0.5) node {$\scriptstyle{5}$};
\end{tikzpicture}\end{array} & 
\!\!\overset{s_{3}}{\longrightarrow}\!\!&
 \begin{array}{c}\begin{tikzpicture}[scale=.30,line width=1.0pt] 
\draw (0,0) rectangle (1,1); \path (0.5,0.5) node {$\scriptstyle{1}$};
\draw (0,-1) rectangle (1,0); \path (0.5,-0.5) node {$\scriptstyle{4}$};
\draw (0,-2) rectangle (1,-1); \path (0.5,-1.5) node {$\scriptstyle{3}$};
\draw (1,0) rectangle (2,1); \path (1.5,0.5) node {$\scriptstyle{2}$};
\draw (1,-1) rectangle (2,0); \path (1.5,-0.5) node {$\scriptstyle{6}$};
\draw (2,0) rectangle (3,1); \path (2.5,0.5) node {$\scriptstyle{5}$};
\end{tikzpicture}\end{array}\bigg)
\end{array}
$$
\end{example}

Given a path $\pi = (\T_0 \overset{s_{i_1}}{\longrightarrow} \T_1  \overset{s_{i_2}}{\longrightarrow}  \T_2 \overset{s_{i_3}}{\longrightarrow} \cdots \overset{s_{i_k}}{\longrightarrow} \T_k)$ and a subpath $\omega = ( \T_0 = \S_0  \overset{z_{1}}{\longrightarrow} \S_1  \overset{z_{2}}{\longrightarrow}  \S_2 \overset{z_{3}}{\longrightarrow} \cdots \overset{z_{k}}{\longrightarrow} \S_k)$ of $\pi$ in $\mathcal{B}_n^{\lambda/\mu}$, we define the \defn{$\pi$-weight} of $\omega$ to be
\begin{equation}\label{Sn:weight}
\wt_\pi(\omega) = \prod_{j=1}^k b_j, 
\qquad\text{where}\quad
b_j = \begin{cases}
a_{i_j} (\S_{j-1}), & \text{if $\S_{j-1}\overset{e}{\longrightarrow} \S_{j}$,  i.e.,  if $\S_j = \S_{j-1}$}, \\
a_{i_j} (\S_{j-1})+1, & \text{if $\S_{j-1}\overset{s_{i_j}}{\longrightarrow} \S_{j}$, i.e., if $\S_j = s_{i_j}(\S_{i-1})$},
\end{cases}
\end{equation}
with $\wt_\emptyset(\emptyset) = 1$.
Alternatively, we can describe $\wt_\pi(\omega)$ with a Kronecker delta:
\begin{equation}
\wt_\pi(\omega) = \prod_{j=1}^k a_{i_j}(\S_{j-1}) + 1 - \delta_{\S_{j-1},\S_j}.
\end{equation}

\noindent The following is our main theorem.

\begin{thm}\label{thm:main} For a skew shape $\lambda/\mu$ with $n$ boxes, the change-of-basis expressing the natural basis $\mathcal{N}^{\lambda/\mu}$ in terms of the seminormal basis $\mathcal{V}^{\lambda/\mu}$ of  $\Sym_n^{\lambda/\mu}$ satisfies
$$
\n_\T = \sum_{\S \le \T} A_{\S,\T} \, \v_\S \qquad\hbox{with coefficients}\qquad A_{\S,\T} = \sum_{\omega \subseteq \pi} \wt_{\pi}(\omega),
$$
where the first sum is over standard tableaux $\S \le \T$ in (strong) Bruhat order;
$\pi$  is any fixed path from  $\C$ to $\T$ in $\B_n^{\lambda/\mu}$;
and the second sum is over all subpaths $\omega \subseteq\pi$ which terminate at $\S$.
\end{thm}

\begin{proof}
Let $\T \in \SYT(\lambda/\mu)$ and let $\pi = (\C=\T_0 \overset{s_{i_1}}{\longrightarrow} \T_1  \overset{s_{i_2}}{\longrightarrow}  \T_2 \overset{s_{i_3}}{\longrightarrow} \cdots \overset{s_{i_k}}{\longrightarrow} \T_k= \T)$ be a path to $\T$ (not necessarily minimal). We prove the result by induction on the length $k$ of $\pi$. If $k = 0$, then $\pi= \emptyset$ is the empty path. In this case, $\T = \C$, $\n_\C = \v_\C$, there are no subpaths of $\pi$ other than $\pi$ itself, and $A_{\C,\C} =  \wt_\emptyset(\emptyset) =1$. 

Now assume $\pi$ has length $k \ge 1$.  Let $\T' = \T_{k-1}$ so that $s_{i_k}(\T') = \T$ and  $\pi' = (\C=\T_0 \overset{s_{i_1}}{\longrightarrow} \T_1  \overset{s_{i_2}}{\longrightarrow}  \cdots \overset{s_{i_{k-1}}}{\longrightarrow} \T_{k-1}= \T')$ is a  path to $\T'$. We have
\begin{align*}
\n_{\T} = \n_{s_{i_k}(\T')} = s_{i_k}  \n_{\T'} & = s_{i_k}  \sum_{\U\in \SYT(\lambda/\mu)} A_{\U,\T'} \v_\U 
= \sum_{\U\in \SYT(\lambda/\mu)} A_{\U,\T'} \left( s_{i_k}  \v_\U \right) \\
& =  \sum_{\U\in \SYT(\lambda/\mu)} A_{\U,\T'} \left( a \v_\U  + (1 + a) \v_{s_{i_k}(\U)} \right) \\
& =  \sum_{\U\in \SYT(\lambda/\mu)} A_{\U,\T'}  a \v_\U  + A_{\U,\T'} (1 + a) \v_{s_{i_k}(\U)},
\end{align*}
where $a = a_{i_k}(\U)$ and where $\v_{s_{i_k}(\U)} = 0$ if $s_{i_k}(\U)$ is nonstandard. The terms which contribute to $A_{\S,\T}$ are those where $\v_\S = \v_\U$ and those where $\v_\S = \v_{s_{i_k}(\U)}$ when $s_{i_k}(\S) = \U$ is standard.
Thus,
\begin{equation}\label{RecursiveFormulaI}
A_{\S,\T} = 
\begin{cases}
 a_{i_k}(\S) A_{\S,\T'} + \left(1 + a_{i_k}(\S') \right)  A_{\S',\T'} & \text{if $s_{i_k}(\S) = \S'$ is standard}, \\
a_{i_k}(\S) A_{\S,\T'} & \text{if $s_{i_k}(\S)$ is nonstandard}. \\
 \end{cases}
\end{equation}

By the inductive hypothesis,
$$
A_{\S,\T'}  = \sum_{\omega_1' \subseteq \pi'} \wt_{\pi'}(\omega_1') \quad\hbox{and}\quad
A_{\S',\T'}  = \sum_{\omega_2' \subseteq \pi'} \wt_{\pi'}(\omega_2'), 
$$
where the first sum is over subpaths $\omega_1' \subseteq \pi'$ which terminate at $\S$ and the second is over subpaths $\omega_2' \subseteq \pi'$ which terminate at $\S' = s_{i_k}(\S)$.  In the first case, we append the edge $\S  \overset{e}{\longrightarrow} \S$ to $\omega_1'$ to obtain $\omega_1 \subseteq \pi$ and in the second case we append the edge 
$\S' \overset{s_{i_k}}{\longrightarrow}  \S$ to $\omega_2'$ to obtain $\omega_2 \subseteq \pi$ as seen here:
\begin{align*}
\omega_1 &= (
\underbrace{\C=\S_0 \overset{z_1}{\longrightarrow} \S_1  \overset{z_2}{\longrightarrow}  \S_2 \overset{z_3}{\longrightarrow} \cdots \overset{z_{k-1}}{\longrightarrow} \S}_{\omega_1'}  \overset{e}{\longrightarrow} \S), \\
\omega_2 &= (
\underbrace{\C=\S_0 \overset{z_1}{\longrightarrow} \S_1  \overset{z_2}{\longrightarrow}  \S_2 \overset{z_3}{\longrightarrow} \cdots \overset{z_{k-1}}{\longrightarrow} \S'}_{\omega_2'}  \overset{s_{i_k}}{\longrightarrow}  \S).
\end{align*}
In the first case, $\wt_\pi(\omega_1) = \wt_{\pi'}(\omega_1') a_{i_k}(\S)$ and in the second case $\wt_\pi(\omega_2) = \wt_{\pi'}(\omega_2') (1+a_{i_k}(\S'))$. Moreover, every subpath of $\pi$ must end with an edge weighted by $s_{i_k}$ or $e$ and therefore is of the form $\omega_1 = \omega_1' \overset{e}{\longrightarrow} \S$ or $\omega_2 = \omega_2'\overset{s_{i_k}}{\longrightarrow} \S$, where $\omega_1'$ is a subpath of $\pi'$ ending at $\S$ and $\omega_2'$ is a subpath of $\pi'$ ending at $\S'$. Thus, the sum over $\pi$-weights on all subpaths of $\pi$ equals $A_{\S,\T}$.

Notice that by construction, the definition of $A_{\S,\T}$ is independent of the particular path $\pi$ to $\T$. Thus we can choose a \emph{minimal} path $\pi = (\C=\T_0 \overset{s_{j_1}}{\longrightarrow} \T_1  \overset{s_{j_2}}{\longrightarrow}   
\cdots \overset{s_{j_m}}{\longrightarrow} \T_m= \T)$ from $\C$ to $\T$ so that its corresponding word $w = s_{j_m} s_{j_{m-1}} \cdots s_{j_1}$ is a reduced word for $w_\T$. If $\omega = ( \C = \S_0  \overset{z_{1}}{\longrightarrow} \S_1  \overset{z_{2}}{\longrightarrow}  
\cdots \overset{z_{m}}{\longrightarrow} \S_m = \S)$ is a subpath of $\pi$ which terminates at $\S$, then $w' = z_m z_{m-1} \cdots z_1$ is a permutation such that $w'(\C) = \S$. Since  $z_i \in\{ s_{j_i},e\}$ for each $i$,  $w' = z_m z_{m-1} \cdots z_1$ is a subword of $w$. It follows from the subword property (see Section \ref{subsec:BruhatTab}) that $w' \le w$, and thus $\S \le \T$. Therefore, the only $\S \in \SYT(\lambda/\mu)$ for which $A_{\S,\T}$ is nonzero are those for which $\S \le \T$ in Bruhat order, so that
$\n_\T = \sum_{\S \le \T} A_{\S,\T} \, \v_\S,$ where $A_{\S,\T} = \sum_{\omega \subseteq \pi} \wt_{\pi}(\omega).$
\end{proof}

\begin{example}\label{ex:path2}  We illustrate Theorem \ref{thm:main} for $\lambda = (3,2,1)$. The weak Bruhat graph $\B_6^{(3,2,1)}$ is shown in Figure \ref{fig:Bruhat321} and the transition matrix $\mathcal{A}_{(3,2,1)}$ is given in Example \ref{ex:diagonal_submatrix}. We compute the matrix entry $A_{\S,\T}$ for 
$$
\T =  \begin{array}{c}\begin{tikzpicture}[scale=.37,line width=1.0pt] 
\draw (0,0) rectangle (1,1); \path (0.5,0.5) node {$1$};
\draw (0,-1) rectangle (1,0); \path (0.5,-0.5) node {$4$};
\draw (0,-2) rectangle (1,-1); \path (0.5,-1.5) node {$5$};
\draw (1,0) rectangle (2,1); \path (1.5,0.5) node {$2$};
\draw (1,-1) rectangle (2,0); \path (1.5,-0.5) node {$6$};
\draw (2,0) rectangle (3,1); \path (2.5,0.5) node {$3$};
\end{tikzpicture}\end{array}\qquad\hbox{and}\qquad
\S = \begin{array}{c}\begin{tikzpicture}[scale=.37,line width=1.0pt] 
\draw (0,0) rectangle (1,1); \path (0.5,0.5) node {$1$};
\draw (0,-1) rectangle (1,0); \path (0.5,-0.5) node {$2$};
\draw (0,-2) rectangle (1,-1); \path (0.5,-1.5) node {$4$};
\draw (1,0) rectangle (2,1); \path (1.5,0.5) node {$3$};
\draw (1,-1) rectangle (2,0); \path (1.5,-0.5) node {$6$};
\draw (2,0) rectangle (3,1); \path (2.5,0.5) node {$5$};
\end{tikzpicture}\end{array}.
$$
Start with a path $\pi$ in $\B_6^{(3,2,1)}$ from $\C$ to $\T$ (our path is minimal, but as noted, this is not necessary):
$$
\begin{array}{r rc cc cc cc cc c }
\pi = & \bigg( \C =
\!\begin{array}{c}\begin{tikzpicture}[scale=.30,line width=1.0pt] 
\draw (0,0) rectangle (1,1); \path (0.5,0.5) node {$\scriptstyle{1}$};
\draw (0,-1) rectangle (1,0); \path (0.5,-0.5) node {$\scriptstyle{2}$};
\draw (0,-2) rectangle (1,-1); \path (0.5,-1.5) node {$\scriptstyle{3}$};
\draw (1,0) rectangle (2,1); \path (1.5,0.5) node {$\scriptstyle{4}$};
\draw (1,-1) rectangle (2,0); \path (1.5,-0.5) node {$\scriptstyle{5}$};
\draw (2,0) rectangle (3,1); \path (2.5,0.5) node {$\scriptstyle{6}$};
\end{tikzpicture}\end{array}\!
&\!\!\overset{s_{3}}{\longrightarrow}\!\!&
\! \begin{array}{c}\begin{tikzpicture}[scale=.30,line width=1.0pt] 
\draw (0,0) rectangle (1,1); \path (0.5,0.5) node {$\scriptstyle{1}$};
\draw (0,-1) rectangle (1,0); \path (0.5,-0.5) node {$\scriptstyle{2}$};
\draw (0,-2) rectangle (1,-1); \path (0.5,-1.5) node {$\scriptstyle{4}$};
\draw (1,0) rectangle (2,1); \path (1.5,0.5) node {$\scriptstyle{3}$};
\draw (1,-1) rectangle (2,0); \path (1.5,-0.5) node {$\scriptstyle{5}$};
\draw (2,0) rectangle (3,1); \path (2.5,0.5) node {$\scriptstyle{6}$};
\end{tikzpicture}\end{array} \!& 
\!\!\overset{s_{2}}{\longrightarrow}\!\! & 
\! \begin{array}{c}\begin{tikzpicture}[scale=.30,line width=1.0pt] 
\draw (0,0) rectangle (1,1); \path (0.5,0.5) node {$\scriptstyle{1}$};
\draw (0,-1) rectangle (1,0); \path (0.5,-0.5) node {$\scriptstyle{3}$};
\draw (0,-2) rectangle (1,-1); \path (0.5,-1.5) node {$\scriptstyle{4}$};
\draw (1,0) rectangle (2,1); \path (1.5,0.5) node {$\scriptstyle{2}$};
\draw (1,-1) rectangle (2,0); \path (1.5,-0.5) node {$\scriptstyle{5}$};
\draw (2,0) rectangle (3,1); \path (2.5,0.5) node {$\scriptstyle{6}$};
\end{tikzpicture}\end{array} \!& 
\!\!\overset{s_{5}}{\longrightarrow}\!\! & 
\! \begin{array}{c}\begin{tikzpicture}[scale=.30,line width=1.0pt] 
\draw (0,0) rectangle (1,1); \path (0.5,0.5) node {$\scriptstyle{1}$};
\draw (0,-1) rectangle (1,0); \path (0.5,-0.5) node {$\scriptstyle{3}$};
\draw (0,-2) rectangle (1,-1); \path (0.5,-1.5) node {$\scriptstyle{4}$};
\draw (1,0) rectangle (2,1); \path (1.5,0.5) node {$\scriptstyle{2}$};
\draw (1,-1) rectangle (2,0); \path (1.5,-0.5) node {$\scriptstyle{6}$};
\draw (2,0) rectangle (3,1); \path (2.5,0.5) node {$\scriptstyle{5}$};
\end{tikzpicture}\end{array}\! & 
\!\!\overset{s_{4}}{\longrightarrow}\!\!&
\! \begin{array}{c}\begin{tikzpicture}[scale=.30,line width=1.0pt] 
\draw (0,0) rectangle (1,1); \path (0.5,0.5) node {$\scriptstyle{1}$};
\draw (0,-1) rectangle (1,0); \path (0.5,-0.5) node {$\scriptstyle{3}$};
\draw (0,-2) rectangle (1,-1); \path (0.5,-1.5) node {$\scriptstyle{5}$};
\draw (1,0) rectangle (2,1); \path (1.5,0.5) node {$\scriptstyle{2}$};
\draw (1,-1) rectangle (2,0); \path (1.5,-0.5) node {$\scriptstyle{6}$};
\draw (2,0) rectangle (3,1); \path (2.5,0.5) node {$\scriptstyle{4}$};
\end{tikzpicture}\end{array}\! & 
\!\!\overset{s_{3}}{\longrightarrow}\!\!&
 \!\begin{array}{c}\begin{tikzpicture}[scale=.30,line width=1.0pt] 
\draw (0,0) rectangle (1,1); \path (0.5,0.5) node {$\scriptstyle{1}$};
\draw (0,-1) rectangle (1,0); \path (0.5,-0.5) node {$\scriptstyle{4}$};
\draw (0,-2) rectangle (1,-1); \path (0.5,-1.5) node {$\scriptstyle{5}$};
\draw (1,0) rectangle (2,1); \path (1.5,0.5) node {$\scriptstyle{2}$};
\draw (1,-1) rectangle (2,0); \path (1.5,-0.5) node {$\scriptstyle{6}$};
\draw (2,0) rectangle (3,1); \path (2.5,0.5) node {$\scriptstyle{3}$};
\end{tikzpicture}\end{array}\! = \T \bigg).
\end{array}
$$
There are the two subpaths $\omega_1$ and $\omega_2$ of $\pi$ which terminate at $\S$. Below each subpath we show its $\pi$-weight.
$$
\begin{array}{r rc cc cc cc cc c }
\omega_1 = &\bigg(\C=
\!\begin{array}{c}\begin{tikzpicture}[scale=.30,line width=1.0pt] 
\draw (0,0) rectangle (1,1); \path (0.5,0.5) node {$\scriptstyle{1}$};
\draw (0,-1) rectangle (1,0); \path (0.5,-0.5) node {$\scriptstyle{2}$};
\draw (0,-2) rectangle (1,-1); \path (0.5,-1.5) node {$\scriptstyle{3}$};
\draw (1,0) rectangle (2,1); \path (1.5,0.5) node {$\scriptstyle{4}$};
\draw (1,-1) rectangle (2,0); \path (1.5,-0.5) node {$\scriptstyle{5}$};
\draw (2,0) rectangle (3,1); \path (2.5,0.5) node {$\scriptstyle{6}$};
\end{tikzpicture}\end{array}
&\hskip-.085in\overset{s_{3}}{\longrightarrow}\hskip-.085in&
\! \begin{array}{c}\begin{tikzpicture}[scale=.30,line width=1.0pt] 
\draw (0,0) rectangle (1,1); \path (0.5,0.5) node {$\scriptstyle{1}$};
\draw (0,-1) rectangle (1,0); \path (0.5,-0.5) node {$\scriptstyle{2}$};
\draw (0,-2) rectangle (1,-1); \path (0.5,-1.5) node {$\scriptstyle{4}$};
\draw (1,0) rectangle (2,1); \path (1.5,0.5) node {$\scriptstyle{3}$};
\draw (1,-1) rectangle (2,0); \path (1.5,-0.5) node {$\scriptstyle{5}$};
\draw (2,0) rectangle (3,1); \path (2.5,0.5) node {$\scriptstyle{6}$};
\end{tikzpicture}\end{array} & 
\hskip-.085in\overset{e}{\longrightarrow}\hskip-.085in& 
\! \begin{array}{c}\begin{tikzpicture}[scale=.30,line width=1.0pt] 
\draw (0,0) rectangle (1,1); \path (0.5,0.5) node {$\scriptstyle{1}$};
\draw (0,-1) rectangle (1,0); \path (0.5,-0.5) node {$\scriptstyle{2}$};
\draw (0,-2) rectangle (1,-1); \path (0.5,-1.5) node {$\scriptstyle{4}$};
\draw (1,0) rectangle (2,1); \path (1.5,0.5) node {$\scriptstyle{3}$};
\draw (1,-1) rectangle (2,0); \path (1.5,-0.5) node {$\scriptstyle{5}$};
\draw (2,0) rectangle (3,1); \path (2.5,0.5) node {$\scriptstyle{6}$};
\end{tikzpicture}\end{array} & 
\hskip-.085in\overset{s_{5}}{\longrightarrow}\hskip-.085in & 
\! \begin{array}{c}\begin{tikzpicture}[scale=.30,line width=1.0pt] 
\draw (0,0) rectangle (1,1); \path (0.5,0.5) node {$\scriptstyle{1}$};
\draw (0,-1) rectangle (1,0); \path (0.5,-0.5) node {$\scriptstyle{2}$};
\draw (0,-2) rectangle (1,-1); \path (0.5,-1.5) node {$\scriptstyle{4}$};
\draw (1,0) rectangle (2,1); \path (1.5,0.5) node {$\scriptstyle{3}$};
\draw (1,-1) rectangle (2,0); \path (1.5,-0.5) node {$\scriptstyle{6}$};
\draw (2,0) rectangle (3,1); \path (2.5,0.5) node {$\scriptstyle{5}$};
\end{tikzpicture}\end{array} & 
\hskip-.085in\overset{e}{\longrightarrow}\hskip-.085in&
 \begin{array}{c}\begin{tikzpicture}[scale=.30,line width=1.0pt] 
\draw (0,0) rectangle (1,1); \path (0.5,0.5) node {$\scriptstyle{1}$};
\draw (0,-1) rectangle (1,0); \path (0.5,-0.5) node {$\scriptstyle{2}$};
\draw (0,-2) rectangle (1,-1); \path (0.5,-1.5) node {$\scriptstyle{4}$};
\draw (1,0) rectangle (2,1); \path (1.5,0.5) node {$\scriptstyle{3}$};
\draw (1,-1) rectangle (2,0); \path (1.5,-0.5) node {$\scriptstyle{6}$};
\draw (2,0) rectangle (3,1); \path (2.5,0.5) node {$\scriptstyle{5}$};
\end{tikzpicture}\end{array} & 
\hskip-.085in\overset{e}{\longrightarrow}\hskip-.085in &
\! \begin{array}{c}\begin{tikzpicture}[scale=.30,line width=1.0pt] 
\draw (0,0) rectangle (1,1); \path (0.5,0.5) node {$\scriptstyle{1}$};
\draw (0,-1) rectangle (1,0); \path (0.5,-0.5) node {$\scriptstyle{2}$};
\draw (0,-2) rectangle (1,-1); \path (0.5,-1.5) node {$\scriptstyle{4}$};
\draw (1,0) rectangle (2,1); \path (1.5,0.5) node {$\scriptstyle{3}$};
\draw (1,-1) rectangle (2,0); \path (1.5,-0.5) node {$\scriptstyle{6}$};
\draw (2,0) rectangle (3,1); \path (2.5,0.5) node {$\scriptstyle{5}$};
\end{tikzpicture}\end{array}\! = \S \bigg),
\\ 
&&\hskip-.085in 1 + \frac{1}{3} \hskip-.085in&&\hskip-.085in \frac{1}{2} \hskip-.085in&&\hskip-.085in 1 + \frac{1}{2} \hskip-.085in&& \hskip-.085in \frac{1}{4} \hskip-.085in&&\hskip-.085in- \frac{1}{3} \hskip-.085in& 
\\ \\ 
\omega_2 = & \bigg(\C=
\!\begin{array}{c}\begin{tikzpicture}[scale=.30,line width=1.0pt] 
\draw (0,0) rectangle (1,1); \path (0.5,0.5) node {$\scriptstyle{1}$};
\draw (0,-1) rectangle (1,0); \path (0.5,-0.5) node {$\scriptstyle{2}$};
\draw (0,-2) rectangle (1,-1); \path (0.5,-1.5) node {$\scriptstyle{3}$};
\draw (1,0) rectangle (2,1); \path (1.5,0.5) node {$\scriptstyle{4}$};
\draw (1,-1) rectangle (2,0); \path (1.5,-0.5) node {$\scriptstyle{5}$};
\draw (2,0) rectangle (3,1); \path (2.5,0.5) node {$\scriptstyle{6}$};
\end{tikzpicture}\end{array}
&\hskip-.085in\overset{e}{\longrightarrow}\hskip-.085in & 
\!\begin{array}{c}\begin{tikzpicture}[scale=.30,line width=1.0pt] 
\draw (0,0) rectangle (1,1); \path (0.5,0.5) node {$\scriptstyle{1}$};
\draw (0,-1) rectangle (1,0); \path (0.5,-0.5) node {$\scriptstyle{2}$};
\draw (0,-2) rectangle (1,-1); \path (0.5,-1.5) node {$\scriptstyle{3}$};
\draw (1,0) rectangle (2,1); \path (1.5,0.5) node {$\scriptstyle{4}$};
\draw (1,-1) rectangle (2,0); \path (1.5,-0.5) node {$\scriptstyle{5}$};
\draw (2,0) rectangle (3,1); \path (2.5,0.5) node {$\scriptstyle{6}$};
\end{tikzpicture}\end{array} \!& 
\hskip-.085in\overset{e}{\longrightarrow}\hskip-.085in & 
\!\begin{array}{c}\begin{tikzpicture}[scale=.30,line width=1.0pt] 
\draw (0,0) rectangle (1,1); \path (0.5,0.5) node {$\scriptstyle{1}$};
\draw (0,-1) rectangle (1,0); \path (0.5,-0.5) node {$\scriptstyle{2}$};
\draw (0,-2) rectangle (1,-1); \path (0.5,-1.5) node {$\scriptstyle{3}$};
\draw (1,0) rectangle (2,1); \path (1.5,0.5) node {$\scriptstyle{4}$};
\draw (1,-1) rectangle (2,0); \path (1.5,-0.5) node {$\scriptstyle{5}$};
\draw (2,0) rectangle (3,1); \path (2.5,0.5) node {$\scriptstyle{6}$};
\end{tikzpicture}\end{array}& 
\hskip-.085in\overset{s_{5}}{\longrightarrow}\hskip-.085in & 
\!\begin{array}{c}\begin{tikzpicture}[scale=.30,line width=1.0pt] 
\draw (0,0) rectangle (1,1); \path (0.5,0.5) node {$\scriptstyle{1}$};
\draw (0,-1) rectangle (1,0); \path (0.5,-0.5) node {$\scriptstyle{2}$};
\draw (0,-2) rectangle (1,-1); \path (0.5,-1.5) node {$\scriptstyle{3}$};
\draw (1,0) rectangle (2,1); \path (1.5,0.5) node {$\scriptstyle{4}$};
\draw (1,-1) rectangle (2,0); \path (1.5,-0.5) node {$\scriptstyle{6}$};
\draw (2,0) rectangle (3,1); \path (2.5,0.5) node {$\scriptstyle{5}$};
\end{tikzpicture}\end{array} & 
\hskip-.085in\overset{e}{\longrightarrow}\hskip-.085in&
\!\begin{array}{c}\begin{tikzpicture}[scale=.30,line width=1.0pt] 
\draw (0,0) rectangle (1,1); \path (0.5,0.5) node {$\scriptstyle{1}$};
\draw (0,-1) rectangle (1,0); \path (0.5,-0.5) node {$\scriptstyle{2}$};
\draw (0,-2) rectangle (1,-1); \path (0.5,-1.5) node {$\scriptstyle{3}$};
\draw (1,0) rectangle (2,1); \path (1.5,0.5) node {$\scriptstyle{4}$};
\draw (1,-1) rectangle (2,0); \path (1.5,-0.5) node {$\scriptstyle{6}$};
\draw (2,0) rectangle (3,1); \path (2.5,0.5) node {$\scriptstyle{5}$};
\end{tikzpicture}\end{array} & 
\hskip-.085in\overset{s_{3}}{\longrightarrow}\hskip-.085in&
\!\begin{array}{c}\begin{tikzpicture}[scale=.30,line width=1.0pt] 
\draw (0,0) rectangle (1,1); \path (0.5,0.5) node {$\scriptstyle{1}$};
\draw (0,-1) rectangle (1,0); \path (0.5,-0.5) node {$\scriptstyle{2}$};
\draw (0,-2) rectangle (1,-1); \path (0.5,-1.5) node {$\scriptstyle{4}$};
\draw (1,0) rectangle (2,1); \path (1.5,0.5) node {$\scriptstyle{3}$};
\draw (1,-1) rectangle (2,0); \path (1.5,-0.5) node {$\scriptstyle{6}$};
\draw (2,0) rectangle (3,1); \path (2.5,0.5) node {$\scriptstyle{5}$};
\end{tikzpicture}\end{array}\! = \S\bigg).
\\ 
&&\hskip-.085in \frac{1}{3} \hskip-.085in&&\hskip-.085in -1 \hskip-.085in&&\hskip-.085in 1 + \frac{1}{2}\hskip-.085in &&\hskip-.085in 1 \hskip-.085in&&\hskip-.085in 1+ \frac{1}{3} \hskip-.085in&
\end{array}
$$
Thus, the $A_{\S,\T}$ entry is the sum of these two weights:
$$
A_{\S,\T} = 
\left(1+\tfrac{1}{3}\right) \cdot \tfrac{1}{2} \cdot \left(1 + \tfrac{1}{2}\right) \cdot \tfrac{1}{4} \cdot \left(- \tfrac{1}{3}\right)+ \tfrac{1}{3} \cdot (-1) \cdot \left(1 + \tfrac{1}{2}\right) \cdot 1 \cdot \left(1 + \tfrac{1}{3}\right) = -\tfrac{1}{12} - \tfrac{2}{3} = - \tfrac{3}{4},
$$
which is the (6,15) entry in Example \ref{ex:diagonal_submatrix}.
\end{example}

\begin{rem}\label{rem:upper-triangular} Let $\mathcal{A}_{\lambda/\mu} = (A_{\S,\T})_{\S,\T \in \SYT(\lambda)}$ denote the transition matrix defined in \eqref{TransitionCoefficients} whose entries are computed in Theorem \ref{thm:main}. This matrix satisfies the following properties:
\begin{enumerate}[(a)]

\item Since $A_{\S,\T} = 0$ if $\S \not\le \T$ in Bruhat order, $\mathcal{A}_{\lambda/\mu}$ is upper-triangular under any linear ordering of standard tableaux that extends  Bruhat order.

\item If $\pi$ is a minimal path from $\C$ to $\T$ of length $k$, then the only subpath of $\pi$ that reaches depth $k$ in the weak Bruhat graph is $\pi$ itself, so $A_{\S,\T} = 0$ if $\S$ has depth $k$ and $\S \not=\T$. It follows that if the standard tableaux are grouped  by depth, the matrix $\mathcal{A}_{\lambda/\mu}$ has diagonal matrix blocks on the main diagonal. Bruhat order respects depth, so any linear extension of Bruhat order will exhibit this phenomenon. See Example \ref{ex:diagonal_submatrix}.
\end{enumerate}
\end{rem}

\begin{example}\label{chgA32}
The weak Bruhat graph $\B_5^{(3,2)}$ and the corresponding transition matrix  $\mathcal{A}_{(3,2)}$ between Young's natural and seminormal representations of $\Sym_5^{(3,2)}$.
\begin{equation*}
\Yboxdim{7pt}
\B^{(3,2)}_5=
\begin{array}{c}
 \includegraphics[height=2.5in,page=1]{BruhatGraphs.pdf} 
 \end{array}\hskip.15in
 \mathcal{A}_{(3,2)}=
\begin{blockarray}{cccccc}
\phantom{\tiny\young(1,1,1)} & \n_{\tiny\young(135,24)}& \n_{\tiny\young(125,34)}&\n_{\tiny\young(134,25)}&\n_{\tiny\young(124,35)}&\n_{\tiny\young(123,45)} \\
\begin{block}{c(ccccc)}
\v_{\tiny\young(135,24)} & 1 & \frac{1}{2} & \frac{1}{2} & \frac{1}{4} & -\frac{1}{4} \\
\v_{\tiny\young(125,34)} & 0 & \frac{3}{2} & 0 & \frac{3}{4} & \frac{3}{4} \\
\v_{\tiny\young(134,25)} & 0 & 0 & \frac{3}{2} & \frac{3}{4} & \frac{3}{4} \\
\v_{\tiny\young(124,35)} & 0 & 0 & 0 & \frac{9}{4} & \frac{3}{4} \\
\v_{\tiny\young(123,45)} & 0 & 0 & 0 & 0 & 3 \\
\end{block}
\end{blockarray}
\end{equation*}
\end{example}

\begin{cor}\label{cor:recursive}
If $\S,\T \in \SYT(\lambda/\mu)$ and $\T'=s_\ell(\T)$ is standard, then
\begin{equation*}
A_{\S,\T} = 
\begin{cases}
 a_{\ell}(\S) A_{\S,\T'} + \left(1 - a_{\ell}(\S) \right)  A_{\S',\T'} & \text{if $s_{\ell}(\S) = \S'$ is standard}, \\
a_{\ell}(\S) A_{\S,\T'} & \text{if $s_{\ell}(\S)$ is nonstandard}. \\
 \end{cases}
\end{equation*}
\end{cor}

\begin{proof}
This follows from equation  \eqref{RecursiveFormulaI} in the proof of Theorem \ref{thm:main}, and from the fact that $a_\ell(\S') = - a_\ell(\S)$ since $\S' = s_\ell(\S)$.   \end{proof}

\begin{rem} \label{rem:complexity}  Corollary  \ref{cor:recursive}  says that the entries of the $\T$-th column of the matrix $\mathcal{A}_{\lambda/\mu}$ equal a weighted sum of at most two entries of the $\T'$-th column.  For $\T \not=\C$, we can choose $\ell$ so that $s_\ell(\T) \lessdot_W \T$ (i.e., choose $\T'$ immediately above $\T$ in the weak Bruhat graph). Thus, each entry  of $\mathcal{A}_{\lambda/\mu}$ is a weighted sum of two entries in a column to its its left (under any ordering on tableaux that respects Bruhat order). Since $\mathcal{A}_{\lambda/\mu}$ is upper-triangular with $f^{\lambda / \mu}$ rows and columns, computing the columns of $\mathcal{A}_{\lambda/\mu}$ in any order that respects Bruhat order requires a maximum of $2 (1+ 2 + \cdots + f^{\lambda/\mu}) = (f^{\lambda/\mu})^2 + f^{\lambda/\mu}$ operations once the contents of the boxes of the tableaux have been computed (which takes $n$ opertations for each tableau).  At the time of the writing of this paper, we were able to compute the transition matrices for the 42 irreducible representations of $\Sym_{10}$ (3.6 million elements) in a total of 24 minutes using {\it Mathematica} on a laptop computer.
\end{rem}

\Yboxdim{7pt}
\begin{example} We illustrate Corollary \ref{cor:recursive} with $\lambda = (3,2,1)$ and compute two of the entries of Example \ref{ex:diagonal_submatrix}. To compute the entry $A_{\T_1,\T_{16}} = \frac{1}{12}$ we use $\ell = 5$, since $s_5(\T_{16}) = \T_{15}$. Thus,
$$
A_{\tiny{\young(146,25,3)},\tiny\young(123,45,6)}  = \frac{1}{2}  A_{\tiny{\young(146,25,3)},\tiny\young(123,46,5)}  + \left(1 - \frac{1}{2}\right) A_{\tiny{\young(145,26,3)},\tiny\young(123,46,5)}
=  \frac{1}{2} \cdot \frac{1}{6}  + \left(\frac{1}{2}\right) \cdot 0  = \frac{1}{12}.
$$
We could also have chosen $\ell = 3$, since $s_3(\T_{16}) = \T_{14}$. In this case,
$$
A_{\tiny{\young(146,25,3)},\tiny\young(123,45,6)}  = \frac{1}{3} A_{\tiny{\young(146,25,3)},\tiny\young(124,35,6)} + \left(1 - \frac{1}{3}\right) A_{\tiny{\young(136,25,4)},\tiny\young(124,35,6)}
=  \frac{1}{3} \cdot \left(-\frac{1}{6}\right)  + \frac{2}{3} \cdot \frac{5}{24} = \frac{1}{12}.
$$
Similarly, to compute the entry $A_{\T_2, \T_{13}} = \frac{5}{12}$, we choose $\ell = 4$ since $s_4(\T_{13}) = \T_9$, and get
$$
A_{\tiny{\young(136,25,4)},\tiny\young(134,25,6)}  = \frac{1}{2} A_{\tiny{\young(136,25,4)},\tiny\young(135,24,6)} + \left(1 - \frac{1}{2}\right) A_{\tiny{\young(136,24,5)},\tiny\young(135,24,6)}
=  \frac{1}{2} \cdot \frac{1}{3}  + \frac{1}{2} \cdot \frac{1}{2} = \frac{5}{12}.
$$
\end{example}

As in Section \ref{subsec:BruhatTab}, we define an inversion in a standard tableau $\T$ to be  a pair $(i,j)$ such that $i > j$ and $i$ is strictly south and strictly west of $j$ in $\T$, and we let $\inv(\T)$  be the set of inversions in $\T$. Then, the diagonal entries of the transition matrix satisfy the following formula.

\begin{prop}\label{prop:diagonals} For any $\T \in \SYT(\lambda/\mu)$, we have
$\displaystyle{
A_{\T,\T} = \!\prod_{(i,j) \in \inv(\T)} (1 +  {a_{i,j}(\T)}). 
}$
\end{prop} 

\begin{proof}
Let $\T \in \SYT(\lambda/\mu)$, let $w_\T = s_{i_k}\cdots s_{i_2} s_{i_1}$ be a reduced word, and proceed by induction on $k$. When $k = 0$,  $A_{\T,\T} = A_{\C,\C} = 1$ is the product over the empty inversion set.  For $k \ge 1$, let $s_{i_k} = s_\ell$ and $\T' = s_\ell( \T)$ so that $\T' \lessdot \T$. By Corollary \ref{cor:recursive},
\begin{equation*}
\label{eq:inductivestep}
A_{\T,\T} =  a_{\ell}(\T) A_{\T,\T'} + \left(1 -a_\ell(\T) \right)  A_{\T',\T'} 
= \left(1 -a_{\ell,\ell+1}(\T) \right)  A_{\T',\T'}  
= \left(1 + a_{\ell+1,\ell}(\T) \right)\!\!\! \prod_{(i,j) \in \inv(\T')} (1 +  {a_{i,j}(\T')}),
\end{equation*}
where the second equality uses the fact that $A_{\T,\T'} = 0$ since $\T \not\le \T'$, and the third equality uses the inductive hypothesis. 

From Lemma \ref{lemma:inversionswapping}, $\inv(\T) = s_\ell(\inv(\T)) \cup \{(\ell+1, \ell)\}$, and in this pairing, the inversions in $\T$ correspond to inversions in the same position in the tableau $\T'$. Thus, the axial distances (and therefore the values of $a_{i,j}$) between the corresponding boxes are the same, and the result follows:
\begin{equation*}
A_{\T,\T} =  (1 + a_{\ell + 1, \ell}(\T)) \prod_{(i,j) \in \inv(\T')} (1 +  {a_{i,j}(\T')}) =  \prod_{(i,j) \in \inv(\T)} (1 +  {a_{i,j}(\T)}).
\end{equation*}
\end{proof}

\Yboxdim{10pt}
\begin{example} The tableau $\T_{12}$ from Example \ref{ex:diagonal_submatrix} has four inversions: (3,2), (5,2), (5,4), (6,4), and by Proposition  \ref{prop:diagonals}, the matrix entry $A_{\T_{12},\T_{12}}$ is the given product over those inversions,
$$
\begin{array}{rcl}
\begin{array}{l}
\T_{12} = \young(124,36,5)
\end{array}& & 
\renewcommand{\arraystretch}{1.5}
\begin{array}{rcl}
A_{\T_{12},\T_{12}} & = & ( 1 + a_{3,2}(\T_{12}) ) ( 1 + a_{5,2}(\T_{12}) ) ( 1 + a_{5,4}(\T_{12}) ) ( 1 + a_{6,4}(\T_{12}) ) \\
& = & (1 + \frac{1}{2}) (1 + \frac{1}{3}) (1 + \frac{1}{4} ) (1 + \frac{1}{2})  = \frac{3}{2} \cdot \frac{4}{3} \cdot \frac{5}{4} \cdot \frac{3}{2} = \frac{15}{4}.
\end{array}
\end{array}
$$
\end{example}


\begin{example}\label{ex:diagonal_submatrix}
To demonstrate Remark \ref{rem:upper-triangular}(b), if we order the standard tableaux of shape $(3,2,1)$ according to their depth in the weak Bruhat graph as follows:
$$
\Yboxdim{7pt}
{\tiny
\T_1 = \!\! \begin{array}{c} \young(146,25,3)\end{array}\!\!\!\!,\ 
\T_2 = \!\! \begin{array}{c} \young(136,25,4)\end{array}\!\!\!\!,\ 
\T_3 = \!\! \begin{array}{c} \young(145,26,3)\end{array}\!\!\!\!,\ 
\T_4 = \!\! \begin{array}{c} \young(126,35,4)\end{array}\!\!\!\!,\ 
\T_5 = \!\! \begin{array}{c} \young(136,24,5)\end{array}\!\!\!\!,\ 
\T_6 = \!\! \begin{array}{c} \young(135,26,4)\end{array}\!\!\!\!,\
\T_7 = \!\! \begin{array}{c} \young(126,34,5)\end{array}\!\!\!\!,\ 
\T_8 = \!\! \begin{array}{c} \young(125,36,4)\end{array}\!\!\!\!,}\hskip.4in
$$
$$
\Yboxdim{7pt}
\tiny{
\T_9 = \!\! \begin{array}{c} \young(135,24,6)\end{array}\!\!\!\!,\ 
\T_{10} = \!\! \begin{array}{c} \young(134,26,5)\end{array}\!\!\!\!,\ 
\T_{11} = \!\! \begin{array}{c} \young(125,34,6)\end{array}\!\!\!\!,\ 
\T_{12} = \!\! \begin{array}{c} \young(124,36,5)\end{array}\!\!\!\!,\
\T_{13} = \!\! \begin{array}{c} \young(134,25,6)\end{array}\!\!\!\!,\ 
\T_{14} = \!\! \begin{array}{c} \young(124,35,6)\end{array}\!\!\!\!,\ 
\T_{15} = \!\! \begin{array}{c} \young(123,46,5)\end{array}\!\!\!\!,\ 
\T_{16} = \!\! \begin{array}{c} \young(123,45,6)\end{array}\!\!\!\!,
}$$
then with this indexing, $\mathcal{A}_{(3,2,1)}$ has a diagonal submatrix for each distinct row in the weak Bruhat graph $\mathcal{B}_6^{(3,2,1)}$, as shown below.
$$\renewcommand*{\arraystretch}{1.5}
\mathcal{A}_{(3,2,1)} = 
\left(
\begin{array}{c|cc|ccc|cccc|ccc|cc|c}
 1 & \frac{1}{3} & \frac{1}{2} & -\frac{1}{3} & -\frac{1}{3} & \frac{1}{6} & \frac{1}{3} & -\frac{1}{6} & -\frac{1}{6} & -\frac{1}{6} & \frac{1}{6} & \frac{1}{6} & \frac{1}{6} & -\frac{1}{6} & \frac{1}{6} & \frac{1}{12} \\ \hline
 0 & \frac{4}{3} & 0 & \frac{2}{3} & \frac{2}{3} & \frac{2}{3} & \frac{1}{3} & \frac{1}{3} & \frac{1}{3} & \frac{1}{3} & \frac{1}{6} & \frac{1}{6} & \frac{5}{12} & \frac{5}{24} & \frac{1}{6} & -\frac{7}{24} \\
 0 & 0 & \frac{3}{2} & 0 & 0 & \frac{1}{2} & 0 & -\frac{1}{2} & -\frac{1}{2} & \frac{1}{2} & \frac{1}{2} & -\frac{1}{2} & -\frac{1}{2} & \frac{1}{2} & 0 & \frac{1}{4} \\ \hline
 0 & 0 & 0 & 2 & 0 & 0 & 1 & 1 & 0 & 0 & \frac{1}{2} & \frac{1}{2} & 0 & \frac{5}{8} & -\frac{1}{2} & -\frac{5}{8} \\ 
 0 & 0 & 0 & 0 & 2 & 0 & 1 & 0 & \frac{1}{2} & 1 & \frac{1}{4} & \frac{1}{2} & \frac{1}{4} & \frac{1}{8} & -\frac{1}{2} & -\frac{1}{8} \\
 0 & 0 & 0 & 0 & 0 & 2 & 0 & 1 & 1 & \frac{1}{2} & \frac{1}{2} & \frac{1}{4} & \frac{1}{4} & \frac{1}{8} & -\frac{3}{4} & \frac{5}{8} \\ \hline
 0 & 0 & 0 & 0 & 0 & 0 & 3 & 0 & 0 & 0 & \frac{3}{4} & \frac{3}{2} & 0 & \frac{3}{8} & \frac{3}{2} & \frac{3}{8} \\ 
 0 & 0 & 0 & 0 & 0 & 0 & 0 & 3 & 0 & 0 & \frac{3}{2} & \frac{3}{4} & 0 & \frac{3}{8} & -\frac{3}{4} & -\frac{3}{8} \\
 0 & 0 & 0 & 0 & 0 & 0 & 0 & 0 & \frac{5}{2} & 0 & \frac{5}{4} & 0 & \frac{5}{4} & \frac{5}{8} & 0 & -\frac{5}{8} \\
 0 & 0 & 0 & 0 & 0 & 0 & 0 & 0 & 0 & \frac{5}{2} & 0 & \frac{5}{4} & \frac{5}{4} & \frac{5}{8} & \frac{5}{4} & \frac{5}{8} \\ \hline
 0 & 0 & 0 & 0 & 0 & 0 & 0 & 0 & 0 & 0 & \frac{15}{4} & 0 & 0 & \frac{15}{8} & 0 & \frac{15}{8} \\ 
 0 & 0 & 0 & 0 & 0 & 0 & 0 & 0 & 0 & 0 & 0 & \frac{15}{4} & 0 & \frac{15}{8} & \frac{5}{4} & \frac{5}{8} \\
 0 & 0 & 0 & 0 & 0 & 0 & 0 & 0 & 0 & 0 & 0 & 0 & \frac{15}{4} & \frac{15}{8} & 0 & \frac{15}{8} \\ \hline
 0 & 0 & 0 & 0 & 0 & 0 & 0 & 0 & 0 & 0 & 0 & 0 & 0 & \frac{45}{8} & 0 & \frac{15}{8} \\ 
 0 & 0 & 0 & 0 & 0 & 0 & 0 & 0 & 0 & 0 & 0 & 0 & 0 & 0 & 5 & \frac{5}{2} \\ \hline
 0 & 0 & 0 & 0 & 0 & 0 & 0 & 0 & 0 & 0 & 0 & 0 & 0 & 0 & 0 & \frac{15}{2} \\
\end{array}
\right)$$
\end{example}


\section{Affine Hecke Algebras}
\label{sec:Hecke}

In this section, we follow the notations and definitions of \cite{Ram-skew,ram-affine,RR}.  
Fix an element $q \in \CC^\ast$ that is not a root of unity.
The \defn{affine Hecke algebra} $\tilde{\HH}_n$ of type A is the associative algebra with 1 over $\CC$ generated by $\Th_i, 1 \le i \le n-1$, and $X^{\varepsilon_i},  1 \le i \le n$, subject to the  following relations:
\begin{equation*}
\begin{array}{cllccll}
\text{(a)} &  T_i T_j = T_j T_i, &   |i - j| > 1,  && \text{(d)} &    T_i X^{\varepsilon_j} =  X^{\varepsilon_j} T_i , &   |i - j| > 1,  \\
\text{(b)} &   T_i T_{i+1} T_i = T_{i+1} T_i T_{i+1} , &  1 \le i \le n-2, \ && \text{(e)} &   X^{\varepsilon_1} s_{1} X^{\varepsilon_1} s_{1}  = s_1 X^{\varepsilon_1} s_1 X^{\varepsilon_1},   \\
\text{(c)} & \Th_i^2 = (q - q^{-1}) \Th_i + 1,& 1 \le i \le n-1, && \text{(f)} &    X^{\varepsilon_{i+1}}=  T_i X^{\varepsilon_i} T_i , &  1 \le i \le n-1.  \\
\end{array}
\end{equation*}
If $w \in \Sym_n$ with reduced word $w = s_{i_1} s_{i_2} \cdots s_{i_k}$, then we define $\Th_w = \Th_{i_1} \Th_{i_2} \cdots \Th_{i_k} \in \tilde{\HH}_n$.  Since the $\Th_i$ satisfy the braid relations (a) and (b), the definition of $\Th_w$ is independent of the reduced expression for $w$.

The irreducible, calibrated (see \cite{Cherednik,Ram-skew,ram-affine,RR}) $\tilde{\HH}_n$-modules are indexed by \defn{placed skew shapes}. These are created by decomposing a skew shape $\lambda/\mu$ into a disjoint union of skew shapes,
\begin{equation}
\lambda/\mu = \bigsqcup_\beta \left(\lambda^{(\beta)}/\mu^{(\beta)}\right), \qquad \text{ indexed by } \quad \beta \in [0,1) + i [0 , 2 \pi/\ln(q^2)),
\end{equation}
where $\beta$ is called the \emph{page number} of the skew shape $\lambda^{(\beta)}/\mu^{(\beta)}$.  The corresponding \defn{content function} (see \cite{Ram-skew,RR}) is the function $c: \{\text{boxes of $\lambda/\mu$}\}  \to  \RR + i [0 , 2 \pi/\ln(q^2)),$ given by
\begin{equation}\label{def:Affinecontent}
c(b) = \beta + ct(b) = \beta + y - x, \qquad \text{if $b$ is in row $x$ and column $y$ of $\lambda^{(\beta)}/\mu^{(\beta)}$},
\end{equation}
where $ct(b)$ is the content \eqref{def:Sncontent} of the box $b$ used in the seminormal representations of $\Sym_n$, except we now allow each skew shape to be placed arbitrarily, so that the top left box need not lie in position (1,1) (see, in particular,  \cite[(3.5)]{RR}). The pair $(c, \lambda/\mu)$ determines the corresponding placed skew shape $\lambda/\mu = \bigsqcup_\beta \left(\lambda^{(\beta)}/\mu^{(\beta)}\right)$.

For a skew shape $\lambda/\mu$ with $n$ boxes, Ram \cite{Ram-skew} proves that the irreducible, calibrated  $\tilde{\HH}_n$-module corresponding to the placed skew shape $(c, \lambda/\mu)$, is the vector space
$
\tilde{\HH}_n^{(c,\lambda/\mu)} = \CC\text{-span}\left\{ \v_\T \mid \T \in \SYT(\lambda/\mu) \right\},
$
with $\tilde{\HH}_n$-action given by the formulas
\begin{equation} \label{affine-local-action}
X^{\varepsilon_i} \v_\T  = q^{2 c(\T(i))} \v_\T \qquad\hbox{and} \qquad \Th_i \v_\T  = \tilde{a}_i(\T) \v_\T +  \left(q^{-1} + \tilde{a}_i(\T)  \right) \v_{s_i(\T)},
\end{equation}
where $ \v_{s_i(\T)} = 0$ if $s_i(\T)$ is nonstandard, $\T(i)$ is the box of $\T$ containing $i$, and 
\begin{equation}\label{def:q-diagonal}
\tilde{a}_i(\T) = \frac{q-q^{-1}} {1 -  q^{2 (c (\T(i)) - c(\T(i+1))}}.
\end{equation}
This is the \emph{seminormal basis} $\tilde{\mathcal{V}}^{(c,\lambda/\mu)}= \{ \v_\T \mid \T \in \SYT(\lambda / \mu) \}$ of $\tilde{\HH}_n^{(c,\lambda/\mu)}$, and the \emph{natural basis} $\tilde{\mathcal{N}}^{(c,\lambda/\mu)} = \{ \n_\T \mid \T \in \SYT(\lambda/\mu) \}$ is defined (\cite[(5.2)]{Ram-skew})  from the seminormal basis analogously to how it is done for the symmetric group; namely,
\begin{equation}\label{affine-natural-rep}
\n_\T := \Th_{w_\T} \v_\C, \qquad \text{where $\C$ is the column reading tableau and $w_\T \in \Sym_n$ such that $w_\T(\C) = \T$}.
\end{equation}

Analogous to \eqref{Sn:weight}, for a path $\pi = (\T_0 \overset{s_{i_1}}{\longrightarrow} \T_1  \overset{s_{i_2}}{\longrightarrow}  \T_2 \overset{s_{i_3}}{\longrightarrow} \cdots \overset{s_{i_k}}{\longrightarrow} \T_k)$ and a subpath $\omega = ( \T_0 = \S_0  \overset{z_{1}}{\longrightarrow} \S_1  \overset{z_{2}}{\longrightarrow}  \S_2 \overset{z_{3}}{\longrightarrow} \cdots \overset{z_{k}}{\longrightarrow} \S_k)$ of $\pi$ in $\mathcal{B}_n^{\lambda/\mu}$, define 
the \defn{$q$-weight} of $\omega$ as
\begin{equation}\label{def:q-weight}
\tilde\wt_\pi(\omega) = \prod_{j=1}^k \tilde b_j, 
\qquad\text{where}\quad
\tilde b_j = \begin{cases}
\tilde a_{i_j} (\S_{j-1}), & \text{if $\S_{j-1}\overset{e}{\longrightarrow} \S_{j}$,  i.e.,  if $\S_j = \S_{j-1}$}, \\
\tilde a_{i_j} (\S_{j-1})+q^{-1}, & \text{if $\S_{j-1}\overset{s_{i_j}}{\longrightarrow} \S_{j}$, i.e., if $\S_j = s_{i_j}(\S_{i-1})$},
\end{cases}
\end{equation}
or, equivalently, 
$\tilde{\wt}_\pi(\omega) = \prod_{j=1}^k \tilde{a}_{i_j}(\S_{j-1}) +q^{-1} - q^{-1}\delta_{\S_{j-1},\S_j}.$

Since the natural basis is related to the seminormal basis \eqref{affine-natural-rep} in the same way that they are related for the symmetric group, and since the seminormal representation satisfies  \eqref{affine-local-action}, the following change-of-basis theorem for $\tilde{\HH}_n$ is identical to that of $\Sym_n$ except that weights are updated with the values from \eqref{def:q-weight}.

\begin{thm}\label{thm:affine} For a placed skew shape $(c,\lambda/\mu)$ with $n$ boxes, the change-of-basis for $\tilde{\HH}_n^{(c,\lambda/\mu)}$ expressing the natural basis in terms of the seminormal basis satisfies
$$
\n_\T = \sum_{\S \le \T} \tilde{A}_{\S,\T} \, \v_\S \quad\hbox{with coefficients}\quad \tilde{A}_{\S,\T} = \sum_{\omega \subseteq \pi} \tilde{\wt}_{\pi}(\omega),
$$
where the first sum is over standard tableaux $\S \le \T$ in (strong) Bruhat order;
the coefficient $\tilde{A}_{\S,\T}$ is defined using any path $\pi$ from $\C$ to $\T$ in $\B_n^{\lambda/\mu}$; 
and the second sum is over all subpaths $\omega \subseteq \pi$ which terminate at $\S$.
\end{thm}

\begin{proof} By comparing  \eqref{affine-local-action} with  \eqref{eq:seminormal} we see that $T_i$ acts on $\v_\T$ the same as $s_i$ acts on $\v_\T$ except that the diagonal entry $a_i(\T)$ is replaced by $\tilde a_i(\T)$ and the off-diagonal entry $1 + a_i(\T)$ is replaced by $q^{-1} + \tilde{a}_i(\T)$ .  Moreover, if $w_\T = s_{j_1} s_{j_2} \ldots s_{j_\ell}$ is any reduced expression for the word of $\T$, then $T_{w_\T} = T_{j_1} T_{j_2} \cdots T_{j_\ell}$.  Thus, the vector $\n_\T = T_{w_\T} \v_c$ for the affine Hecke algebra is the same as $\n_\T = w_\T \v_c$ for the symmetric group, except that the $a_{j_i}(\T)$ are replaced by $\tilde a_{j_i}(\T)$ and $1 + a_{j_i}(\T)$ is replaced by $q^{-1} + \tilde{a}_{j_i}(\T)$. The proof is then identical to the proof of Theorem \ref{thm:main}.
\end{proof}

\begin{cor}\label{cor:affinediagonal}
For any\  $\T \in \SYT(\lambda/\mu)$, \
$\displaystyle{
\tilde{A}_{\T,\T} = \!\! \prod_{(i,j) \in \inv(\T)}\! (q^{-1} +  {\tilde{a}_{i,j}(\T)}), 
}$ where
$
\displaystyle{\tilde{a}_{i,j}(\T) = \frac{q-q^{-1}} {1 -  q^{2 (c (\T(i)) - c(\T(j))}}}.
$
\end{cor}

\begin{proof}
As in the proof of Theorem \ref{thm:affine}, notice that $T_i$ acts on $\v_\T$ the same as $s_i$ acts on $\v_\T$, except that the weights in \eqref{eq:seminormal} are replaced by the weights given in \eqref{affine-local-action}. Thus the same inductive proof of Proposition \ref{prop:diagonals} gives the desired result.
\end{proof}

\begin{rem} \label{rem:q-recursive}
The transition matrix $\mathcal{A}_{(c,\lambda/\mu)}$, defined by Theorem  \ref{thm:affine}, satisfies the same upper triangular properties as in the symmetric group case, discussed in Remark \ref{rem:upper-triangular}. The recursive formula of  Corollary \ref{cor:recursive} also holds for $\tilde{A}_{\S,\T}$ after replacing $a_\ell$ with $\tilde{a}_\ell$ and replacing $1 - a_\ell$ with $q^{-1} - \tilde{a}_\ell$.
\end{rem}

\section{Cyclotomic Hecke Algebras $\HH_{r,n}$ and the Wreath Product Group $\GG_{r,n}$}

\subsection{$r$-tableaux} \label{sec:r-tableaux}

An \defn{$r$-partition} of $n$ is an ordered $r$-tuple $\rpar = (\lambda^{(1)} \ldots, \lambda^{(r)})$ such that each $\lambda^{(i)}$ is a partition and $|\lambda^{(1)}| + \cdots + |\lambda^{(r)}| = n$. 
An \defn{$r$-tableau} of shape $\rpar$ is a filling of the boxes of $\rpar$ with the integers $1, 2, \ldots, n$ such that each integer appears exactly once. An $r$-tableau is {standard} if, for each $1 \le i \le r$, the entires of $\lambda^{(i)}$  increase from left to right in each row and from top to bottom in each column. We denote the set of $r$-tableaux of shape $\rpar$ by $\YT(\rpar)$ and the set of standard $r$-tableaux by $\SYT(\rpar)$. The number of standard $r$-tableaux of shape $\rpar$ is given by $f^{\rpar} = \binom{n}{|\lambda^{(1)}|, \ldots, |\lambda^{(r)}|} f^{\lambda^{(1)}} \cdots f^{\lambda^{(r)}}$, 
where $\binom{n}{|\lambda^{(1)}|, \ldots, |\lambda^{(r)}|}$ is the multinomial coefficient (choose the entries for each $\lambda^{(i)}$) and $f^{\lambda^{(i)}}$ is the number of standard tableaux of shape $\lambda^{(i)}$ discussed in Section \ref{sec:tableaux}.

The column reading tableau $\C$ of shape $\rpar$ is the standard $r$-tableau obtained by entering $1,2,\ldots,n$ consecutively down the columns of $\rpar$, beginning with the leftmost component and filling in the columns from left to right. The {row reading tableau} $\R$ of shape $\rpar$ is the standard $r$-tableau obtained by entering $1,2\ldots,n$ consecutively across the rows of $\rpar$, beginning with the first row of $\lambda^{(r)}$ and filling in the rows from top to bottom, then continuing the process from right to left. If $\T \in \YT(\rpar)$ then the {word of $\T$} is the unique permutation $w_\T \in \Sym_n$ such that $w_\T(\C) = \T$, and an inversion in $\T$ is a pair $(i,j)$ such that $i > j$ and either (1) $i$ and $j$ are in the same component and $i$ is strictly south and strictly west of $j$, or (2) $i$ is in a component left of the component containing $j$. In this way, $(i,j)$ is an inversion in $\T$ if and only if $(i,j)$ is an inversion in $w_\T$.

\begin{example}\label{ex:r-skew}
If $\rpar=( (3,2) , \emptyset, (2), (2,1))$, then
$$
\C=\begin{array}{c}
\begin{tikzpicture}[scale=.37,line width=1.0pt] 
\path (-1,1) node[scale = 2.5] {$($};
\draw (0,0) rectangle (1,1); \path (0.5,0.5) node{$2$};
\draw (0,1) rectangle (1,2); \path (0.5,1.5) node {$1$};
\draw (1,0) rectangle (2,1); \path (1.5,0.5) node {$4$};
\draw (1,1) rectangle (2,2); \path (1.5,1.5) node {$3$};
\draw (2,1) rectangle (3,2); \path (2.5,1.5) node {$5$};
\path (3.5,0) node {,};

\draw (4.5,1) node[scale = 1.25] {$\emptyset$};

\path (5.5,0) node {,};

\draw (6,1) rectangle (7,2); \path (6.5,1.5) node {$6$};
\draw (7,1) rectangle (8,2); \path (7.5,1.5) node {$7$};

\path (8.5,0) node {,};

\draw (9,0) rectangle (10,1); \path (9.5,0.5) node {$9$};
\draw (9,1) rectangle (10,2); \path (9.5,1.5) node {$8$};
\draw (10,1) rectangle (11,2); \path (10.5,1.5) node {$10$};

\path (12,1) node[scale = 2.5] {$)$};
\end{tikzpicture}
\end{array},
\qquad
\R=\begin{array}{l}
\begin{tikzpicture}[scale=.37,line width=1.0pt] 
\path (-1,1) node[scale = 2.5] {$($};
\draw (0,0) rectangle (1,1); \path (0.5,0.5) node{$9$};
\draw (0,1) rectangle (1,2); \path (0.5,1.5) node {$6$};
\draw (1,0) rectangle (2,1); \path (1.5,0.5) node {$10$};
\draw (1,1) rectangle (2,2); \path (1.5,1.5) node {$7$};
\draw (2,1) rectangle (3,2); \path (2.5,1.5) node {$8$};

\path (3.5,0) node {,};

\draw (4.5,1) node[scale = 1.25] {$\emptyset$};

\path (5.5,0) node {,};

\draw (6,1) rectangle (7,2); \path (6.5,1.5) node {$4$};
\draw (7,1) rectangle (8,2); \path (7.5,1.5) node {$5$};

\path (8.5,0) node {,};

\draw (9,0) rectangle (10,1); \path (9.5,0.5) node {$3$};
\draw (9,1) rectangle (10,2); \path (9.5,1.5) node {$1$};
\draw (10,1) rectangle (11,2); \path (10.5,1.5) node {$2$};

\path (12,1) node[scale = 2.5] {$)$};
\end{tikzpicture}
\end{array},
$$
and
$$
\text{ if } \quad
\T=\begin{array}{l}
\begin{tikzpicture}[scale=.37,line width=1.0pt] 
\path (-1,1) node[scale = 2.5] {$($};
\draw (0,0) rectangle (1,1); \path (0.5,0.5) node{$4$};
\draw (0,1) rectangle (1,2); \path (0.5,1.5) node {$3$};
\draw (1,0) rectangle (2,1); \path (1.5,0.5) node {$8$};
\draw (1,1) rectangle (2,2); \path (1.5,1.5) node {$5$};
\draw (2,1) rectangle (3,2); \path (2.5,1.5) node {$7$};

\path (3.5,0) node {,};

\draw (4.5,1) node[scale = 1.5] {$\emptyset$};

\path (5.5,0) node {,};

\draw (6,1) rectangle (7,2); \path (6.5,1.5) node {$1$};
\draw (7,1) rectangle (8,2); \path (7.5,1.5) node {$6$};

\path (8.5,0) node {,};

\draw (9,0) rectangle (10,1); \path (9.5,0.5) node {$10$};
\draw (9,1) rectangle (10,2); \path (9.5,1.5) node {$2$};
\draw (10,1) rectangle (11,2); \path (10.5,1.5) node {$9$};

\path (12,1) node[scale = 2.5] {$)$};
\end{tikzpicture}
\end{array} \\
\hbox{then} \quad w_\T =
\left( 
\arraycolsep=3pt
\begin{array}{ccccccccccc}
1 & 2 & 3 & 4 & 5 & 6 & 7 & 8 & 9 & {10} \\
3 & 4 & 5 & 8 & 7 & 1 & 6 & 2 & {10} & 9
\end{array}
\right).
$$
Furthermore, the inversions in $\T$ are
\begin{align*}
\inv(\T) &= \{ (3,1), (3,2), (4,1), (4,2), (5,1), (5,2), (8,7), (8,1), (8,6), (8,2),  (7,1), (7,6), (7,2), (6,2), (10,9) \}. \\
\end{align*}
\end{example}

Let $\S, \T \in \SYT(\rpar)$. As in the case of standard tableaux of partition shape (Section \ref{subsec:BruhatTab}), we define Bruhat and weak order on $\SYT(\rpar)$ by defining $\S \le \T$ if and only if $w_\S \le w_\T$ and $\S \le_W \T$ if and only if $w_\S \le_W w_\T$.
The weak Bruhat graph $\B_n^{\rpar}$ is the Hasse diagram of weak order on $\SYT(\rpar)$, which corresponds to the interval $[w_\C,w_\R]$ in weak order on $\Sym_n$ under the map $\T \mapsto w_\T$.

\subsection{Cyclotomic Hecke algebras}\label{sec:cyclotomic}

Let $u_1, \ldots, u_r,q \in \CC^\ast$, $q$ not a root of unity. The \defn{cyclotomic Hecke algebra} $\HH_{r,n}=\HH_{r,n}(u_1, \ldots, u_r; q)$ of Ariki and Koike \cite{AK}  is the associative $\CC$-algebra with 1  generated by 
$\Th_i, 0 \le i \le n-1$, subject to the relations
\begin{equation}
\begin{array}{clllccclll}
\text{(a)} &  \Th_i \Th_j = \Th_j \Th_i, &&   |i - j| > 1, \\
\text{(b)} &    \Th_i \Th_{i+1} \Th_i = \Th_{i+1} \Th_i \Th_{i+1} , &&    1 \le i \le n-1, && & \text{(b$'$)} &   \Th_0 \Th_{1} \Th_0 \Th_{1}  = \Th_1 \Th_{0} \Th_1 \Th_{0} ,  \\
\text{(c)} &  \Th_i^2 = (q - q^{-1}) \Th_i + 1, && 1 \le i \le n-1, & && \text{(c$'$)} &    (\Th_0 - u_1)   \cdots  (\Th_0 - u_r) = 0.   
\end{array}
\end{equation}
If $\xi \in \CC$ is a primitive $r$-th root of unity, then upon letting $u_i \to \xi^{i-1}$ and $q \to 1$, the cyclotomic Hecke algebra $\HH_{r,n}$ becomes the group algebra $\CC\GG_{r,n}$, where $\GG_{r,n} = \ZZ_r \wr \Sym_n$ is the wreath product group discussed in Section \ref{sec:wreath}, and $\dim(\HH_{r,n}) = |\GG_{r,n}| = r^n n!$.

Let $u_1, \ldots, u_r, q$ be chosen so that $\HH_{r,n}$ is semisimple. By \cite{Ariki} this requires that $u_i u_j^{-1} \not\in \{1, q^2, \ldots, q^{2n}\}$ for $i \neq j$ and $[n]_q! \not = 0$, where $[n]_{q}! = [n]_q [n-1]_q \cdots [1]_q$ and $[k]_q = (q^k-q^{-1})/(q-q^{-1})$. The irreducible representations of $\HH_{r,n}$ and $\GG_{r,n}$ are indexed by $r$-partitions $\rpar = (\lambda^{(1)}, \ldots, \lambda^{(r)})$ with a total of $n$ boxes.

There is a surjective algebra homomorphism $\tilde \HH_n \to \HH_{r,n}$ from the affine Hecke algebra to the cyclotomic Hecke algebra given by sending $T_i \to T_i$ and $X^{\varepsilon_1} \to T_0$. Ram and Rammage \cite[Thm.~3.18]{RR} use this mapping to show that the irreducible modules for cyclotomic Hecke algebras $\HH_{r,n}$ are inherited from those for the affine Hecke algebra $\tilde{\HH}_n$. The $r$-partition $\rpar=(\lambda^{(1)}, \ldots, \lambda^{(r)})$ corresponds to the placed skew shape $(c,\lambda/\mu)$ which has $r$ connected components equal to $\lambda^{(1)}, \ldots, \lambda^{(r)}$ where $\lambda^{(k)}$ is placed on page  $\beta_k$, satisfying $q^{2\beta_k} = u_k$. Under this identification, for a box  $b \in \rpar$ the content function satisfies
\begin{equation}\label{AK-content}
q^{2c(b)} = u_k q^{2(y-x)}, \qquad \hbox{where $b$ is in position $(x,y)$ of the $k$-th component $\lambda^{(k)}$ of $\rpar$}.
\end{equation}

For an  $r$-partition $\rpar = (\lambda^{(1)}, \ldots, \lambda^{(r)})$, the irreducible $\HH_{r,n}$-module indexed by $\rpar$ is the vector space
$
\HH_{r,n}^{\rpar} = \CC\text{-span}\left\{ \v_\T \mid \T \in \SYT(\rpar) \right\},
$
with $\HH_{r,n}$-action given by the formulas 
\begin{equation}
\begin{aligned}  \label{AK-seminormal}
\Th_0 \v_\T & = u_i \v_\T, \qquad \text{if $1 \in \T^{(i)}$},  \\
\Th_i \v_\T & = \tilde{a}_i(\T) \v_\T +  \left(q^{-1} + \tilde{a}_i(\T)  \right) \v_{s_i(\T)}, 
\end{aligned}
\end{equation}
where $ \v_{s_i(\T)} = 0$ if $s_i(\T)$ is nonstandard, $\T^{(i)}$ is the component of $\T$ containing $i$, and the coefficient $\tilde{a}_i(\T)$ defined in \eqref{def:q-diagonal} simplifies as
\begin{equation}\label{def:cyclotomic-weights}
\tilde{a}_i(\T) = 
\frac{q-q^{-1}} {1 -  q^{2 (c (\T(i)) - c(\T(i+1))}}
= \frac{q-q^{-1}} {1 - \frac{u_{k_i}}{u_{k_{i+1}}} q^{2 (ct(\T(i)) - ct(\T(i+1)))}},
\end{equation}
such that $i$ is in position $(x_i,y_i)$ of $\lambda^{(k_i)}$ with $ct(\T(i)) = y_i-x_i$ and $i+1$ is in position $(x_{i+1},y_{i+1})$ of $\lambda^{(k_{i+1})}$ with $ct(\T(i+1)) = y_{i+1}-x_{i+1}$.

The basis $\tilde{\mathcal{V}}^{\rpar}=\{ \v_\T \mid \T \in \SYT(\rpar)\}$ is the {seminormal basis} of $\HH_{r,n}^{\rpar}$ and the representation in \eqref{AK-seminormal} is due to \cite{AK}. Ram defines the {natural basis} $\tilde{\mathcal{N}}^{\rpar}=\{ \n_\T \mid \T \in \SYT(\rpar) \}$ of  $\HH_{r,n}^{\rpar}$ 
 from the seminormal basis as:
\begin{equation}\label{AK-natural-rep}
\n_\T := \Th_{w_\T} \v_\C, \qquad \text{where $\C$ is the column reading tableau and $w_\T \in \Sym_n$ such that $w_\T(\C) = \T$}.
\end{equation}
The following change-of-basis theorem follows from \eqref{AK-natural-rep} and \eqref{AK-seminormal}.

\begin{thm}\label{thm:AK} For an $r$-partition $\rpar$ with $n$ boxes and the content function $c$ determined in \eqref{AK-content}, the change-of-basis for $\HH_{r,n}^{\rpar}$ expressing the natural basis $\tilde{\mathcal{N}}^{\rpar}$ in terms of the seminormal basis $\tilde{\mathcal{V}}^{\rpar}$ satisfies
$$
\n_\T = \sum_{\S \le \T} \tilde{A}_{\S,\T} \, \v_\S \quad\hbox{with coefficients}\quad \tilde{A}_{\S,\T} = \sum_{\omega \subseteq \pi} \tilde{\wt}_{\pi}(\omega),
$$
where the first sum is over standard tableaux $\S \le \T$ in Bruhat order;
the coefficient $\tilde{A}_{\S,\T}$ is defined using any path $\pi$  from  $\C$ to $\T$ in $\B_n^{\rpar}$;
 the second sum is over all subpaths $\omega \subseteq\pi$ which terminate at $\S$;
and the weight $\tilde{\wt}_\pi(\omega)$ of the subpath $\omega\subseteq \pi$ is  given by \eqref{def:q-weight} but with the coefficients $\tilde a_i(\T)$ as in \eqref{def:cyclotomic-weights}.
\end{thm}

\begin{proof} The result follows from Theorem \ref{thm:affine} using the surjective algebra homomorphism $\tilde{\HH}_n \to \HH_{r,n}$ of Ram and Rammage \cite[(2.2)]{RR}. Alternatively, by comparing  \eqref{AK-seminormal} with  \eqref{eq:seminormal} we see that $T_i$ acts on $\v_\T$ the same as $s_i$ acts on $\v_\T$ except that the diagonal entry $a_i(\T)$ is replaced by $\tilde a_i(\T)$ and the off-diagonal entry $1 + a_i(\T)$ is replaced by $q^{-1} + \tilde{a}_i(\T)$.  Moreover if $w_\T = s_{j_1} s_{j_2} \ldots s_{j_\ell}$ is any reduced expression for the word of $\T$, then $T_{w_\T} = T_{j_1} T_{j_2} \cdots T_{j_\ell}$.  Thus, the natural basis vector $\n_\T = T_{w_\T} \v_c$ for the cyclotomic Hecke algebra is the same as $\n_\T = w_\T \v_c$, for the symmetric group, except that the weights $a_{j_i}(\T)$ are replaced by those $\tilde a_{j_i}(\T)$ in \ref{def:cyclotomic-weights}
 and $1 + a_{j_i}(\T)$ is replaced by $q^{-1} + \tilde{a}_{j_i}(\T)$. The proof is then identical to the proof of Theorem \ref{thm:main}.
\end{proof}

\subsection{Iwahori-Hecke algebras of type A and B}

The \defn{Iwahori-Hecke algebra} $\HH_n(q) = \HH_{1,n}$ of type A is the special case of cyclotomic Hecke algebra when $r = 1$, $u_1 = 1$, and $T_0 = 1$. Its irreducible modules $\HH_n^\lambda$ are labeled by partitions $\lambda \vdash n$, and Theorem \ref{thm:AK} gives the change-of-basis between the seminormal and natural bases.
In this case, with only one component to $\lambda$, the seminormal action reduces to the one defined in \eqref{AK-seminormal} with the weight $\tilde a_i(\T)$ given by
\begin{equation}
\tilde{a}_i(\T) = 
\frac{q-q^{-1}} {1 -  q^{2 (ct (\T(i)) - ct(\T(i+1))}}.
\end{equation}
This is precisely the seminormal action of Hoefsmit \cite{Hoefsmit}, and the corresponding natural representation, defined using \eqref{AK-natural-rep}, was given by Ram \cite{Ram-skew}. In the $q \to 1$ limit, the coefficients $\tilde{a}_i(\T)$ become the coefficients of the symmetric group; that is, $\lim_{q \to 1} \tilde{a}_i(\T)  = a_i(\T)$ defined in \eqref{def:symmetric-diagonal}. Thus, if $\mathcal{A}_\lambda(q)$ is the transition matrix between the natural and seminormal bases of $\HH_n^\lambda$, then $\lim_{q \to 1} \mathcal{A}_\lambda(q) = \mathcal{A}_\lambda$, the transition matrix for the symmetric group module $\Sym_n^\lambda$, as seen in Example \ref{eg:Hecke2Symmetric}.

\begin{example} \label{eg:Hecke2Symmetric} If $\lambda = (3,2)$ then the transition matrix $\mathcal{A}_\lambda(q) = (A_{\S,\T})_{\S,\T \in \SYT(\lambda)}$ between the natural and seminormal bases of $\HH_n^\lambda$ is given by
$$
\Yboxdim{7pt}
\mathcal{A}_{\lambda}(q) =
\begin{blockarray}{cccccc}
\phantom{\tiny\young(1,1,1)} & \n_{\tiny\young(135,24)}& \n_{\tiny\young(125,34)}&\n_{\tiny\young(134,25)}&\n_{\tiny\young(124,35)}&\n_{\tiny\young(123,45)} \\
\begin{block}{c(ccccc)}
\v_{\tiny\young(135,24)} & 1 & \frac{q^3}{q^2+1} & \frac{q^3}{q^2+1}  & \frac{q^6}{(q^2+1)^2} & -\frac{q^5}{(q^2+1)^2} \\
\v_{\tiny\young(125,34)} & 0 & \frac{q^3 + q + q^{-1}}{q^2 + 1} & 0 & \frac{q^6 + q^4 + q^2}{(q^2 + 1)^2} & \frac{q^7 + q^5 + q^3}{(q^2+1)^2} \\
\v_{\tiny\young(134,25)} & 0 & 0 & \frac{q^3 + q + q^{-1}}{q^2 + 1} & \frac{q^6 + q^4 + q^2}{(q^2 + 1)^2} & \frac{q^7 + q^5 + q^3}{(q^2+1)^2} \\
\v_{\tiny\young(124,35)} & 0 & 0 & 0 & \left(\frac{q^3 + q + q^{-1}}{q^2 + 1}\right)^2 & \frac{q^5}{q^2 + q + 1} \cdot \left(\frac{q^3 + q + q^{-1}}{q^2 + 1}\right)^2 \\
\v_{\tiny\young(123,45)} & 0 & 0 & 0 & 0 & \frac{q^5 + q + 1 + q^{-1}}{q^2 + q + 1} \cdot \left(\frac{q^3 + q + q^{-1}}{q^2 + 1}\right)^2 \\
\end{block}
\end{blockarray}\ ,
$$
and one can verify that $\lim_{q \to 1} \mathcal{A}_\lambda(q) = \mathcal{A}_\lambda$, which is shown in Example \ref{chgA32}. 
\end{example}

The {Iwahori-Hecke algebra} $\HH \mathbb{B}_n(q) = \HH_{2,n}$ of type B is the special case of cyclotomic Hecke algebra with $r = 2$ and $u_1 = u_2^{-1}$.  Its irreducible modules are labeled by pairs of partitions $\rpar = (\lambda^{(1)}, \lambda^{(2)})$ with a total of $n$ boxes, and Theorem  \ref{thm:AK} gives the change basis between the seminormal and natural bases.  As in type A, the seminormal and natural bases of $\HH\mathbb{B}_n^{\rpar}$ were defined by Hoefsmit \cite{Hoefsmit} and Ram \cite{Ram-skew}, respectively.

\subsection{The complex reflection group $\GG_{r,n}$}
\label{sec:wreath}

The group  $\GG_{r,n} = \ZZ_r \wr \Sym_n$ is the wreath product of the symmetric group $\Sym_n$ and the finite cyclic group $\ZZ_r$ of order $r$.  It has order $r^n n!$ and is generated by
$s_i, 0 \le i \le n-1$, subject to the relations
\begin{equation}
\begin{array}{clllccclll}
\text{(a)} &  s_i s_j = s_j s_i, &&   |i - j| > 1, \\
\text{(b)} &   s_i s_{i+1} s_i = s_{i+1} s_i s_{i+1} , &&  1 \le i \le n-2, &&& \text{(b$'$)} &   s_0 s_{1} s_0 s_{1}  = s_1 s_{0} s_1 s_{0},   \\
\text{(c)} & s_i^2 = 1, && 1 \le i \le n-1, &&& \text{(c$'$)} &   s_0^r = 1.  \\
\end{array}
\end{equation}
This group is denoted by  $\GG_{r,n} = \GG_{r,1,n}$ in the Shepard and Todd classification of complex reflection groups, and it contains the symmetric group $\Sym_n \subseteq \GG_{r,n}$ as the subgroup generated by $s_1, \ldots, s_{n-1}$. Upon letting $q \to 1$ and identifying $T_i$ with $s_i$ for $0 \le i \le n-1$, the cyclotomic Hecke algebra $\HH_{r,n}$ of Section \ref{sec:cyclotomic} specializes to the group algebra $\CC\GG_{r,n}.$

For an  $r$-partition $\rpar = (\lambda^{(1)}, \ldots, \lambda^{(r)})$, define the irreducible $\GG_{r,n}$-module indexed by $\rpar$ as the vector space
$
\GG_{r,n}^{\rpar} = \CC\text{-span}\left\{ \v_\T \mid \T \in \SYT(\rpar) \right\},
$
with $\GG_{r,n}$-action given by the formulas 
\begin{equation}\label{seminormal_action_g_rn}
\begin{array}{lcl}
s_0 \v_\T  = \xi^{i-1} \v_\T, &\quad& \text{if $1 \in \T^{(i)}$},   \\
s_i \v_\T  = \v_{s_i(\T)}, &\phantom{\Big\vert}& \text{if $i,i+1$ are in different components of $\T$}, \\
s_i \v_\T  = {a}_i(\T) \v_\T +  \left(1 + {a}_i(\T)  \right) \v_{s_i(\T)}, && \text{if $i,i+1$ are in the same component of $\T$}, 
\end{array}
\end{equation}
where $ \v_{s_i(\T)} = 0$ if $s_i(\T)$ is nonstandard, and ${a}_i(\T)$ is the same coefficient as defined for the symmetric group in \eqref{def:symmetric-diagonal}.
This representation is obtained from \eqref{AK-seminormal} by letting $q \to 1$,  $u_i \to \xi^{i-1}$, where $\xi$ is a complex $r$th root of unity, and identifying $T_i$ with $s_i$ for $0 \le i \le n-1$.  The basis $\calV^{\rpar} = \{ \v_\T \mid \T \in \SYT(\rpar)\}$ is the seminormal basis of $\GG_{r,n}^{\rpar}$.

The natural basis $\calN^{\rpar} = \{ \n_\T \mid \T \in \SYT(\rpar) \}$ of $\GG_{r,n}^\rpar$ is defined from the seminormal basis by
\begin{equation}\label{wreath-natural-rep}
\n_\T := w_\T \v_\C, \qquad\text{where $w_\T \in \Sym_n$ such that $w_\T(\C) =\T$ (i.e., $w_\T$ is the word of $\T$)},
\end{equation}
and $\C$ is the column reading tableau of shape $\rpar$.  The change-of-basis coefficients from the natural basis to the seminormal basis are given by sending $q  \to 1$ and $u_i \to \xi^{i-1}$ in Theorem \ref{thm:AK}.

Upon specializing to the group $\GG_{r,n}$, the transition matrix takes an especially nice form that can be understood in terms of the alphabet of entries in each component of the underlying tableaux.
Let $\rpar = (\lambda^{(1)}, \ldots, \lambda^{(r)})$ be an $r$-partition with $n$ boxes, and let $n_i = |\lambda^{(i)}|$.
A \defn{$\rpar$-alphabet} is an ordered set partition of $\{1, \ldots, n\}$ of the form
$$
\mathfrak{A} = \left\{ a_1^{(1)}, \ldots, a_{n_1}^{(1)} \mid  a_1^{(2)}, \ldots, a_{n_2}^{(2)} \mid \cdots \mid a_1^{(r)},\ldots, a_{n_r}^{(r)}\right \},
\quad\hbox{where  $a_1^{(i)} <  a_2^{(i)}  < \cdots < a_{n_i}^{(i)}$ for each $1 \le i \le r$}.
$$
The corresponding permutation (in two-line notation),
$$
\beta = 
\left(
\begin{array}{ccc|ccc|c|ccc}
1 &\cdots & n_1 & n_1 + 1 & \cdots & n_1 + n_2 & \cdots & n - n_r + 1 & \cdots & n \\
a_1^{(1)} & \cdots & a_{n_1}^{(1)} & a_1^{(2)} & \cdots & a_{n_2}^{(2)} & \cdots & a_1^{(r)} & \cdots & a_{n_r}^{(r)}  \\
\end{array}
\right) \in \Sym_n,
$$
is the \defn{alphabetizer} corresponding to $\mathfrak{A}$. The  \defn{standard alphabet} of $\rpar$ is 
$$
\mathfrak{A}_0 = \left\{ 1, \ldots, n_1 \mid n_1 + 1, \ldots, n_1 + n_2 \mid \cdots \mid n_1+\cdots + n_{r-1}+1, \ldots, n \right\},
$$
whose alphabetizer is the identity permutation.  If $\T = (\T^{(1)}, \ldots, \T^{(r)}) \in \SYT(\rpar)$, then the $\rpar$-alphabet of $\T$ is  given by partitioning $\{1, \ldots, n\}$ according to the entries of each $\T^{(i)}$. For instance,  if $\T$ is the standard $4$-tableau in Example \ref{ex:r-skew} then
$$
\mathfrak{A} = \left\{ 3,4,5,7,8 \mid   \mid  1,6\mid 2,9,10 \right\}
\qquad \hbox{and}\qquad
\beta = \left(\begin{array} {ccccc||cc|ccc}
1&2&3&4&5&6&7&8&9&10\\
3&4&5&7&8 &1&6 &2&9&10
\end{array}\right).
$$
The alphabetizer is a minimal-length representative of the left coset containing $\beta$ corresponding to the Young subgroup $\Sym_{n_1} \times \Sym_{n_2} \times \cdots \times \Sym_{n_r} \subseteq \Sym_n$.

\begin{lemma}\label{lemma:alphabetizer}
If $\rpar$ is an $r$-partition of $n$, $\T \in \SYT(\rpar)$ has the standard alphabet, and $\beta \in \Sym_n$ is an alphabetizer, then $\beta \v_\T = \v_{\beta(\T)}$.
\end{lemma}

\begin{proof}
We proceed by induction on the length of $\beta$. If $\ell(\beta) = 0$,  then $\beta$ is the identity and the result is immediate. If $\ell(\beta) \ge 1$, then $\inv(\beta)$ is nonempty, so pick an inversion $(i+1,i)$ such that $i+1$ lies in a component to the left of the component containing $i$. Such an inversion must exist, for if the only inversions in $\beta(\T)$ occur within a single component, then $\beta(\T)$ has the standard alphabet, which is impossible since $\beta \neq e$ is an alphabetizer and the alphabetizer of the standard alphabet is the identity.  It  follows  that $\ell(s_i\beta) < \ell(\beta)$ since $s_i\beta(\T)$ has fewer inversions than $\beta(\T)$. By the Exchange Condition (see, for example, \cite[(1.7)]{humphreys1990reflection}) $\beta = s_i \beta'$ for some $\beta' \in \Sym_n$ with $\ell(\beta') < \ell(\beta)$. Furthermore, we claim that $\beta'$ is also an alphabetizer. We have
$$
\beta = 
\left(
\begin{array}{c|ccccc|c|ccccc|c}
\cdots & \cdots & p-1 & p & p+1 & \cdots & \cdots & \cdots & q-1 & q & q+1 & \cdots & \cdots \\
\cdots & \cdots & a_{c-1}^{(x)} & i+1 & a_{c+1}^{(x)} & \cdots & \cdots & \cdots & a_{d-1}^{(y)} & i & a_{d+1}^{(y)} & \cdots & \cdots 
\end{array}
\right)
$$
and
$$
\beta' = 
\left(
\begin{array}{c|ccccc|c|ccccc|c}
\cdots & \cdots & p-1 & p & p+1 & \cdots & \cdots & \cdots & q-1 & q & q+1 & \cdots & \cdots \\
\cdots & \cdots & a_{c-1}^{(x)} & i & a_{c+1}^{(x)} & \cdots & \cdots & \cdots & a_{d-1}^{(y)} & i+1 & a_{d+1}^{(y)} & \cdots & \cdots 
\end{array}
\right).
$$
Since $\beta$ is an alphabetizer we have $a_{c-1}^{(x)} < i+1 < a_{c+1}^{(x)}$ and $a_{d-1}^{(y)} < i < a_{d+1}^{(y)}$. But as there are no numbers between $i$ and $i+1$ we  have
$a_{c-1}^{(x)} < i < a_{c+1}^{(x)}$ and $a_{d-1}^{(y)} < i+1 < a_{d+1}^{(y)}$. Thus $\beta'$ is also an alphabetizer since all other entries of $\beta$ and $\beta'$ agree.
As $\ell(\beta') < \ell(\beta)$, we have by induction $\beta' \v_\T = \v_{\beta'(\T)}$, and therefore
$$
\beta \v_\T = s_i \beta' \v_\T = s_i \v_{\beta'(\T)} = \v_{s_i\beta'(\T)} = \v_{\beta(\T)},
$$
where the second-to-last equality holds by \eqref{seminormal_action_g_rn} since $i$ and $i+1$ are in different components of $\beta'(\T)$.
\end{proof}

\newpage

\begin{thm}\label{thm:tensors}
For an $r$-partition $\rpar = (\lambda^{(1)},\ldots,\lambda^{(r)})$ with $n$ boxes, the change-of-basis for $\GG_{r,n}^{\rpar}$ expressing the natural basis $\calN^\rpar$ in terms of the seminormal basis $\calV^\rpar$ satisfies
$$
\n_\T = \sum_{\S \le \T} \underline{A}_{\S,\T} \v_\S \quad \text{with coefficients} \quad \underline{A}_{\S,\T} = \prod_{i=1}^r A_{\S^{(i)},\T^{(i)}}
$$
where the first sum is over standard tableaux $\S \le \T$ in Bruhat order which have the same alphabet as $\T$ and  $A_{\S^{(i)},\T^{(i)}}$ is the coefficient from the symmetric group $\Sym_n$ defined in Theorem \ref{thm:main}. 

\end{thm}

\begin{proof}
Let $\T \in \SYT(\underline{\lambda})$ have the standard alphabet and fix a component $\lambda^{(i)}$ of $\underline{\lambda}$. The entries in $\T^{(i)}$ are $\mathfrak{A}^{(i)} = \{ m_i + 1, m_i + 2, \ldots, m_i + n_i \}$, where $n_i = |\lambda^{(i)}|$ and $m_i = n_1 + \cdots + n_{i-1}$. We have a natural bijection $\mathfrak{A}^{(i)} \to \{ 1, \ldots, n_i \}$ given by $m_i+ j \mapsto j$, and  it follows from Theorem \ref{thm:main} that
$$
\n_{(\C^{(1)}, \ldots, \C^{(i-1)}, \T^{(i)}, \C^{(i+1)}, \ldots, \C^{(r)})} = \sum_{\S^{(i)} \le \T^{(i)}} A_{\S^{(i)},\T^{(i)}} \v_{(\C^{(1)}, \ldots, \C^{(i-1)},\S^{(i)}, \C^{(i+1)},\ldots,\C^{(r)})}.
$$
where the sum is over all $\S^{(i)} \le \T^{(i)}$ in Bruhat order which have the alphabet $\mathfrak{A}^{(i)}$ and $\C = (\C^{(1)}, \ldots, \C^{(r)})$ is the column reading tableau of shape $\underline{\lambda}$. By applying the same argument in each component, we obtain
$$
\n_{(\T^{(1)}, \ldots, \T^{(r)})} = \sum_{\S \le \T} \left( \prod_{i=1}^r A_{\S^{(i)},\T^{(i)}} \right) \v_{(\S^{(1)}, \ldots,,\S^{(r)})},
$$
where by construction each $\S \le \T$ has the same alphabet as $\T$.

Now let  $\T \in \SYT(\rpar)$ have partitioned alphabet $\mathfrak{A}$ (not necessarily standard), and let $\beta \in \Sym_n$ be its alphabetizer. Then $\T':=\beta^{-1}(\T)$ is a standard tableau with the standard alphabet and
$$
\n_{\T'} = \sum_{\S' \le \T'} \underline{A}_{\S',\T'} \v_{\S'} \quad \text{with coefficients} \quad  \underline{A}_{\S',\T'} = \prod_{i=1}^r A_{\S'^{(i)},\T'^{(i)}},
$$
where the sum is over $\S' \le \T'$ with the standard alphabet. By Lemma \ref{lemma:alphabetizer} we have $\beta \v_{\S'} = \v_{\beta(\S')}$, so
\begin{equation*}
\beta \n_{\T'} = \sum_{\S' \le \T'}  \underline{A}_{\S',\T'} (\beta \v_{\S'}) 
 \sum_{\S' \le \T'}  \underline{A}_{\S',\T'} \v_{\beta(\S')} = n_{\beta(\T')}.
\end{equation*}
Let $\S = \beta(\S')$ for all $\S' \le \T'$ with the standard alphabet. Then $\S \le \T$ if and only if $\S' \le \T'$ since $w_\S = \beta w_{\S'}$ and $w_\T = \beta w_{\T'}$, so we obtain
$
\n_\T = \sum_{\S \le \T} \underline{A}_{\S',\T'} \v_\S,
$
where the sum is over all $\S \le \T$ with partitioned alphabet $\mathfrak{A}$.
\end{proof}

\begin{cor}\label{cor:decomp}
For an $r$-partition $\rpar = (\lambda^{(1)},\ldots,\lambda^{(r)})$ of $n$ with $n_i = |\lambda^{(i)}|$, the transition matrix $\calA_{\rpar} = (A_{\S,\T})_{\S,\T \in \SYT(\rpar)}$ is of the form
$$
\calA_{\rpar} = \left( \calA_{\lambda^{(1)}} \otimes \cdots \otimes \calA_{\lambda^{(r)}} \right)^{\oplus \binom{n}{n_1, \ldots, n_r}}.
$$
That is, $\calA_{\rpar}$ is the direct sum of multinomial $\binom{n}{n_1, \ldots, n_r}$ identical copies (one for each choice of alphabet) of the tensor product $\calA_{\lambda^{(1)}} \otimes \cdots \otimes \calA_{\lambda^{(r)}}$ such that each $\calA_{\lambda^{(i)}}$ is the symmetric group $\Sym_{n_i}$ transition matrix as described  in Theorem \ref{thm:main}. In particular, the matrix blocks are independent of the choice of alphabet.\end{cor}


\begin{example}\label{eg:tensor}
If $\rpar = ((2,1),(3,1))$ then the change-of basis matrix $\calA_\rpar$ is given as the direct sum over $f^\rpar = \binom{7}{3, 4} \cdot f^{(2,1)} \cdot f^{(3,1)} = 35 \cdot 2 \cdot 3 = 210$ copies of the tensor product of $\calA_{(2,1)}$ and $\calA_{(3,1)}$. The first block of this matrix, corresponding to the $r$-tableaux with the standard alphabet $\mathfrak{A}_0$, looks like this:
\setlength{\tabcolsep}{0pt}
\renewcommand{\arraystretch}{1.5}
\Yboxdim{6.0pt}
\begin{align*}
\calA_\rpar\! \mid_{\mathfrak{A}_0} &=
\begin{blockarray}{ccccccc}
~ & \n_{( \tiny\young(13,2),\tiny\young(467,5) )} & \n_{( \tiny\young(13,2),\tiny\young(457,6) )} & \n_{( \tiny\young(13,2),\tiny\young(456,7) )} & \n_{( \tiny\young(12,3),\tiny\young(467,5) )} & \n_{( \tiny\young(12,3),\tiny\young(457,6) )} & \n_{( \tiny\young(12,3),\tiny\young(456,7) )} \\
\begin{block}{c(cccccc)}
\v_{( \tiny\young(13,2),\tiny\young(467,5) )}  & 1 & \frac{1}{2} & \frac{1}{2} &  \frac{1}{2} & \frac{1}{4} & \frac{1}{4} \\
\v_{( \tiny\young(13,2),\tiny\young(457,6) )} & 0 & \frac{3}{2} & \frac{1}{2} &  0 & \frac{3}{4} & \frac{1}{4}\\
\v_{( \tiny\young(13,2),\tiny\young(456,7) )} & 0 & 0 & 2 &  0 & 0 & 1 \\
 \v_{( \tiny\young(12,3),\tiny\young(467,5) )} &  0 & 0 & 0 &  \frac{3}{2} & \frac{3}{4} & \frac{3}{4} \\
 \v_{( \tiny\young(12,3),\tiny\young(457,6) )} & 0 & 0 & 0 &  0 & \frac{9}{4} & \frac{3}{4} \\
 \v_{( \tiny\young(12,3),\tiny\young(456,7) )} & 0 & 0 & 0 &  0 & 0 & 3 \\
 \end{block}
\end{blockarray} \\
&= \begin{blockarray}{ccc}
~ & \n_{\tiny\young(13,2)} & \n_{\tiny\young(12,3)} \\
\begin{block}{c(cc)}
\v_{\tiny\young(13,2)} & 1 & \frac{1}{2} \\
\v_{\tiny\young(12,3)} & 0 & \frac{3}{2} \\
\end{block}
\end{blockarray} \quad \otimes \quad \begin{blockarray}{cccc}
~ & \n_{\tiny\young(134,2)} & \n_{\tiny\young(124,3)} & \n_{\tiny\young(123,4)} \\
\begin{block}{c(ccc)}
\v_{\tiny\young(134,2)} & 1 & \frac{1}{2} & \frac{1}{2} \\
\v_{\tiny\young(124,3)} & 0 & \frac{3}{2} & \frac{1}{2} \\
\v_{\tiny\young(123,4)} & 0 & 0 & 2 \\
\end{block}
\end{blockarray}\quad = \quad \calA_{(2,1)} \otimes \calA_{(3,1)}.
\end{align*}
\end{example}

\section{Transition Between Orthogonal and Seminormal Representations}
\label{sec:orthogonal}

In this section we define a new basis of  $\tilde{\HH}_n^{(c,\lambda/\mu)}$ that is derived from the seminormal basis by a diagonal transition matrix. In Theorem \ref{thm:orthogonal-transition} we give the action of the generators on this basis that holds for all of the algebraic structures of this paper. When the change-of-basis formula \eqref{eq:Hecke-orthogonal} is specialized to $\Sym_n$ by letting $q \to 1$ and $\tilde a_{i,j} \to a_{i,j}$, the action the same as that of the symmetric group on Young's orthogonal basis \eqref{eq:orthogonal}. This diagonal relationship between the seminormal and orthogonal representations for partition shapes was known to Rutherford \cite{Ru} where it is given in a slightly different form. 

For a placed skew shape $(c, \lambda / \mu)$ with $n$ boxes, define a basis $\tilde{\mathcal{O}}^{(c,\lambda/\mu)} = \{ \u_\T \mid \T \in \SYT(\lambda / \mu) \}$ for the $\tilde{\HH}_n$-module $\tilde{\HH}_n^{(c,\lambda/\mu)}$ given by the formula
\begin{equation}\label{eq:Hecke-orthogonal}
\u_\T = D_{\T,\T} \v_\T, \quad \text{where} \quad D_{\T,\T} = \prod_{(i,j) \in \inv(\T)} \frac{q^{-1} + \tilde{a}_{i,j}(\T)}{\sqrt{q^{-2} - \left(\tilde{a}_{i,j}(\T)\right)^2}},
\end{equation}
where $\tilde{a}_{i,j}(\T)$ is defined in \eqref{cor:affinediagonal}.

\begin{thm}\label{thm:orthogonal-transition}
For a placed skew shape $(c, \lambda / \mu)$ with $n$ boxes, the $\tilde{\HH}_n$-action on the basis $\tilde{\mathcal{O}}^{(c,\lambda/\mu)}$ satisfies
\begin{equation*}
X^{\varepsilon_i} \u_\T  = q^{2 c(\T(i))} \u_\T \qquad \hbox{and} \qquad
\Th_i \u_\T  = \tilde{a}_i(\T) \u_\T +  \sqrt{q^{-2} - \tilde{a}_i(\T)^2} \u_{s_i(\T)},
\end{equation*}
where $ \u_{s_i(\T)} = 0$ if $s_i(\T)$ is nonstandard.
\end{thm}

\begin{proof} The first equality is easily verified, so it suffices to prove the second. For any generator $T_i$, $1 \le i \le n-1$, we have
\begin{align*}
T_i \u_\T  
= D_{\T,\T} T_i \v_\T  
& =  D_{\T,\T}\left( \tilde{a}_{i,i+1}(\T) \v_\T + (q^{-1}+\tilde{a}_{i,i+1}(\T))  \v_{s_i(\T)} \right) \\
& =  \tilde{a}_{i,i+1}(\T) D_{\T,\T} \v_\T + (q^{-1}+\tilde{a}_{i,i+1}(\T)) D_{\T,\T}  \v_{s_i(\T)} \\
& =  \tilde{a}_{i,i+1}(\T)  \u_\T + (q^{-1}+\tilde{a}_{i,i+1}(\T))\frac{D_{\T,\T}}{D_{s_i(\T),s_i(\T)}} \u_{s_i(\T)}.
\end{align*}
If $s_i(\T)$ is nonstandard, we are done. If $s_i(\T)$ is standard, then it suffices to prove that
\begin{equation}\label{eq:diagonalconversion}
\frac{q^{-1} + \tilde{a}_{i,i+1}(\T)}{\sqrt{q^{-2} - \left(\tilde{a}_{i,i+1}(\T)\right)^2}} D_{\T,\T} = D_{s_i(\T),s_i(\T)}.
\end{equation}
The inversions in $\T$ and $s_i(\T)$ that do not involve $i$ or $i+1$ are exactly the same. The inversions that involve $i$ or $i+1$ and another entry $j$ are in bijection by swapping $i$ and $i+1$, and the differences in content do not change, since they depend only on the positions of $i, i+1$ and $j$. The only inversion that changes is between $i$ and $i+1$.  If $(i+1,i) \not\in\inv(\T)$, then $(i+1,i) \in \inv(s_i(\T))$. In this case, $\tilde{a}_{i,i+1}(\T) = \tilde{a}_{i+1,i}(s_i(\T))$ and the leftmost term in  \eqref{eq:diagonalconversion} is exactly what is needed to be multiplied by $D_{\T,\T}$ to get to $D_{s_i(\T), s_i(\T)}$.

Conversely, if $(i+1,i) \in\inv(\T)$, then $(i+1,i) \not\in \inv(s_i(\T))$.  In this case, $(q^{-1} + \tilde{a}_{i+1,i}(\T))/\sqrt{q^{-2} - \left(\tilde{a}_{i+1,i}(\T)\right)^2}$ is the factor of $D_{\T,\T}$ corresponding to $(i+1,i) \in\inv(\T)$, so that
$$
D_{\T,\T} = \frac{q^{-1} + \tilde{a}_{i+1,i}(\T)}{\sqrt{q^{-2} - \left(\tilde{a}_{i+1,i}(\T)\right)^2}} D_{s_i(\T),s_i(\T)}.
$$
When we multiply this factor by the leftmost term in \eqref{eq:diagonalconversion} it cancels, as seen here:
\begin{align*}
\frac{q^{-1} + \tilde{a}_{i,i+1}(\T)}{\sqrt{q^{-2} - \left(\tilde{a}_{i,i+1}(\T)\right)^2}} \frac{q^{-1} + \tilde{a}_{i+1,i}(\T)}{\sqrt{q^{-2} - \left(\tilde{a}_{i+1,i}(\T)\right)^2}} & =
\frac{q^{-1} + \tilde{a}_{i,i+1}(\T)}{\sqrt{q^{-2} - \left(\tilde{a}_{i,i+1}(\T)\right)^2}} \frac{q^{-1} - \tilde{a}_{i,i+1}(\T)}{\sqrt{q^{-2} - \left(-\tilde{a}_{i,i+1}(\T)\right)^2}} \\
&= 
\frac{q^{-2} - (\tilde{a}_{i,i+1}(\T))^2}{q^{-2} - \left(\tilde{a}_{i,i+1}(\T)\right)^2} = 
1.
\end{align*}
\end{proof}

\section{Examples}\label{sec:examples}

\subsection{Symmetric group transition matrices}\label{sec:Sn-matrices} 

In this section we give all of the nontrivial transition matrices for irreducible $\Sym_n$-modules for $n = 3,4,5$.

\begin{multicols}{2}
\underline{$\lambda = (2,1)$}
\begin{equation*}
\Yboxdim{10pt}
\begin{array}{c}
\begin{tikzpicture}[line width=.5pt,xscale=0.22,yscale=0.22]
\path (20,30) node (T1) {$\young(12,3)$};
\path (20,20) node (T2) {$\young(13,2)$};

\path (T1) edge[black,thick,left] node{$s_2$} (T2);
\end{tikzpicture}
 \end{array}\hskip.25in
 \Yboxdim{7pt}
\begin{blockarray}{ccc}
\phantom{\tiny\young(1,1,1)} & {\n_{\tiny\young(13,2)}} & {\n_{\tiny\young(12,3)}} \\
\begin{block}{c(cc)}
{\v_{\tiny\young(13,2)}} & 1 & \frac{1}{2} \\
{\v_{\tiny\young(12,3)}} & 0 & \frac{3}{2} \\
\end{block}
\end{blockarray}
\end{equation*}

\underline{$\lambda = (3,1)$}
\begin{equation*}
\Yboxdim{10pt}
\begin{array}{c}
\begin{tikzpicture}[line width=.5pt,xscale=0.22,yscale=0.22]
\path (20,30) node (T1) {$\young(134,2)$};
\path (20,20) node (T2) {$\young(124,3)$};
\path (20,10) node (T3) {$\young(123,4)$};

\path (T1) edge[black,thick,left] node{$s_2$} (T2);
\path (T2) edge[black,thick,left] node{$s_3$} (T3);
\end{tikzpicture}
 \end{array}\hskip.25in
  \Yboxdim{7pt}
\begin{blockarray}{cccc}
\phantom{\tiny\young(1,1,1)}  & {\n_{\tiny\young(134,2)}} & {\n_{\tiny\young(124,3)}} & {\n_{\tiny\young(123,4)}} \\
\begin{block}{c(ccc)}
{\v_{\tiny\young(134,2)}} & 1 & \frac{1}{2} & \frac{1}{2} \\
{\v_{\tiny\young(124,3)}} & 0 & \frac{3}{2} & \frac{1}{2} \\
{\v_{\tiny\young(123,4)}} & 0 & 0 & 2 \\
\end{block}
\end{blockarray}
\end{equation*}

\underline{$\lambda = (2,2)$}
\begin{equation*}
\Yboxdim{10pt}
\begin{array}{c}
\begin{tikzpicture}[line width=.5pt,xscale=0.22,yscale=0.22]
\path (20,30) node (T1) {$\young(13,24)$};
\path (20,20) node (T2) {$\young(12,34)$};

\path (T1) edge[black,thick,left] node{$s_2$} (T2);
\end{tikzpicture}
 \end{array}\hskip.25in
  \Yboxdim{7pt}
\begin{blockarray}{ccc}
\phantom{\tiny\young(1,1,1)} & {\n_{\tiny\young(13,24)}} & {\n_{\tiny\young(12,34)}} \\
\begin{block}{c(cc)}
{\v_{\tiny\young(13,24)}} & 1 & \frac{1}{2} \\
{\v_{\tiny\young(12,34)}} & 0 & \frac{3}{2} \\
\end{block}
\end{blockarray}
\end{equation*}

\underline{$\lambda = (2,1,1)$}
\begin{equation*}
\Yboxdim{10pt}
\begin{array}{c}
 \begin{tikzpicture}[line width=.5pt,xscale=0.22,yscale=0.22]
\path (20,30) node (T1) {$\young(14,2,3)$};
\path (20,20) node (T2) {$\young(13,2,4)$};
\path (20,10) node (T3) {$\young(12,3,4)$};

\path (T1) edge[black,thick,left] node{$s_3$} (T2);
\path (T2) edge[black,thick,left] node{$s_2$} (T3);
\end{tikzpicture}
 \end{array}\hskip.25in
  \Yboxdim{7pt}
\begin{blockarray}{cccc}
\phantom{\tiny\young(1,1,1,1)} & {\n_{\tiny\young(14,2,3)}} & {\n_{\tiny\young(13,2,4)}} & \n_{\tiny\young(12,3,4)}\phantom{-} \\
\begin{block}{c(ccc)}
\v_{\tiny\young(14,2,3)} & 1 & \frac{1}{3} & -\frac{1}{3} \\
\v_{\tiny\young(13,2,4)} & 0 & \frac{4}{3} & \frac{2}{3} \\
\v_{\tiny\young(12,3,4)} & 0 & 0 & 2 \\
\end{block}
\end{blockarray}
\end{equation*}

\end{multicols}

\underline{$\lambda = (4,1)$}
\begin{equation*}
\Yboxdim{10pt}
\begin{array}{c}
\begin{tikzpicture}[line width=.5pt,xscale=0.22,yscale=0.22]

\path (20,30) node (T1) {$\young(1345,2)$};
\path (20,20) node (T2) {$\young(1245,3)$};
\path (20,10) node (T3) {$\young(1235,4)$};
\path (20,0) node (T4) {$\young(1234,5)$};

\path (T1) edge[black,thick,left] node{$s_2$} (T2);
\path (T2) edge[black,thick,left] node{$s_3$} (T3);
\path (T3) edge [black,thick,left] node{$s_4$} (T4);
\end{tikzpicture}
 \end{array}\hskip.25in
  \Yboxdim{7pt}
\begin{blockarray}{ccccc}
\phantom{\tiny\young(1,1,1)} & \n_{\tiny\young(1345,2)} & \n_{\tiny\young(1245,3)} & \n_{\tiny\young(1235,4)} & \n_{\tiny\young(1234,5)} \\
\begin{block}{c(cccc)}
\v_{\tiny\young(1345,2)} & 1 & \frac{1}{2} & \frac{1}{2} & \frac{1}{2} \\
\v_{\tiny\young(1245,3)} & 0 & \frac{3}{2} & \frac{1}{2} & \frac{1}{2} \\
\v_{\tiny\young(1235,4)} & 0 & 0 & 2 & \frac{1}{2} \\
\v_{\tiny\young(1234,5)} & 0 & 0 & 0 & \frac{5}{2} \\
\end{block}
\end{blockarray}
\end{equation*}

\underline{$\lambda = (3,2)$}
\begin{equation*}
\Yboxdim{10pt}
\begin{array}{c}
\begin{tikzpicture}[line width=.5pt,xscale=0.22,yscale=0.22]
\path (20,30) node (T1) {$\young(135,24)$};
\path (10,20) node (T2) {$\young(125,34)$};
\path (30,20) node (T3) {$\young(134,25)$};
\path (20,10) node (T4) {$\young(124,35)$};
\path (20,0) node (T5) {$\young(123,45)$};

\path (T1) edge[black,thick,left] node{{\color{black}$s_2$}}  (T2);
\path (T1) edge[black,thick,right] node{{\color{black}$s_4$}} (T3);
\path (T2) edge[black,thick,left] node{{\color{black}$s_4$}} (T4);
\path (T3) edge[black,thick,right] node{{\color{black}$s_2$}} (T4);
\path (T4) edge[black,thick,right] node{{\color{black}$s_3$}}  (T5);
\end{tikzpicture}
 \end{array}\hskip.25in
  \Yboxdim{7pt}
\begin{blockarray}{cccccc}
\phantom{\tiny\young(1,1,1)}  & \n_{\tiny\young(135,24)} & \n_{\tiny\young(125,34)} & \n_{\tiny\young(134,25)} & \n_{\tiny\young(124,35)} & \n_{\tiny\young(123,45)} \\
\begin{block}{c(ccccc)}
\v_{\tiny\young(135,24)} & 1 & \frac{1}{2} & \frac{1}{2} & \frac{1}{4} & -\frac{1}{4} \\
\v_{\tiny\young(125,34)} & 0 & \frac{3}{2} & 0 & \frac{3}{4} & \frac{3}{4} \\
\v_{\tiny\young(134,25)} & 0 & 0 & \frac{3}{2} & \frac{3}{4} & \frac{3}{4} \\
\v_{\tiny\young(124,35)} & 0 & 0 & 0 & \frac{9}{4} & \frac{3}{4} \\
\v_{\tiny\young(123,45)} & 0 & 0 & 0 & 0 & 3 \\
\end{block}
\end{blockarray}
\end{equation*}

\underline{$\lambda = (3,1,1)$}
\begin{equation*}
\Yboxdim{10pt}
\begin{array}{c}
\begin{tikzpicture}[line width=.5pt,xscale=0.22,yscale=0.22]
\path (20,30) node (T1) {$\young(145,2,3)$};
\path (20,20) node (T2) {$\young(135,2,4)$};
\path (10,10) node (T3) {$\young(125,3,4)$};
\path (30,10) node (T4) {$\young(134,2,5)$};
\path (20, 0) node (T5) {$\young(124,3,5)$};
\path (20,-10) node (T6) {$\young(123,4,5)$};

\path (T1) edge[black,thick,left] node{$s_3$} (T2);
\path (T2) edge[black,thick,right] node{$s_2$} (T3);
\path (T2) edge[black,thick,left] node{$s_4$} (T4);
\path (T3) edge[black,thick,right] node{$s_4$} (T5);
\path (T4) edge[black,thick,left] node{$s_2$} (T5);
\path (T5) edge[black,thick,left] node{$s_3$} (T6);
\end{tikzpicture}
 \end{array}\hskip.25in
  \Yboxdim{7pt}
\begin{blockarray}{ccccccc}
\phantom{\tiny\young(1,1,1,1)} & \n_{\tiny\young(145,2,3)} & \n_{\tiny\young(135,2,4)} & \n_{\tiny\young(125,3,4)} & \n_{\tiny\young(134,2,5)} & \n_{\tiny\young(124,3,5)} & \n_{\tiny\young(123,4,5)} \\
\begin{block}{c(cccccc)}
\v_{\tiny\young(145,2,3)} & 1 & \frac{1}{3} & -\frac{1}{3} & \frac{1}{3} & -\frac{1}{3} & 0 \\
\v_{\tiny\young(135,2,4)} & 0 & \frac{4}{3} & \frac{2}{3} & \frac{1}{3} & \frac{1}{6} & -\frac{1}{2} \\
\v_{\tiny\young(125,3,4)} & 0 & 0 & 2 & 0 & \frac{1}{2} & -\frac{1}{2} \\
\v_{\tiny\young(134,2,5)} & 0 & 0 & 0 & \frac{5}{3} & \frac{5}{6} & \frac{5}{6} \\
\v_{\tiny\young(124,3,5)} & 0 & 0 & 0 & 0 & \frac{5}{2} & \frac{5}{6} \\
\v_{\tiny\young(123,4,5)} & 0 & 0 & 0 & 0 & 0 & \frac{10}{3} \\
\end{block}
\end{blockarray}
\end{equation*}

\underline{$\lambda = (2,2,1)$}
\begin{equation*}
\Yboxdim{10pt}
\begin{array}{c}
\begin{tikzpicture}[line width=.5pt,xscale=0.22,yscale=0.22]
\path (20,30) node (T1) {$\young(14,25,3)$};
\path (20,20) node (T2) {$\young(13,25,4)$};
\path (10,10) node (T3) {$\young(12,35,4)$};
\path (30,10) node (T4) {$\young(13,24,5)$};
\path (20,0) node (T5) {$\young(12,34,5)$};

\path (T1) edge[black,thick,left] node{$s_3$} (T2);
\path (T2) edge[black,thick,right] node{$s_2$} (T3);
\path (T2) edge[black,thick,left]  node{$s_4$} (T4);
\path (T3) edge[black,thick,right] node{$s_4$} (T5);
\path (T4) edge[black,thick,left] node{$s_2$} (T5);
\end{tikzpicture}
 \end{array}\hskip.25in
  \Yboxdim{7pt}
\begin{blockarray}{cccccc}
\phantom{\tiny\young(1,1,1,1)} & \n_{\tiny\young(14,25,3)} & \n_{\tiny\young(13,25,4)} & \n_{\tiny\young(12,35,4)} & \n_{\tiny\young(13,24,5)} & \n_{\tiny\young(12,34,5)} \\
\begin{block}{c(ccccc)}
\v_{\tiny\young(14,25,3)} & 1 & \frac{1}{3} & -\frac{1}{3} & -\frac{1}{3} & \frac{1}{3} \\
\v_{\tiny\young(13,25,4)} & 0 & \frac{4}{3} & \frac{2}{3} & \frac{2}{3} & \frac{1}{3} \\
\v_{\tiny\young(12,35,4)} & 0 & 0 & 2 & 0 & 1 \\
\v_{\tiny\young(13,24,5)} & 0 & 0 & 0 & 2 & 1 \\
\v_{\tiny\young(12,34,5)} & 0 & 0 & 0 & 0 & 3 \\
\end{block}
\end{blockarray}
\end{equation*}

\underline{$\lambda = (2,1,1,1)$}
\begin{equation*}
\Yboxdim{10pt}
\begin{array}{c}
\begin{tikzpicture}[line width=.5pt,xscale=0.22,yscale=0.22]
\path (20,30) node (T1) {$\young(15,2,3,4)$};
\path (20,20) node (T2) {$\young(14,2,3,5)$};
\path (20,10) node (T3) {$\young(13,2,4,5)$};
\path (20,0) node (T4) {$\young(12,3,4,5)$};

\path (T1) edge[black,thick,left] node{$s_4$} (T2);
\path (T2) edge[black,thick,left] node{$s_3$} (T3);
\path (T3) edge [black,thick,left] node{$s_2$} (T4);
\end{tikzpicture}
 \end{array}\hskip.25in
  \Yboxdim{7pt}
\begin{blockarray}{ccccc}
\phantom{\tiny\young(1,1,1,1,1)} & \n_{\tiny\young(15,2,3,4)} & \n_{\tiny\young(14,2,3,5)} & \n_{\tiny\young(13,2,4,5)} & \n_{\tiny\young(12,3,4,5)} \\
\begin{block}{c(cccc)}
\v_{\tiny\young(15,2,3,4)} & 1 & \frac{1}{4} & -\frac{1}{4} & \frac{1}{4} \\
\v_{\tiny\young(14,2,3,5)} & 0 & \frac{5}{4} & \frac{5}{12} & -\frac{5}{12} \\
\v_{\tiny\young(13,2,4,5)} & 0 & 0 & \frac{5}{3} & \frac{5}{6} \\
\v_{\tiny\young(12,3,4,5)} & 0 & 0 & 0 & \frac{5}{2} \\
\end{block}
\end{blockarray}
\end{equation*}

\subsection{Transition matrices for Hecke algebras}\label{sec:Hn-matrices}
\label{H24-example}
Here we give change-of-basis matrix between the natural and seminormal representations of the irreducible $\HH_{2,4}$-module $\HH_{2,4}^{((2,1),(1))}$, which continues on the next page.

\setlength{\tabcolsep}{0pt}
\renewcommand{\arraystretch}{2}
\Yboxdim{7pt}
$$
\begin{blockarray}{ccccccc}
~ & \n_{( \tiny\young(13,2), \tiny\young(4) )} & \n_{( \tiny\young(12,3), \tiny\young(4) )} & \n_{( \tiny\young(14,2),\tiny\young(3) )} & \n_{( \tiny\young(12,4),\tiny\young(3) )} & \n_{( \tiny\young(14,3),\tiny\young(2) )}  \\
\begin{block}{c(cccccc)}
 \v_{( \tiny\young(13,2),\tiny\young(4) )} & 1 & \frac{q^3}{q^2+1} & \frac{{u_1}-q^2 {u_1}}{q^3 {u_2}-q {u_1}} & -\frac{q^2 \left(q^2-1\right) {u_1}}{\left(q^2+1\right) \left(q^2 {u_2}-{u_1}\right)} & -\frac{q^2 \left(q^2-1\right) {u_1}}{\left(q^2+1\right) \left(q^2 {u_2}-{u_1}\right)}  \\
 \v_{( \tiny\young(12,3),\tiny\young(4) )} & 0 & \frac{q^4+q^2+1}{q^3+q} & 0 & \frac{\left(q^6-1\right) {u_2}}{\left(q^2+1\right) \left(q^2 {u_2}-{u_1}\right)} & \frac{{u_1}-q^6 {u_1}}{q^2 \left(q^2+1\right) \left(q^2 {u_2}-{u_1}\right)}  \\
 \v_{( \tiny\young(14,2),\tiny\young(3) )} & 0 & 0 & \frac{q ({u_2}-{u_1})}{q^2 {u_2}-{u_1}} & -\frac{q^4 ({u_1}-{u_2})}{\left(q^2+1\right) \left(q^2 {u_2}-{u_1}\right)} & -\frac{q^2 \left(q^2-1\right) {u_2} ({u_1}-{u_2})}{\left({u_1}-q^2 {u_2}\right)^2}  \\
 \v_{( \tiny\young(12,4),\tiny\young(3) )} & 0 & 0 & 0 & \frac{\left(q^4+q^2+1\right) \left(q^4 {u_2}-{u_1}\right)}{q^2 \left(q^2+1\right) \left(q^2 {u_2}-{u_1}\right)} & 0   \\
 \v_{( \tiny\young(14,3),\tiny\young(2) )} & 0 & 0 & 0 & 0 & \frac{({u_1}-{u_2}) \left({u_1}-q^4 {u_2}\right)}{\left({u_1}-q^2 {u_2}\right)^2}  \\
 \v_{( \tiny\young(13,4),\tiny\young(2) )} & 0 & 0 & 0 & 0 & 0  \\
 \v_{( \tiny\young(23,4),\tiny\young(1) )} & 0 & 0 & 0 & 0 & 0  \\
 \end{block}
\end{blockarray}\quad \cdots \qquad
$$

\vfill\eject

$$
\qquad\cdots \quad
\begin{blockarray}{ccccc}
~ &  \n_{( \tiny\young(13,4),\tiny\young(2) )} & \n_{( \tiny\young(24,3),\tiny\young(1) )} & \n_{( \tiny\young(23,4),\tiny\young(1) )} \\
\begin{block}{c(cccc)}
 \v_{( \tiny\young(13,2),\tiny\young(4) )} &  \frac{q \left(q^2-1\right)^2 {u_1}^2}{\left(q^2+1\right) \left({u_1}-q^2 {u_2}\right)^2} & \frac{q \left(q^2-1\right) {u_1}}{\left(q^2+1\right) \left(q^2 {u_2}-{u_1}\right)} & -\frac{\left(q^2-1\right)^2 {u_1}^2}{\left(q^2+1\right) \left({u_1}-q^2 {u_2}\right)^2} \\
 \v_{( \tiny\young(12,3),\tiny\young(4) )} &  -\frac{\left(q^2-1\right)^2 \left(q^4+q^2+1\right) {u_1} {u_2}}{\left(q^3+q\right) \left({u_1}-q^2 {u_2}\right)^2} & \frac{{u_1}-q^6 {u_1}}{q \left(q^2+1\right) \left(q^2 {u_2}-{u_1}\right)} & -\frac{\left(q^2-1\right)^2 \left(q^4+q^2+1\right) {u_1} {u_2}}{\left(q^2+1\right) \left({u_1}-q^2 {u_2}\right)^2} \\
 \v_{( \tiny\young(14,2),\tiny\young(3) )} &  -\frac{\left(q^2-1\right)^2 ({u_1}-{u_2}) \left(q^4 {u_2}^2+{u_1}^2\right)}{q \left(q^2 {u_2}-{u_1}\right)^3} & -\frac{q^5 \left(q^2-1\right) {u_2} ({u_1}-{u_2})}{\left(q^2+1\right) \left({u_1}-q^2 {u_2}\right)^2} & -\frac{\left(q^2-1\right) ({u_1}-{u_2}) \left(q^2 {u_1}+{u_2}\right)}{\left(q^2+1\right) \left({u_1}-q^2 {u_2}\right)^2} \\
 \v_{( \tiny\young(12,4),\tiny\young(3) )} &  -\frac{\left(q^2-1\right) \left(q^4+q^2+1\right) {u_1} \left(q^4 {u_2}-{u_1}\right)}{q^3 \left(q^2+1\right) \left({u_1}-q^2 {u_2}\right)^2} & 0 & -\frac{\left(q^2-1\right) \left(q^4+q^2+1\right) {u_1} \left(q^4 {u_2}-{u_1}\right)}{\left(q^2+1\right) \left(q {u_1}-q^3 {u_2}\right)^2} \\
 \v_{( \tiny\young(14,3),\tiny\young(2) )} &  -\frac{q^3 ({u_1}-{u_2}) \left(q^4 {u_2}-{u_1}\right)}{\left(q^2+1\right) \left({u_1}-q^2 {u_2}\right)^2} & \frac{\left(q^2-1\right) {u_2} \left(q^4 {u_2}-{u_1}\right)}{q \left({u_1}-q^2 {u_2}\right)^2} & \frac{q^2 \left(q^2-1\right) {u_2} \left(q^4 {u_2}-{u_1}\right)}{\left(q^2+1\right) \left({u_1}-q^2 {u_2}\right)^2} \\
 \v_{( \tiny\young(13,4),\tiny\young(2) )} &  -\frac{\left(q^4+q^2+1\right) ({u_1}-{u_2}) \left(q^4 {u_2}-{u_1}\right)}{\left(q^3+q\right) \left({u_1}-q^2 {u_2}\right)^2} & 0 & \frac{\left(q^2-1\right) {u_2} ({u_1}-{u_2}) \left(q^4 {u_2}-{u_1}\right)}{\left({u_1}-q^2 {u_2}\right)^3} \\
 \v_{( \tiny\young(24,3),\tiny\young(1) )} &  0 & \frac{{u_1}-q^4 {u_2}}{q {u_1}-q^3 {u_2}} & \frac{q^2 \left(q^4 {u_2}-{u_1}\right)}{\left(q^2+1\right) \left(q^2 {u_2}-{u_1}\right)} \\
 \v_{( \tiny\young(23,4),\tiny\young(1) )} &  0 & 0 & \frac{\left(q^4+q^2+1\right) \left(q^4 {u_2}-{u_1}\right)}{q^2 \left(q^2+1\right) \left(q^2 {u_2}-{u_1}\right)} \\
 \end{block}
\end{blockarray}
$$
By sending $u_1 \to 1, u_2 \to -1, q \to 1$ in the matrix above we obtain the change-of-basis matrix between the natural and seminormal representations of the irreducible $\GG_{2,4}$-module $\GG_{2,4}^{((2,1),(1))}$:
$$
\setlength{\tabcolsep}{0pt}
\renewcommand{\arraystretch}{1.5}
\Yboxdim{7pt}
\begin{blockarray}{ccccccccc}
~ & \n_{( \tiny\young(13,2),\tiny\young(4) )} & \n_{( \tiny\young(12,3), \tiny\young(4) )} & \n_{( \tiny\young(14,2),\tiny\young(3) )} & \n_{( \tiny\young(12,4),\tiny\young(3) )} & \n_{( \tiny\young(14,3),\tiny\young(2) )} & \n_{( \tiny\young(13,4),\tiny\young(2) )} & \n_{( \tiny\young(24,3),\tiny\young(1) )} & \n_{( \tiny\young(23,4),\tiny\young(1) )} \\
\begin{block}{c(cccccccc)}
 \v_{( \tiny\young(13,2),\tiny\young(4) )} & 1 & \frac{1}{2} & 0 & 0 & 0 & 0 & 0 & 0 \\
\v_{( \tiny\young(12,3), \tiny\young(4) )} & 0 & \frac{3}{2} & 0 & 0 & 0 & 0 & 0 & 0 \\
\v_{( \tiny\young(14,2),\tiny\young(3) )} & 0 & 0 & 1 & \frac{1}{2} & 0 & 0 & 0 & 0 \\
\v_{( \tiny\young(12,4),\tiny\young(3) )} & 0 & 0 & 0 & \frac{3}{2} & 0 & 0 & 0 & 0 \\
\v_{( \tiny\young(14,3),\tiny\young(2) )} & 0 & 0 & 0 & 0 & 1 & \frac{1}{2} & 0 & 0 \\
\v_{( \tiny\young(13,4),\tiny\young(2) )} & 0 & 0 & 0 & 0 & 0 & \frac{3}{2} & 0 & 0 \\
\v_{( \tiny\young(24,3),\tiny\young(1) )} & 0 & 0 & 0 & 0 & 0 & 0 & 1 & \frac{1}{2} \\
\v_{( \tiny\young(23,4),\tiny\young(1) )} & 0 & 0 & 0 & 0 & 0 & 0 & 0 & \frac{3}{2} \\
 \end{block}
\end{blockarray}\ ,
$$
which is the direct sum over four copies (one for each possible partitioned alphabet) of the tensor product $\calA_{(2,1)} \otimes \calA_{(1)} = \calA_{(2,1)}$, as given in Section \ref{sec:Sn-matrices}.

\newpage

\bibliographystyle{math}

\bibliography{TransitionMatrices}

\end{document}